\documentclass[12pt]{amsart}
\usepackage{tikz}

\usepackage{amsmath,amsfonts,amsthm,amssymb,amscd, verbatim, graphicx,textcomp, hyperref}
\makeatletter
\newcommand{\xRightarrow}[2][]{\ext@arrow 0359\Rightarrowfill@{#1}{#2}}
\makeatother
\usepackage[notcite, notref]{}
\usepackage{upgreek}
\usepackage[margin=2.5cm]{geometry}
\usepackage{pinlabel}
\usepackage{pdflscape}
\usepackage{tabularx}
\usepackage{latexsym}
\usepackage{hyperref}
\usepackage{multirow}
\usepackage{xypic}

\usepackage{color}
\usepackage{stmaryrd}

\usepackage{makecell}

\newtheorem{Thm}{Theorem}[section]
\newtheorem{Cor}[Thm]{Corollary}
\newtheorem{Prop}[Thm]{Proposition}

\newtheorem{Lem}[Thm]{Lemma}
\newtheorem{Conj}[Thm]{Conjecture}

\theoremstyle{definition}
\newtheorem{Def}[Thm]{Definition}

\newtheorem{Rem}[Thm]{Remark}

\theoremstyle{remark}

\newtheorem{remark}[Thm]{Remark}

\numberwithin{equation}{section}

\newcommand\myeq{\mathrel{\stackrel{\makebox[0pt]{\mbox{\normalfont\tiny def}}}{=}}}

\newcommand{\Aut}{\operatorname{Aut}}
\newcommand\Br{\mathrm{Br}}

\newcommand{\Hom}{\operatorname{Hom}}

\newcommand{\Id}{\operatorname{Id}}

\newcommand{\Iso}{\operatorname{Iso}}
\newcommand{\Out}{\operatorname{Out}}

\renewcommand{\dim}{\operatorname{dim}}
\newcommand{\Inn}{\operatorname{Inn}}

\newcommand{\Res}{\operatorname{Res}}

\newcommand{\cl}{\operatorname{cl}}

\def \a {\alpha}\def \b {\beta}

\newcommand{\D}{\mathcal{D}}

\newcommand{\m}{\textbf{m}}
\renewcommand{\c}{\textbf{c}}

\renewcommand{\t}{\textbf{t}}

\renewcommand{\a}{\textbf{a}}
\newcommand{\w}{\textbf{w}}
\renewcommand{\b}{\textbf{b}}
\newcommand{\z}{\textbf{z}}
\newcommand{\f}{\textbf{f}}
\newcommand{\x}{\textbf{x}}
\newcommand{\res}{\operatorname{res}}

\newcommand{\IBr}{\operatorname{IBr}}
\newcommand{\Irr}{\operatorname{Irr}}

\newcommand{\M}{\mathcal{M}}
\newcommand{\MM}{\operatorname{M}}

\newcommand{\J}{\operatorname{J}}
\newcommand{\PSL}{\operatorname{PSL}}

\newcommand{\Th}{\operatorname{Th}}
\newcommand{\He}{\operatorname{He}}
\newcommand{\Fi}{\operatorname{Fi}}
\newcommand{\ON}{\operatorname{O'N}}
\newcommand{\GL}{\operatorname{GL}}
\newcommand{\SL}{\operatorname{SL}}

\newcommand{\Ind}{\operatorname{Ind}}
\renewcommand{\epsilon}{\varepsilon}
\renewcommand{\bar}{\overline}
\renewcommand{\hat}{\widehat}
\renewcommand{\leq}{\leqslant}
\renewcommand{\geq}{\geqslant}

\renewcommand{\phi}{\varphi}
\newcommand{\opdef}{\operatorname{def}}
\newcommand{\ab}{\operatorname{ab}}
\newcommand{\y}{\mathbf{y}}
\newcommand{\xra}{\xrightarrow}
\newcommand{\tra}{\operatorname{tra}}
\newcommand{\id}{\operatorname{id}}

\newcommand{\RV}{\operatorname{RV}}

\newcommand{\N}{\mathcal{N}}
\newcommand{\F}{\mathcal{F}}

\newcommand{\G}{\mathcal{G}}

\newcommand{\C}{\mathcal{C}}

\renewcommand{\P}{\mathcal{P}}
\newcommand{\A}{\mathcal{A}}
\newcommand{\E}{\mathcal{E}}
\newcommand{\W}{\mathcal{W}}

\renewcommand{\C}{\mathcal{C}}

\renewcommand{\O}{\mathcal{O}}

\renewcommand{\k}{\textbf{k}}

\newcommand{\gen}[1]{\langle #1 \rangle}

\begin{document}

\title{Weight conjectures for fusion systems}

\author{Radha Kessar}
\address{Department of Mathematics, City, University of London EC1V 0HB, 
United Kingdom}
\email{radha.kessar.1@city.ac.uk}

\author{Markus Linckelmann}
\address{Department of Mathematics, City, University of London EC1V 0HB, 
United Kingdom}
\email{markus.linckelmann.1@city.ac.uk}

\author{Justin Lynd}
\address{Department of Mathematics, University of Louisiana at 
Lafayette, Lafayette, LA 70504}
\email{lynd@louisiana.edu}

\author{Jason Semeraro}
\address{Heilbronn Institute for Mathematical Research, Department of 
Mathematics, University of Leicester,  United Kingdom}
\email{jpgs1@leicester.ac.uk}

\begin{abstract} 
Many of the conjectures of current interest in the representation 
theory of finite groups in characteristic $p$ are local-to-global 
statements, in that they predict consequences  for the representations 
of a finite group $G$ given data about the representations of the 
$p$-local subgroups of $G$. The local structure of a block of a group 
algebra is encoded in  the fusion system of the block together with a 
compatible family of K\"ulshammer-Puig cohomology classes.  
Motivated by conjectures in block theory, we state and initiate 
investigation of a number of seemingly local conjectures for arbitrary 
triples $(S,\F,\alpha)$ consisting of a saturated fusion system $\F$  on a 
finite $p$-group $S$ and a compatible family $\alpha$.
\end{abstract}


\subjclass[2010]{}

\maketitle

\section{Introduction}
\label{introSection}

Throughout this paper we fix a prime number $p$ and an algebraically
closed field $k$ of characteristic $p$. A block $B$ of a finite
group algebra $kG$ gives rise to three fundamental invariants encoding the
local structure of $B$: a defect group $S$, a saturated fusion system $\F$ on
$S$, and a family $\alpha=$ $(\alpha_Q)_{Q\in\F^c}$ of second cohomology
classes $\alpha_Q\in$ $H^2(\Out_\F(Q), k^\times)$. The $\alpha_Q$ are called
the K\"ulshammer-Puig classes of the block $B$. They are defined for each
$\F$-centric subgroup $Q$ of $S$, and satisfy a certain compatibility condition
(recalled in Section \ref{s:KP}). The triple $(S,\F,\alpha)$ is determined 
by $B$ uniquely up to $G$-conjugacy.  If $B$ is the principal block of $kG$, 
then $S$ is a Sylow $p$-subgroup, $\F=$ $\F_S(G)$, and all the classes 
$\alpha_Q$ are trivial.   
In what follows, we freely use standard notation on fusion systems as in 
\cite{AschbacherKessarOliver2011}. For a finite dimensional $k$-algebra 
$B$, we denote by $\ell(B)$ the number of isomorphism classes  of 
simple $B$-modules and by $z(B)$  the number of isomorphism classes  
of simple and projective $B$-modules. If $B$ is a block of a finite group 
algebra $kG$, then we denote by $\k(B)$  the number of ordinary irreducible 
characters  of $G$ associated with $B$.

The prominent counting conjectures in the block theory of finite groups 
express numerical invariants of $B$ in terms of $(S,\F,\alpha)$. 
Alperin's weight conjecture (henceforth abbreviated AWC) 
predicts the equality
$$\ell(B) = \sum_{Q\in \F^c/\F}\ z(k_\alpha\Out_\F(Q))\ ,$$
where $\F^c/\F$ is a set of representatives of the isomorphism classes 
in $\F$ of $\F$-centric subgroups of $S$, and where $k_\alpha\Out_\F(Q)$ 
is the group algebra of $\Out_\F(Q)=$ $\Aut_\F(Q)/\Inn(Q)$ twisted by 
$\alpha_Q$.
The right side in this version of AWC  clearly makes  sense for arbitrary
saturated fusion systems  and arbitrary choices of second cohomology,
classes, and this is the starting point of the present paper.

\medskip
Let  $(S,\F,\alpha)$ be a triple consisting of a finite $p$-group $S$, 
a saturated fusion system $\F$ on $S$, and a family $\alpha=$ 
$(\alpha_Q)_{Q\in\F^c}$ of classes $\alpha_Q\in$ 
$H^2(\Out_\F(Q);k^\times)$, for any $\F$-centric subgroup $Q$ of $S$, 
such that the family $\alpha$ is $\F$-compatible in the sense of 
Definition \ref{def:Fcompatible} below. If $\alpha$ is the family of 
K\"ulshammer-Puig classes of a fusion system $\F$ of a block $B$ with 
defect group $S$, then $\alpha$ is $\F$-compatible by Theorem  
\cite[8.14.5]{LinckelmannBT}; in that case we will say that the triple 
$(S,\F,\alpha)$ is {\it block realizable} and that it is 
{\it realized by the block} $B$. 

For any  $\F$-centric subgroup $Q$ of $S$ and any subgroup $H$ of 
$\Out_{\F}(Q)$ or of $\Aut_\F(Q)$, by $k_{\alpha}H$ we  will mean  
the twisted group algebra of $H$ over $k$  with respect to the 
restriction of $\alpha_Q$ to $H$.
Using the notation in \cite[Section 8.15]{LinckelmannBT}, the
\textit{number of weights} of $(S,\F,\alpha)$ is the positive integer
$\w(\F,\alpha)$ defined by
$$\w(\F,\alpha):=\sum_{Q \in \F^{c}/\F} z(k_{\alpha} \Out_\F(Q))\ ,$$
where the notation $Q\in \F^c/\F$ means that $Q$ runs over a set of
representatives of the isomorphism classes in $\F$ of $\F$-centric
subgroups of $S$. 
Note that $z(k_{\alpha} \Out_\F(Q)) = 0$ unless $Q$ is also 
$\F$-radical (cf. Lemma \ref{l:z=0} below), and hence we have 
$\w(\F,\alpha)=\displaystyle\sum_{Q \in \F^{cr}/\F} z(k_{\alpha} \Out_\F(Q))$. 
By Proposition \ref{p:links} or the above remarks, if $(S,\F,\alpha)$ is 
realized by a block $B$ of a finite group algebra, then $B$ satisfies
AWC if and only if $\w(\F,\alpha)=$ $\ell(B)$. 

If $x$ is an element in $S$ such that $\langle x\rangle$ is fully
$\F$-centralized, then $C_\F(x)$ is a saturated fusion system on
$C_S(x)$, there is a canonical functor $C_\F(x)^c\to$ $\F^c$,
and restriction along this functor sends the $\F$-compatible
family $\alpha$ to a $C_\F(x)$-compatible family $\alpha(x)$; 
see Proposition  \ref{p:links} below. Denote by $[S/\F]$ a set of 
$\F$-conjugacy class representatives of elements of $S$ such that 
$\langle x \rangle$ fully $\F$-centralized. We set
$$ \k(\F,\alpha) := \sum_{x \in [S/\F]} \w(C_\F(x),\alpha(x))\ .$$
By Proposition \ref{p:links}, if $(S,\F,\alpha)$ is realized by a
block $B$ of a finite group algebra such that $B$ and the
$B$-Brauer pairs satisfy AWC, then $\k(\F,\alpha)=$ $\k(B)$.

For any $\F$-centric subgroup $Q$ of $S$ we define the set $\N_Q$
to be the set of non-empty normal chains $\sigma$ of $p$-subgroups of 
$\Out_\F(Q)$ starting  at the trivial subgroup; that is, chains of 
the form
$$\sigma = (1=X_0 < X_1 <\cdots < X_m)$$
with the property that $X_i$ is normal in $X_m$ for $0\leq i\leq m$. 
We set $|\sigma| = m$, and call $m$ the {\it length} of $\sigma$. 
We define the following two sets:
$$\W_Q = \N_Q \times \Irr(Q)\ ,$$
$$\W_Q^* = \N_Q \times Q^{\cl}\ ,$$
where $\Irr(Q)$ is the set of  ordinary  irreducible characters of $Q$
and where $Q^{\cl}$ is the set of conjugacy classes of $Q$.
There are obvious actions of the group $\Out_\F(Q)$ on the sets
$\N_Q$, $\Irr(Q)$, and $Q^{\cl}$, hence on the sets $\W_Q$, $\W^*_Q$.
We denote by $I(\sigma,\mu)$ and   
by $I(\sigma, [x])$ the stabilisers in $\Out_\F(Q)$ 
under these actions, where $(\sigma,\mu)\in$ $\W_Q$ and 
$(\sigma, [x])\in$ $\W_Q^*$, with $[x]$ the conjugacy class in $Q$ of
an element $x\in$ $Q$. For any $\F$-centric subgroup $Q$ of $S$ we set
$$\w_Q(\F,\alpha) = \sum_{(\sigma,\mu)\in\W_Q/\Out_\F(Q)}\ \    
(-1)^{|\sigma|} z(k_\alpha I(\sigma,\mu))\ $$
$$\w_Q^*(\F,\alpha) = \sum_{(\sigma, [x])\in\W^*_Q/\Out_\F(Q)}\ \    
(-1)^{|\sigma|} z(k_\alpha I(\sigma, [x]))\ ,$$
and we set
$$\m(\F,\alpha) = \sum_{Q\in\F^c/\F}\ \w_Q(\F,\alpha)\ ,$$
$$\m^*(\F,\alpha) = \sum_{Q\in\F^c/\F}\ \w_Q^*(\F,\alpha)\ .$$
There are refinements of the above numbers which 
take into account  defects of ordinary irreducible characters 
and which appear in conjectures of Dade and Robinson.  These will 
be considered in Section \ref{conjSection}.

\begin{Thm}\label{t:main2}
Let $\F$ be a saturated fusion system on a finite $p$-group
$S$ and let $\alpha$ be an $\F$-compatible family. Then
$$\m^*(\F,\alpha)=\k(\F,\alpha) \,.$$
\end{Thm}

Theorem \ref{t:main2} is a cancellation theorem for arbitrary fusion 
systems inspired by cancellation theorems  of Robinson such as in 
\cite[Theorem 1.2]{robinson1996local}. 

\begin{Thm}\label{t:main}
Let $\F$ be a saturated fusion system on a finite $p$-group
$S$ and let $\alpha$ be an $\F$-compatible family. 
If AWC holds, then 
$\m(\F,\alpha)=\m^*(\F,\alpha)$. 
\end{Thm}

Theorem \ref{t:main}    
shows that AWC implies an equality (for arbitrary fusion systems)  of 
two numerical invariants dual to  each other in the sense that  one is  
obtained by summing over  conjugacy classes  of $p$-groups and the other  
by  summing over  irreducible characters.  
Given that the numerical invariants $\m$, $\m^*$, $\k$ are entirely
defined at the `local' level of fusion systems and compatible
families, it seems surprising that Alperin's Weight Conjecture is 
needed to obtain the conclusion of Theorem \ref{t:main}.

\begin{Cor} \label{co:main}
Let $\F$ be a saturated fusion system on a finite $p$-group $S$ and let 
$\alpha$ be an $\F$-compatible family. If AWC  holds, then 
$\m(\F,\alpha)=\k(\F,\alpha)$. 
\end{Cor}

If $(S, \F, \alpha)$ is block realizable, then Corollary \ref{co:main}    
follows from work of Robinson and expresses the fact that a  
coarse version of the Ordinary Weight Conjecture is implied by AWC
(see Theorem \ref{t:connect} below).  

\medskip
The paper is organised as follows. Section \ref{conjSection} 
contains a list of conjectures inspired by their block theoretic counter 
parts. In Section \ref{prelimSection} we collect background material, 
Section \ref{s:KP} contains relevant properties of   $\F$-compatible families, Section \ref{towardsmainSection} contains technicalities needed 
for the proofs of Theorems \ref{t:main2} and \ref{t:main} in Section 
\ref{main2Section} and Section \ref{mainSection}, respectively.
Section \ref{exampleSection} verifies some of the conjectures in 
Section \ref{conjSection} for certain exotic fusion systems. 
In an Appendix, we collect some foundational
material from work of Robinson.

\subsection*{Acknowledgements} 
Many of the key ideas in this paper were worked out during the 
workshop ``Group Representation Theory and Applications'' at the 
Mathematical Sciences Research Institute (MSRI) in   February   2018.  
The authors would like to thank MSRI for its hospitality, and for 
providing such a pleasant environment in which to carry out research.  
The first and second authors  were MSRI   members  during the Spring 2018
semester  which was  supported by  the National Science Foundation under Grant No. DMS-1440140. 

The second author acknowledges support from EPSRC grant 
EP/M02525X/1. The fourth author gratefully acknowledges financial 
support from the Heilbronn Institute.

\section{Conjectures} 
\label{conjSection} 

We formulate conjectures for fusion systems which are motivated by 
conjectural or known statements in block theory. For each  of these 
conjectures, the link  with a block theoretic conjecture is made either 
via AWC or via  the Ordinary Weight Conjecture,  the  statement  of  
which  will be recalled below.    Note that by work of 
Robinson the Ordinary Weight Conjecture implies the AWC.

These conjectures make precise  the   
idea  that the gap between various local-global block 
theoretic conjectures is purely local. 
Proving or disproving any of these is a win-win scenario.  If one
can prove one of these conjectures at the fusion system level, then
one would get that AWC  (or the ordinary weight conjecture)  implies 
the corresponding block theoretic version. If on the other hand one 
could disprove any of these, one
would either have found a counter example to the corresponding 
block theoretic conjecture, or one would have found a way to
distinguish exotic fusion systems from block realizable fusion
systems. Either outcome would be interesting.

We keep  the notation of the previous section. Let $\F$ be a saturated
fusion system on a finite $p$-group $S$ and let $\alpha$ be an
$\F$-compatible family (see Definition \ref{def:Fcompatible}).  
Recall  from Proposition \ref{p:links} that if $(S, \F, \alpha)$ is 
realized by a block $B$ which satisfies  AWC, then 
$\w(\F, \alpha) =\ell (B)$, and if all Brauer correspondents of $B$ 
also satisfy  AWC, then  $\k(\F, \alpha) =\k(B)$.  

\begin{Conj} \label{conj:k(b)}   
Let $\F$ be a saturated fusion system on a finite $p$-group $S$ and let 
$\alpha$ be an $\F$-compatible family. Then 
$\k(\F, \alpha) \leq |S|$.
\end{Conj}

By the above remark, if $(S,\F,\alpha)$ is realizable by a block $B$ 
such that AWC holds for all $B$-Brauer pairs, then Conjecture
\ref{conj:k(b)}  holds if and only if $B$ satisfies Brauer's  
$\k(B)$-conjecture, which predicts the inequality $\k(B)\leq$ $|S|$.  Also,  note that      by   Theorem \ref{t:main2},   the   inequality  of     Conjecture  \ref{conj:k(b)}   is equivalent to the inequality  $\m(\F, \alpha) \leq |S|$.  In  view of  Theorem \ref{t:main} and Corollary \ref{co:main} (see  also Conjecture  \ref{c:AWOW}),   one could consider   versions of the   inequality       with $\k(\F, \alpha)  $ replaced by $\m (\F, \alpha) $. 

\begin{Conj} \label{c: malle-robinson}     
Let $\F$ be a saturated fusion system on a finite $p$-group
$S$ and let $\alpha$ be an $\F$-compatible family. Then
$\w(\F, \alpha) \leq p^s$, where $s$ is the sectional rank of $S$.
\end{Conj}

If $(S,\F,\alpha)$ is realizable by a block $B$ such that AWC holds
for $B$, then the above is  equivalent  to the statement that $B$ 
satsifies the  conjecture by Malle and Robinson
\cite{MalleRobinson2017} predicting the inequality $\ell(B)\leq$ $p^s$. 
Conjecture  \ref{c: malle-robinson}  has been shown to hold for the exotic Solomon fusion systems by 
Lynd and Semeraro \cite{lynd2017weights}. 

\medskip
Next, we refine the integers $\w(\F,\alpha)$, $\m(\F,\alpha)$, 
$\m^*(\F,\alpha)$ by taking into account defects of characters. For $Q$ 
a subgroup of $S$  and $d$ a non-negative integer, we set
$$\Irr_K^d(Q):=\{\mu \in \Irr_K^d(Q) \mid v_p(|Q|/\mu(1))=d\}\ ;$$ 
this is the set of ordinary irreducible characters of $Q$ of defect $d$.
Note that this set is $\Out_\F(Q)$-stable. 
As in the previous section, we denote by $\N_Q$ the set of nonempty
normal chains of $p$-subgroups of $\Out_\F(Q)$ starting with the
trivial subgroup of $\Out_\F(Q)$. Given such a chain $\sigma$ and
an irreducible character $\mu$ of $Q$, we denote by $I(\sigma)$ and
$I(\sigma,\mu)$ the stabilisers of $\sigma$ and of the pair $(\sigma,\mu)$
in $\Out_\F(Q)$.  

Given a saturated fusion system $\F$ on a finite $p$-group $S$, an
$\F$-compatible family $\alpha$, and a non-negative integer $d$, 
following \cite[Part IV, Section 5.7]{AschbacherKessarOliver2011}, we 
set 
$$\w_Q(\F,\alpha,d):=
\sum_{\sigma \in \N_Q/\Out_\F(Q)} (-1)^{|\sigma|} 
\sum_{\mu \in \Irr_K^d(Q)/I(\sigma)} z(k_{\alpha} I(\sigma,\mu)).
$$
and
$$\m(\F,\alpha,d):=\sum_{Q \in \F^{c}/\F} \w_Q(\F,\alpha,d).
$$
We clearly have
$$\m(\F,\alpha)=\sum_{d \geq 0} \m(\F,\alpha,d) 
$$
 
The Ordinary Weight Conjecture (henceforth abbreviated OWC), first stated
in \cite{robinson1996local} and reformulated in 
\cite{Robinsonweight2004}, states that 
if $B$ is  a block of the group algebra $kG$ of a finite group $G$ with    
defect group $S$, fusion system $\F$ and family of K\"ulshammer--Puig 
classes $\alpha$, then for each $d\geq 0 $, $\m(\F, \alpha, d)$ equals the 
number of ordinary irreducible characters of defect $d$ associated to the 
block $B$ (cf. \cite[IV.5.49]{AschbacherKessarOliver2011}).      
As  noted above, $\m(\F,\alpha)=\sum_{d \geq 0} \m(\F,\alpha,d) $. Thus,     
OWC implies  the  following  ``summed up  version" (henceforth    
abbreviated  SOWC): if $B$ is a block of the group algebra $kG$ of a 
finite group $G$ with defect group $S$, fusion system $\F$ and 
family of K\"ulshammer--Puig classes $\alpha$, then 
$\m(\F, \alpha)  = \k(B)$, the number of ordinary irreducible characters 
of $G$ associated with $B$. On the other hand, AWC predicts that 
$\k(\F, \alpha)$ equals $\k(B)$. This leads to the following conjecture.

\begin{Conj} \label{c:AWOW}  
Let $\F$ be a saturated fusion system on a finite $p$-group
$S$ and let $\alpha$ be an $\F$-compatible family. We have
$$\k(\F,\alpha) = \m(\F, \alpha) . $$ 
\end{Conj}  

Now  Corollary \ref{co:main}  may be restated as follows.

\begin{Thm} \label{t:connect}    
Suppose that  AWC  holds for  all blocks. Then  Conjecture  \ref{c:AWOW}  
holds for all $(S, \F, \alpha) $,  $S$  a finite $p$-group, $\F$  a  
saturated fusion system on $S$ and  $\alpha$ an 
$\F$-compatible family.
\end{Thm}  

By \cite{robinson1996local},  \cite{Robinsonweight2004},  
AWC  is equivalent  to  SOWC   in the sense 
that a minimal counter-example to AWC is a a minimal counter-example 
to the other. The  difficult  implication  is that AWC implies    
SOWC.   Theorem \ref{t:connect}  may  be viewed as an extension  of  
Robinson's result to arbitrary fusion systems.
  
\begin{Conj}\label{c:defect}
Let $\F$ be a saturated fusion system on a finite $p$-group
$S$ and let $\alpha$ be an $\F$-compatible family. 
For each positive integer $d$, we have $\m(\F,\alpha,d) \ge 0.$
\end{Conj}

\begin{Rem} 
With the above notation, suppose that $d$ is the integer such that
$|S|= p^d $. The only chain  contributing to the expression for 
$\m(\F, \alpha, d)$ is the chain $S$ of length zero and the 
contribution of  this chain  is a strictly positive integer. This is 
because $\Out_{\F}(S) $ is a $p'$-group. 
\end{Rem}

We consider next Brauer's height zero conjecture.

\begin{Prop} \label{p:easyheightzero} 
Let $\F$ be a saturated fusion system on a finite $p$-group $S$ and let 
$\alpha$ be an $\F$-compatible family. Suppose that $S$ is abelian 
of order $p^d$. Then $\m(\F, \alpha, d')= 0$ for all $d' \ne d$. 
\end{Prop}

\begin{proof}  
Since $S$ is abelian, $S$ is the only $\F$-centric subgroup of $S$, and 
all characters of $S$ are linear, hence of defect $d$. The result 
follows.
\end{proof}

\begin{Conj} \label{c:heightzero} 
Let $\F$ be a saturated fusion system on a finite $p$-group
$S$ and let $\alpha$ be an $\F$-compatible family.
Suppose that $S$ is nonabelian of order $p^d$. Then   
$\m(\F, \alpha, d')  \ne  0$  for some  $d' \ne d$. 
\end{Conj} 

If  $S$  is non-abelian  and $(S,\F,\alpha)$ is realized by a block $B$  
satisfying  OWC, then the above  is equivalent  to the statement that 
$B$  satisfies  Brauer's height zero 
conjecture.  Note that Navarro and Tiep \cite{NavarroTiep2013}   have 
proved that the  height  zero conjecture is a consequence of the Dade  
projective conjecture  and  of the fact that  the  Brauer height zero 
conjecture has been checked  for finite  quasi-simple groups 
\cite{KessarMalle2017}.  Eaton has proved in \cite{Eaton2004} that the  
Dade projective conjecture is equivalent to the OWC  in the sense  
that a minimal  counter-example to one is a minimal counter-example to 
the other. Thus  the above conjecture   for   block 
realizable  triples  is a consequence of  OWC.    

\begin{Conj} \label{c:eatonmoreto}  
Let $\F$ be a saturated fusion system on a finite $p$-group
$S$ and let $\alpha$ be an $\F$-compatible family.
Suppose that $S$ is nonabelian of order $p^d$. 
Let $r >0 $ be the smallest  positive integer such that  
$S$ has  a character of degree $p^r$. Then  $r$  is  the  smallest  
positive integer such that  $\m(\F, \alpha, d-r) \ne 0 $.
\end{Conj} 
 
If  $(S,\F,\alpha)$ is realized by a block $B$  satisfying  OWC, then 
the above is equivalent  to the statement that 
$B$ satisfies the  conjecture  by Eaton  and Moreto  in  
\cite{EatonMoreto2014}.

\begin{Conj} \label{c:malle-navarro} 
Let $\F$ be a saturated fusion system on a finite $p$-group $S$ of order 
$p^d$ and let $\alpha$ be an $\F$-compatible family.  Then

\begin{enumerate} 

\item   
$\k(\F, \alpha)/\m(\F, \alpha, d)$ is at most the number of conjugacy 
classes of $[S,S]$.

\item  
$\k(\F, \alpha)/\w(\F, \alpha)$ is at most the number of 
conjugacy classes of $S$.
\end{enumerate}
\end{Conj}

If $(S,\F,\alpha)$ is realized by a block $B$  satisfying  OWC, then 
the above  is equivalent  to the statement that $B$  satisfies the  
conjecture of  Malle and Navarro in \cite{MalleNavarro2006}.   Similar  to  
Conjecture \ref{conj:k(b)}, one could consider   versions of  the above 
inequalities with $\k(\F, \alpha)$ replaced by $\m(\F, \alpha) $  or   
$ \m^*(\F, \alpha ) $.

\medskip 

If $\F$ is  $p$-solvable (i.e.  if $\F$  is constrained  with  
$p$-solvable model) then for any $\F$-compatible  family  $\alpha $, 
the triple $(S, \F, \alpha)$ is realizable by a block of a $p$-solvable 
group (see  Proposition \ref{p:fareal}). The OWC   has been shown to 
hold  for blocks of $p$-solvable  groups by Robinson,  and  AWC   for 
$p$-solvable groups was   proved  earlier  by Okuyama.   The $k(B)$  
conjecture    for finite  $p$-solvable groups  was proved  in  
\cite{GlMaRiSc2004}  and   the height zero  conjecture  for $p$-solvable 
groups was shown to hold  by  Gluck and Wolf   \cite{GluckWolf84}. Thus   
Conjectures \ref{conj:k(b)},  \ref{c:AWOW},\ref{c:defect}, \ref{c:heightzero} 
all hold for solvable fusion systems. If   moreover  $\F = N_{\F}(S) $, 
then for any $\F$-compatible family $\alpha$, the triple   
$(S, \F, \alpha)$ is realizable by a block of a finite group $G$  
containing $S$ as a normal (and Sylow)  subgroup, 
hence   Conjecture  \ref{c:malle-navarro}  holds  by 
\cite[Theorem 2]{MalleRobinson2017}  and  Conjecture \ref{c:eatonmoreto}  
holds  by \cite{EatonMoreto2014}.

Let $\F$ be a saturated fusion system on a non-trivial finite $p$-group 
$S$ and let $\C$ be the full subcategory of $\F$ of nontrivial subgroups 
of $S$. Following the terminology in \cite{Linckelmann2009}, briefly
reviewed at the end of the next section, we denote 
by $S_\vartriangleleft(\C)$ the subcategory of the subdivision 
category $S(\C)$ of chains
$$\sigma=(Q_0 < Q_1 < \cdots < Q_m)$$
where the $Q_i$ are nontrivial subgroups of $S$ which are normal in the 
maximal term $Q_m$.
Such a chain $\sigma$ is called  fully \textit{$\F$-normalized} if 
$Q_0$ is fully $\F$-normalized, and either $m=0$ or $\sigma_{\geq 1}=$
$(Q_1<\cdots<Q_m)$ is fully $N_\F(Q_0)$-normalized. Denote by
$S_\triangleleft(\C)^f$ the set of all fully $\F$-normalized 
chains. For $\sigma \in$ $S_\triangleleft(\C)^f$, we denote by 
$N_\F(\sigma)$  the saturated fusion system on $N_S(\sigma)$ as in 
\cite[5.2, 5.3]{Linckelmann2009}.
By Proposition \ref{p:nfsigma} below, an $\F$-compatible family 
$\alpha$ induces a canonical $N_\F(\sigma)$-compatible family 
$\alpha(\sigma)$, for each fully $\F$-normalised chain $\sigma$ in 
$S_\triangleleft(\C)$. 
The translation to fusion systems of the Kn\"orr-Robinson reformulation 
of Alperin's Weight Conjecture in \cite{KnorrRobinson1989}
reads as follows. 

\begin{Conj}\label{c:KnRo}
Let $\F$ be a saturated fusion system on a finite non-trivial  $p$-group
$S$ and let $\alpha$ be an $\F$-compatible family. We have
$$\k(\F,\alpha) = \sum_\sigma\ (-1)^{|\sigma|} 
\k(N_\F(\sigma),\alpha(\sigma)) $$
where in the sum $\sigma$ runs over a set of representatives of the
isomorphism classes of fully normalised normal chains of
non-trivial subgroups of $S$.
\end{Conj}

 Again, one could consider versions of the above replacing $\k$ with 
$\m$ or $\m^*$. 
Taking into account defects of characters, we get the following 
conjecture, which is an analogue of Dade's ordinary conjecture:

\begin{Conj}\label{c:dade}
Let $\F$ be a saturated fusion system on a finite  non-trivial $p$-group 
$S$ and let $\alpha$ be an $\F$-compatible family of $\F$. For each 
$d \ge 0$ we have 
$$ \m(\F,\alpha,d)=\sum_{\sigma} 
(-1)^{|\sigma|} \m(N_\F(\sigma),\alpha(\sigma),d)\ ,$$
where in the sum $\sigma$ runs over a set of representatives of the
isomorphism classes of fully normalised normal chains of
non-trivial subgroups of $S$.
\end{Conj}

\section{Background material} \label{prelimSection}

\begin{Lem}[Thompson's $A \times B$ Lemma]\label{l:axb}
Let $S$ be a  finite $p$-group and $A \times B \le \Aut(S)$ be such that 
$A$ is a $p'$-group and $B$ is a $p$-group. If $A$ centralizes $C_S(B)$, 
then $A=1$.
\end{Lem}

\begin{proof}
See \cite[Theorem 5.3.4]{Gorenstein1980}.
\end{proof}

We will use standard terminology on saturated fusion systems, as
can be found in many sources, including \cite{CravenTheory}, 
\cite{AschbacherKessarOliver2011}), for instance.      
We assume familiarity with the notions of centralizers and
normalizers in fusion systems. 

\begin{Lem}\label{l:nfkcentric}
Let $\F$ be a saturated fusion system on a finite $p$-group $S$.
Fix $Q \le S$ and $K \le \Aut(Q)$. Assume that $Q$ is fully $K$-normalized.
Then $PQ$ is $\F$-centric for each $N_\F^K(Q)$-centric subgroup $P \le
N^K_S(Q)$. 
\end{Lem} 

\begin{proof}
The  argument given in the proof of \cite[Lemma 6.2]{BrotoLeviOliver2003}
generalizes: Let $P \le N_S^K(Q)$ be an $N_\F^K(Q)$-centric subgroup
and let $\varphi \in \Hom_\F(PQ,S)$. Then $\phi(PQ) \leq N_S^{\phi K
\phi^{-1}}(\phi(Q))\phi(Q)$. Since $Q$ is fully $K$-normalized in $\F$, 
there is a morphism 
\[
\psi \in \Hom_\F(N_S^{\phi K\phi^{-1}}(\varphi(Q))\phi(Q),S)
\]
such that $\psi\phi(Q) = Q$ and $(\psi \varphi)|_Q \in K$ by 
\cite[Proposition I.5.2(c)]{AschbacherKessarOliver2011}. This means that 
$\psi \varphi$ is a morphism in $\Hom_{N_\F^K(Q)}(PQ,S)$. Since 
$C_S(\varphi(PQ)) \le N_S^{\phi K \phi^{-1}}(\varphi (Q))$,
\[
\psi(C_S(\varphi(PQ))) \le C_S(\psi\varphi(PQ)) \le 
C_S(\psi \varphi(P)) \cap N_S^K(Q) \le \psi \varphi(P),
\]
where the middle inequality holds because $\psi\phi K \phi^{-1}\psi^{-1} =
K$, and where the last inequality holds since $P$ is $N_\F^K(Q)$-centric.
Hence, $C_S(\varphi(PQ)) \le \varphi(P) \leq \phi(PQ)$. Since $\varphi$
was chosen arbitrarily, this shows that $PQ$ is $\F$-centric.
\end{proof}

\begin{Lem}\label{l:cfxcent}
Let $x \in S$ be such that $\langle x \rangle$ is fully $\F$-centralized, and
fix $Q \le C_S(x)$. Then $Q$ is $C_\F(x)$-centric if and only if $Q$ is
$\F$-centric. Moreover, $\Out_{C_\F(x)}(Q)=C_{\Out_\F(Q)}(x)$ under either of
these assumptions.
\end{Lem}

\begin{proof}
Suppose first that $Q$ is $\F$-centric and let $P$ be $C_\F(x)$-conjugate to
$Q$. Then $C_{C_S(x)}(P) \le C_S(P) \le P$ and hence $Q$ is $C_\F(x)$-centric.
Conversely if $Q$ is $C_\F(x)$-centric then $x \in Z(C_S(x)) \le C_{C_S(x)}(Q)
\le Q$ so $Q=Q\langle x \rangle$ is $\F$-centric by Lemma \ref{l:nfkcentric}
applied in the case $K=1$. Since $\Out_\F(Q)$ acts by conjugation on $Z(Q)$,
$C_{\Out_\F(Q)}(x)$ is well-defined. Now $\Aut_{C_\F(x)}(Q)=C_{\Aut_\F(Q)}(x)$
is exactly the set of $\F$-automorphisms of $Q$ which fix $x$, and this group
contains $\Inn(Q)$ by assumption. The lemma follows.
\end{proof}

Given an isomorphism $\varphi$ in $\F$ from $Q$ to $Q'$, the conjugation
map $c_\phi\colon \Aut_\F(Q) \to \Aut_\F(Q')$ given by $\eta \to \varphi \eta
\varphi^{-1}$ is an isomorphism which maps $\Inn(Q)$ onto $\Inn(Q')$. Thus,
conjugation induces a well-defined isomorphism $\Out_\F(Q) \to \Out_\F(Q')$,
which we denote by $\bar{c}_\phi$. 
The following direct application of the extension axiom is needed in
Section \ref{s:towards}.

\begin{Lem}\label{l:chain} 
Let $Q$ and $Q'$ be $\F$-centric subgroups of $S$, and let $R$ be a 
subgroup of $S$ containing $Q$ as a normal subgroup. Let 
$\phi : Q\to Q'$ be an isomorphism in $\F$. Assume
that $c_\phi(\Aut_R(Q)) \leq \Aut_S(Q')$, or, equivalently, that
$\bar{c}_\phi(\Out_R(Q)) \leq \Out_S(Q')$.  Let $R' \leq S$ be the inverse
image of $c_\phi(\Aut_R(Q))$ under the canonical homomorphism $N_S(Q') \to
\Aut_S(Q')$.  Then there exists a morphism $R \to S$ in $\F$ extending 
$\varphi$. Moreover, $\tau(R) = R'$ for any such extension $\tau$. 
\end{Lem}

\begin{proof}  
Since $\Aut_S(Q')$ is the full inverse image of $\Out_S(Q')$ under the
canonical surjection $\Aut_{\F}(Q') \to \Out_{\F}(Q')$, the two conditions
on the image of $R$ are indeed equivalent. Hence, $R \leq N_\phi$ in the
notation of \cite[Definition~2.2]{AschbacherKessarOliver2011}. Since each
$\F$-centric subgroup is fully $\F$-centralised, the extension axiom of
saturation yields the first assertion. 

If $\tau$ and $\tau'$ are two $\F$-morphisms extending $\phi$, then one may
find $z \in Z(Q)$ such that $\tau' = \tau \circ c_z$ by
\cite[Lemma~A.8]{BrotoLeviOliver2003}. Since $z \in Q \leq R$, this shows the
second assertion.
\end{proof}

 Let $\C$ be a full subcategory of $\F$ which is closed under
isomorphisms and taking supergroups. Following the notation in 
\cite[Section 8.13]{LinckelmannBT}, we denote by $S(\C)$ the
subdivision category of $\C$. The objects of $\C$ can be
regarded as non-empty chains of non-isomorphisms 
$$Q_0 \to Q_1 \to \cdots \to Q_m$$
in $\F$ with $Q_i$ belonging to $\C$. Any homomorphism in $S(\C)$ is
a composition of a chain preserving isomorphism in $\F$ and
an inclusion of a chain as a subchain of another chain. There is
a canonical functor $S(\C)\to$ $\C$ mapping a chain to its 
maximal term. 
 
By \cite[Proposition 8.13.3]{LinckelmannBT}, any chain in $S(\C)$
is isomorphic, in $S(\C)$, to a chain of proper inclusions
$$Q_0 < Q_1 <\cdots < Q_m$$
of subgroups $Q_i$ of $S$ belonging to $\C$. In other words, the
category $S(\C)$ is equivalent to its full subcategory, denoted
$S_<(\C)$ consisting of non-empty chains of proper inclusions of
subgroups of $S$ in $C$.
A  chain $\sigma$  above  is said to
have length $m$, and we write $|\sigma| = m$.  When convenient, we occasionally
write $Q_\sigma$ and $Q^{\sigma}$ for the smallest and largest subgroups in
$\sigma$, respectively. 

A morphism between chains $Q_0 < \cdots < Q_m$ and $R_0 < \cdots < R_n$ is a
pair consisting of an injective map $\beta \colon \{0,\dots,m\} \to
\{0,\dots,n\}$ together with a collection of isomorphisms $Q_i \to R_{\beta(i)}$
in $\F$ for each $i \in \{0,\dots,m\}$ which satisfy the obvious
compatibility conditions. Thus, the set of isomorphisms between chains
$\sigma$, $\tau$ in $S_<(\C)$ can be identified with the set of
chain-preserving isomorphisms $\varphi: Q^\sigma \to Q^\tau$ in $\F$. Whenever $\sigma \in
S_<(\C)$, let $\Aut_{\F}(\sigma)$ be the subgroup of
$\Aut_{\F}(Q^\sigma)$ consisting of those automorphisms which preserve each
member of the chain. In other words, $\Aut_{\F}(\sigma)$ is the automorphism
group of $\sigma$ in  $S_<(\C)$. 

We denote by $S_\triangleleft(\C)$ the
full subcategory of $S_<(\C)$ of all chains
$$Q_0 < Q_1 <\cdots < Q_m$$
in $S_<(\C)$ satisfying the additional property that the $Q_i$ are 
normal in the maximal term $Q_m$, for $0\leq i\leq m$.

We denote the set of isomorphism classes of chains in $S(\C)$ by 
$[S(\C)]$. Since $\C$, and hence $S(\C)$, is an EI-category, the set 
$[S(\C)]$ has a canonical partial order given by $[\sigma]\leq [\tau]$, 
whenever $[\sigma]$, $[\tau]$ are the isomorphism classes of chains 
$\sigma$, $\tau$ in $S(\C)$ such that $\Hom_{S(\C)}(\sigma,\tau)$ is
non-empty.

If $\F=$ $\F_S(G)$ for some finite group $G$ having $S$ as a Sylow 
$p$-subgroup, then $[S_<(\C)]$ is isomorphic to the poset of 
$G$-conjugacy classes of chains of subgroups in $\C$. For a more general 
statement regarding $G$-conjugacy classes of chains of Brauer pairs of 
a block, see \cite[Proposition 4.6]{Linckelmann2005}. 

\section{Compatible families of second cohomology classes}  
\label{s:KP}

We describe properties of K\"ulshammer-Puig classes of blocks which are 
needed to ensure that the conjectures stated for saturated fusion systems
do indeed specialize to the block theoretic versions from which they
are inspired in case the triple $(S,\F,\alpha)$ under consideration
is realized by a block. We briefly review the construction of
K\"ulshammer-Puig classes (see e. g. 
\cite[Theorem 5.3.12, Corollary 8.12.9, Section 8.14]{LinckelmannBT} 
for more details and proofs).

Let $M$ be a finite-dimensional simple $k$-algebra; that is, $M$ is
isomorphic to a matrix algebra over $k$. Let $G$ be a finite group
acting on $M$ by algebra automorphisms. By the Skolem-Noether Theorem,
every automorphism of $M$ is inner, and hence for any $g\in$ $G$ there
is an element $s_g\in$ $M^\times$ such that the action of $g$ is 
equal to the conjugation action of $s_g$ on $M$. Since $Z(M)\cong$ $k$,
the elements $s_g$ are only unique up to scalars in $k^\times$.
Thus for $g$ , $h\in$ $G$ we have $s_gs_h =\alpha(g,h) s_{gh}$ for
some $\alpha(g,h)\in$ $k^\times$. The map $\alpha : G\times G\to$
$k^\times$ is then a $2$-cocycle whose class in $H^2(G,k^\times)$ is
independent of the choices of the $s_g$. We call this class the
{\it class determined by the action of $G$ on $M$.} 
If $G$ acts trivially on $M$, then $\alpha$ is the trivial class.

Suppose now that $G$ has a normal subgroup $N$ 
such that the action of $N$ on the simple algebra $M$ lifts to a 
$G$-stable group homomorphism $\tau : N\to$ $M^\times$.
Let $[G/N]$ be a set of representatives of $G/N$ in $G$. For each
$g\in$ $[G/N]$ choose some $s_g$ as above, and for each $h\in N$
set $s_{gh}=s_g\tau(h)$. One checks that the $2$-cocycle $\alpha$ 
determined by this choice has the property that its values 
$\alpha(g,h)$ depend only on the images of $g$, $h$ in $G/N$, for all 
$g$, $h\in$ $G$, and hence $\alpha$ induces a $2$-cocycle $\beta$ on 
$G/N$ whose class in $H^2(G/N,k^\times)$ does not depend on the choices 
of the $s_g$ (but the class of $\beta$  does depend on the choice $\tau$ 
lifting the action of $N$ on $M$). We call this class {\it the class 
determined by the action of $G$ on $M$ together with the group 
homomorphism  $\tau$.} Even if $G$ acts trivially on $M$ this does
not necessarily imply that $\beta$ is trivial (this depends on
whether $\tau$ is trivial). 

This  scenario   arises if $M$ is a simple algebra quotient of $kN$ by a 
$G$-stable maximal ideal in $kN$. Here the action of $G$ is the 
conjugation action and the map $\tau$ is induced by the canonical 
algebra surjection $kN\to$ $M$.  Any   such  scenario   determines  a 
class $\beta$ in $H^2(G/N, k^\times)$ whose restriction to $G$ along 
the canonical surjection $G\to$ $G/N$ is equal to the class $\alpha$ 
determined by the action of $G$ on $M$. For technical Clifford theoretic 
reasons it is usually more convenient to consider the inverse class. 

The K\"ulshammer--Puig classes arise in turn as special cases of this
construction. Let $B$ be a block of $kG$ with maximal $B$-Brauer pair 
$(S,e)$ and associated fusion system $\F$ on $S$. Let $Q$ be an 
$\F$-centric subgroup of $S$. That is, if $f$ is the unique block of 
$kC_G(Q)$ satisfying $(Q,f)\leq$ $(S,e)$, then $Z(Q)$ is a defect group 
of $f$ (which is clearly central), and hence $kC_G(Q)f$ is a nilpotent 
block with a unique simple algebra quotient $M_Q$. The uniqueness ensures 
that $M_Q$ is $N_G(Q,f)$-stable. By standard facts, $M_Q$ is also the
unique simple algebra quotient of $kQC_G(Q)f$. Note that $QC_G(Q)$ is a 
normal subgroup of $N_G(Q,f)$, and that $N_G(Q,f)/QC_G(Q)\cong$ 
$\Out_\F(Q)$. Thus the previous scenario with $N_G(Q,f)$ and $QC_G(Q)$ 
instead of $G$ and $N$, respectively, yields a canonical class in 
$H^2(\Out_\F(Q), k^\times)$. The {\it inverse} of this class is the 
K\"ulshammer--Puig class $\alpha_Q$. Using $N_G(Q,f)$ and $C_G(Q)$ 
would yield the corresponding class, abusively again denoted 
$\alpha_Q$, in $H^2(\Aut_\F(Q), k^\times)$. 

\medskip
Let $\F$ be a saturated fusion system on a finite  $p$-group $S$. We 
denote by $\F^c$ the full subcategory of $\F$-centric subgroups of $S$. 
For any $Q\in \F^c $, we may (and will) identify without further
comment the group $H^2(\Out_\F(Q), k^\times)$ with 
$H^2(\Aut_\F(Q),k^\times)$ via the isomorphism induced by the canonical 
surjection  $ \Aut_{\F}(Q) \to \Out_{\F}(Q)$.  The assignment
$Q\mapsto$ $H^2(\Out_\F(Q), k^\times)$ is not functorial on $\F^c$. 
In order to interpret certain families of classes in 
$\prod_{Q\in\F^c}\ H^2(\Out_\F(Q), k^\times)$ as a limit of a functor, 
we need to pass to the subdivision category $S(\F^c)$ of $\F^c$.   
By \cite[Theorem 1.1]{Linckelmann2009b}, there is a canonical functor    
$\A_\F^2$ from $[S(\F^c)]$ to the category of abelian groups which sends   
an  object $\tau $ of $[S(\F^c)]$ to  
$H^2( \Aut_{S(\F^c)}(\sigma), k^\times )$ for some $\sigma \in $ 
$S(\F^c)$ such that $\tau=[\sigma]$. The choice of representative
$\sigma$ determines this functor up to unique isomorphism.
Let $\alpha= (\alpha_Q)_{Q\in\F^c}$ be a family  of classes  
$\alpha_Q\in$ $H^2(\Out_\F(Q),k^\times)$.  For each $\tau \in [S(\F^c)]$, 
define the element $\alpha_{\tau}\in $ $\A_{\F}^2(\tau) $ to  be the 
restriction  of $\alpha_{Q_m}$ to the subgroup
$\Aut_{S(\F^c)}(\sigma)$ of $\Aut_\F(Q_m)$  where   
$$\sigma = (Q_0  \rightarrow Q_1\rightarrow \cdots  \rightarrow Q_m)$$    
is the representative of $\tau $ in $S[\F^c]$  as above.

\begin{Def} \label{def:Fcompatible}
Let $\F$ be a saturated fusion system on a finite $p$-group $S$.
An {\it $\F$-compatible family} is a  family 
$\alpha = (\alpha_Q)_{Q\in\F^c}$ of classes
$\alpha_Q\in$ $H^2(\Out_\F(Q),k^\times)$ such that the corresponding
family $(\alpha_\tau)_{\tau \in  [S(\F^c)]}$ as above belongs to 
$\displaystyle\lim_{[S(\F^c)]} \A_\F^2$. In that case, we write
$\alpha\in$ $\displaystyle\lim_{[S(\F^c)]} \A_\F^2$ for short. 
\end{Def} 

The set of $\F$-compatible classes forms a subgroup of the abelian
group $\prod_{Q\in \F^c}\ H^2(\Out_\F(Q),k^\times)$. 

By \cite[Theorem 8.14.5]{LinckelmannBT}, the  family $\alpha$ of
K\"ulshammer-Puig classes of a block $B$ of some finite group algebra 
$kG$ with defect group $S$ and fusion system $\F$ is $\F$-compatible. 
By \cite[Theorem 4.7]{Linckelmann2009} the inclusions of categories 
$S_\triangleleft(\F^c)\subseteq$
$S_<(\C) \subseteq$ $S(\C)$ induce isomorphisms
$$\lim_{[S(\F^c)]} \A_\F^2\cong \lim_{[S_<(\F^c)]} \A_\F^2\cong
\lim_{[S_\triangleleft(\F^c)]} \A_\F^2$$
Thus  to check  $\F$-compatibility it suffices to consider
normal chains. In fact, it suffices to consider normal chains
of length at most $1$.

\begin{Lem}[{\cite[Theorem 8.14.5]{LinckelmannBT} and its proof}] 
\label{lem:Fcompatible}
Let $\F$ be a saturated fusion system on a finite $p$-group $S$, and
let $\alpha=$ $(\alpha_Q)_{Q\in\F^c}$ with $\alpha_Q\in$
$H^2(\Out_\F(Q);k^\times)$ for any $\F$-centric subgroup $Q$ of $S$.
The following are equivalent.
\begin{enumerate}
\item The family $\alpha$ is $\F$-compatible.
\item For any proper normal $\F$-centric subgroup $Q$ of an
$\F$-centric subgroup $R$ of $S$, the images of $\alpha_Q$ and 
$\alpha_R$ in $H^2(\Aut_{S(\F^c)}(Q\triangleleft R), k^\times)$ under 
the maps induced by the canonical group homomorphisms
$$\Aut_{S(\F^c)}(Q\triangleleft R) \to \Aut_\F(Q)$$
$$\Aut_{S(\F^c)}(Q\triangleleft R) \to \Aut_\F(R)$$
are equal. 
\end{enumerate}
\end{Lem}

We need to follow compatible families through passages to centralizers 
of elements and normalizers of chains of $p$-subgroups. 

\begin{Lem} \cite[Proposition 8.3.7]{LinckelmannBT}
\label{lem:ccentric}
Let $\F$ be a saturated fusion system on a finite $p$-group $S$, and
let $Q$ be a fully $\F$-centralized subgroup of $S$. If $R$ is a
$C_\F(Q)$-centric subgroup of $C_S(Q)$, then $QR$ is an $\F$-centric
subgroup of $S$. The correspondence $R\mapsto QR$ extends to a
unique functor 
$$C_\F(Q)^c \to \F^c$$
which sends a morphism $\varphi : R\to$ $R'$ in $C_\F(Q)^c$ to the
unique morphism $\psi : QR\to QR'$ in $\F^c$ which is the identity
on $Q$ and coincides with $\varphi$ on $R$.
\end{Lem}

This functor extends obviously to a functor between 
subdivision categories, and hence this functor sends an
$\F$-compatible family $\alpha$ to a $C_\F(Q)$-compatible family
$\alpha(Q)$. In order to ensure that the conjectures involving
this functor specialize to known facts or conjectures, we need to
check that if $\alpha$ is realized by a block $B$ of $kG$, then
$\alpha(Q)$ is realized by the corresponding block of $kC_G(Q)$.

\begin{Prop} \label{p:blockalpha}
Let $G$ be a finite group, $B$ a block of $kG$, and $(S,e)$ a
maximal $B$-Brauer pair. Let $\F$ be the fusion system of $B$
on $S$ determined by the choice of $e$, and let $\alpha=$
$(\alpha_Q)_{Q\in\F^c}$ be the family of K\"ulshammer--Puig classes
of $B$. Denote by $e_Q$ the unique block of $kC_G(Q)$ such that
$(Q,e_Q)\leq$ $(S,e)$ and by $f$ the unique block of 
$C_{C_G(Q)}(C_S(Q))=$ $C_G(QC_S(Q))$ 
satisfying $(C_S(Q),f)\leq$ $(S,e)$. Then $(C_S(Q),f)$ is 
a maximal $(C_G(Q),e)$-Brauer pair which determines the fusion
system $C_\F(Q)$ on $C_S(Q)$. The restriction of $\alpha$ to
a family $\alpha(Q)$ along the canonical functor $C_\F(Q)^c\to$
$\F^c$ is the family of K\"ulshammer--Puig classes of the
block $kC_G(Q)e_Q$ with respect to the maximal $(C_G(Q),e_Q)$-Brauer
pair $(C_S(Q),f)$.
\end{Prop}

\begin{proof}
The fact that $(C_S(Q),f)$ is a maximal $(C_G(Q),e)$-Brauer pair 
which determines the fusion system $C_\F(Q)$ on $C_S(Q)$ is 
well-known, and proved, for instance, in 
\cite[Proposition 8.5.4]{LinckelmannBT}. 
For the statement on K\"ulshammer--Puig classes, we need the
contruction of these classes as reviewed at the beginning of
this Section.  
Let $R$ be a $C_\F(Q)$-centric subgroup of $C_S(Q)$. 
By \ref{lem:ccentric}, $QR$ is $\F$-centric. Note that
$C_{C_G(Q)}(R)=$ $C_G(QR)$. Thus if $g$ is the unique block
of $kC_G(QR)$ such that $(QR,g)\leq$ $(S,e)$, then $g$ is also
the unique block of $kC_{C_G(Q)}(R)$ such that $(R,g)\leq$ 
$(C_S(Q),f)$. These blocks have therefore the same unique
simple quotient (as they are nilpotent blocks), and clearly
$N_{C_G(Q)}(R,f)$ is a subgroup of $N_G(QR,f)$. Since the
K\"ulshammer--Puig classes of $R$ and $QR$ for $C_\F(Q)$
$\F$ are determined by the respective actions of the groups
$N_{C_G(Q)}(R,f)$ and $N_G(QR,f)$ on that simple quotient, it
follows that the class of $R$ in $C_\F(Q)$ is indeed obtained 
from restricting the class of $QR$ in $\F$ along the canonical
map $\Aut_{C_\F(Q)}(R)\to$ $\Aut_\F(QR)$.
\end{proof}

We apply this for cyclic $Q$. 
Let $x$ be an element in $S$ such that $\langle x\rangle$ is fully
$\F$-centralized. For $\alpha$ an $\F$-compatible family, we denote 
by $\alpha(x)$ the corresponding $C_\F(x)$-compatible family, 
obtained from restricting $\alpha$ along the canonical functor
$$C_\F(x)^c \to \F^c$$
from Proposition \ref{lem:ccentric} applied with $Q=$ 
$\langle x\rangle$. 
By Proposition \ref{p:blockalpha}, if $\alpha$ is a family of
K\"ulshammer--Puig classes of a block, then $\alpha(x)$ is a family of
K\"ulshammer--Puig classes of the relevant Brauer correspondent of the 
block.

\begin{Prop} \label{p:links} 
Suppose that $(S,\F,\alpha)$ is realizable by a block $B$ of a finite
group algebra $kG$. Then $\w(\F, \alpha)$ is the number of weights  
associated with $B$. In particular, AWC holds for $B$ if and only if 
$\w(\F,\alpha)=$ $\ell(B)$.  Moreover, if AWC holds  for $B$ and all 
its Brauer pairs, then $\k(\F,\alpha)=$ $\k(B)$, the number of 
ordinary irreducible characters associated with $B$. 
\end{Prop}

\begin{proof}    
For the first assertion see for instance 
\cite[Proposition 5.4]{Kessar2007}. The fusion system $\F$ is
determined by a choice of a block $e$ of $kC_G(S)$ such that
$(S,e)$ is a maximal $B$-Brauer pair (see e. g. 
\cite[Definition 3.8]{Kessar2007}). Let $x \in S$ such that  
$\langle x \rangle$  is fully $\F$-centralized. Let $f$ be
the block of $kC_G(x)$ such that $(\langle  x \rangle, f)$ is the unique 
$B$-Brauer pair contained in $(S,e)$. 
By Proposition \ref{p:blockalpha}, the triple
$(C_S(x), C_{\F}(x), \alpha(x))$  is realized  by the block $f$ of
$kC_G(x)$, and hence it follows that 
$\w(C_\F(x),\alpha(x))=$ $\ell(kC_G(x)f)$ thanks to the assumption
that $B$-Brauer pairs satisfy AWC. A theorem of Brauer (cf. 
\cite[Theorem 6.13.12]{LinckelmannBT}) now implies the second 
assertion (see also \cite[IV. 5.7]{AschbacherKessarOliver2011}). 
\end{proof}

For $\F$ a saturated fusion system on a finite $p$-group $S$, denote
by $\bar\F$ the associated orbit category, obtained from $\F$ by taking 
as morphisms the orbits $\Inn(R)\backslash \Hom_\F(Q,R)$ of morphisms 
in $\F$ from $Q$ to $R$ modulo inner automorphisms of $R$, for any two 
subgroups $Q$, $R$ of $S$. In particular, $\Out_\F(Q)\cong$
$\Aut_{\bar\F}(Q)$. Recall from \cite[Definition 5.1]{Linckelmann2009}
that a normal chain
$$\sigma = (Q_0<Q_1<\cdots <Q_m) \in S_\triangleleft(\F)$$ 
is called {\it fully $\F$-normalized} if $Q_0$ is fully
$\F$-normalized and if either $m=0$ or the chain
$$\sigma_{\geq 1}= (Q_1<\cdots <Q_m)$$ 
is fully $N_\F(Q_0)$-normalized. Every chain in $S_\triangleleft(\F)$ 
is isomorphic to a fully $\F$-normalized chain. Note that since $\sigma$ 
is a normal chain, we have $Q_mC_S(Q_m)\leq N_S(\sigma)$. 
We need an analogue of Proposition \ref{p:blockalpha} for
$N_\F(\sigma)$.

\begin{Prop}\label{p:nfsigma}
Let $\F$ be a saturated fusion system on a finite $p$-group $S$ and 
let $\alpha$ be an $\F$-compatible family. Let 
$\sigma = (Q_0<Q_1<\cdots <Q_m)$ $\in S_\triangleleft(\F)$ be fully 
$\F$-normalized. 

\begin{enumerate}
\item
For every $P \le N_S(\sigma)$, if $P$ is $N_\F(\sigma)$-centric, 
then $Q_mP$ is $\F$-centric.

\item
Let $P$, $R$ be $N_\F(\sigma)$-centric subgroups of $N_S(\sigma)$,
let $\varphi : P\to$ $R$ a morphism in $N_\F(\sigma)$, and let
$\psi, \psi'  : Q_mP\to$ $Q_mR$ be morphisms in $\F$ extending $\varphi$
and satisfying $\psi(Q_i)=Q_i=\psi'(Q_i)$ for $0\leq i\leq m$. 
Then the classes of $\psi$ and $\psi'$ are conjugate by an
element in $Z(P)$. In particular, the correspondence sending $\varphi$
to any choice of $\psi$ induces a functor 
$$\Psi : N_\F(\sigma)^c\to \bar\F^c.$$

\item
For any $N_\F(\sigma)$-centric subgroup $P$ of $N_S(\sigma)$, 
the functor $\Psi$ induces a group homomorphism
$$\Out_{N_\F(\sigma)}(P) \to \Out_\F(Q_mP)\ ,$$
and the restriction along these group homomorphisms induces a map 
from the group of $\F$-compatible families to the group of
$N_\F(\sigma)$-compatible families.

\item
If $(S,\F,\alpha)$ is realized by a block $B$ with respect to
a maximal $B$-Brauer pair $(S,e)$, then 
$(N_S(\sigma), N_\F(\sigma), \alpha(\sigma))$ is realized by the
block $e_m$ of $kN_G(\sigma, e_m)$ such that $(Q_m,e_m)\leq$
$(S,e)$, with respect to the maximal $(N_G(\sigma, e_m),e_m)$-Brauer
pair $(N_S(\sigma), f)$, where $f$ is the unique block of
$C_{N_G(\sigma)}(N_S(\sigma))=$ $C_G(N_S(\sigma))$ satisfying
$(N_S(\sigma), f)\leq$ $(S,e)$. 
\end{enumerate}
\end{Prop}

\begin{proof}
In order to prove the first statement, we argue by induction over the 
length $m$ of the chain $\sigma=$ $Q_0<Q_1<\cdots <Q_m$. Suppose that 
$m=0$, so $\sigma=$ $Q_0$, and $Q_0$ is fully $\F$-normalised. 
Let $P$ be an $N_\F(Q_0)$-centric subgroup of $N_S(Q_0)$. 
Then $Q_0P$ is $\F$-centric by Lemma \ref{l:nfkcentric}. Suppose now 
that $m>0$. Let $P\leq$ $N_S(\sigma)$ be $N_\F(\sigma)$-centric.
Set $\sigma'=$ $Q_0<Q_1<\cdots<Q_{m-1}$ and $\F'=$ $N_\F(\sigma')$. By 
\cite[5.4]{Linckelmann2009}, $\sigma'$ is a fully $\F$-normalized
chain, and $Q_m$ is fully $\F'$-normalized. By the statement for $m=0$ 
applied to $\F'$, it follows that $Q_mP$ is $\F'$-centric. By induction,
$Q_mP$ is $\F$-centric. 

For the second statement, note that the two extensions $\psi$, $\psi'$
of $\varphi$ are both again morphisms in $N_\F(\sigma)$, and their
restrictions to the $N_\F(\sigma)$-centric subgroup $P$ coincide. Thus, 
by a standard fact (see e. g. \cite[Lemma~A.8]{BrotoLeviOliver2003}) 
they differ by conjugation with an element in $Z(P)$. That means that 
the image of $\psi$ in the orbit category $\bar\F$ is uniquely determined 
by $\varphi$, whence  the second statement. The third statement is a 
formal consequence of the second.

For the proof of the fourth statement, note first that this makes
sense: we have $C_G(Q_m)\leq N_G(\sigma, e_m)\leq$ $N_G(Q_m,e_m)$, 
and $\Aut_{S(\F)}(\sigma)\cong$ $N_G(\sigma, e_m)/C_G(Q_m)$. In
particular, by standard block theory, $e_m$ remains a block of
$kN_G(\sigma,e_m)$. An interated application of
\cite[Proposition 8.5.4]{LinckelmannBT} shows that $N_\F(\sigma)$ is
the fusion system of this block with respect to the maximal
Brauer pair as stated. The same argument as at the end of the proof
of Proposition \ref{p:blockalpha} shows that restricting $\alpha$
yields the family of K\"ulshammer--Puig classes of $e_m$ as a 
block of $kN_G(\sigma, e_m)$.  
\end{proof} 

Recall that a saturated fusion system $\F$ on a finite $p$-group $S$ is 
\textit{constrained} if $\F=$ $N_\F(Q)$ for some normal $\F$-centric 
subgroup $Q$ of $S$. In that case, by \cite[Proposition C]{BCGLO2005} (see  \cite[Theorem~4.9]{AschbacherKessarOliver2011}), 
$\F$ is the fusion system of a finite group $L$ with $S$ as 
Sylow $p$-subgroup, such that $Q$ is normal in $L$ satisfying 
$C_L(Q)=Z(Q)$; that is, $L$ is $p$-constrained. In particular,
we have canonical isomorphisms $L/Q\cong$ $\Out_\F(Q)$ and $L/Z(Q)\cong$ 
$\Aut_\F(Q)$.   The group $L$ is called a  model for $\F$.

\begin{Prop}[{\cite[Section 6]{Linckelmann2009b}}]
\label{constrained-alpha}
Let $\F$ be a saturated fusion system on a finite $p$-group $S$ such that
$\F=$ $N_\F(Q)$ for some normal $\F$-centric subgroup $Q$ of $S$. 
Let $L$ be a finite group such that $S$ is a Sylow $p$-subgroup of $L$,
such that $Q$ is normal in $L$ satisfying $C_L(Q)=Z(Q)$, and such
that $\F=$ $\F_S(L)$.
The restriction from $\F^c$ to $\Aut_\F(Q)$ and the canonical map $L\to$ 
$\Aut_\F(Q)$ induce isomorphisms
$$H^2(\F^c, k^\times)\cong H^2(\Aut_\F(Q), k^\times)\cong 
H^2(L, k^\times) .$$
In particular, any $\F$-compatible family $\alpha$ is uniquely 
determined by the component $\alpha_Q$. 
\end{Prop}

\begin{Prop}[{cf. \cite[Proposition IV.5.34]{AschbacherKessarOliver2011},
\cite[5.3]{Linckelmann2004}} ] 
\label{p:fareal}
Let $\F$ be a saturated fusion system on a finite $p$-group $S$ such that
$\F=$ $N_\F(Q)$ for some normal $\F$-centric subgroup $Q$ of $S$. 
Let $\alpha$ be an $\F$-compatible family.
Let $L$ be a finite group such that $S$ is a Sylow $p$-subgroup of $L$,
such that $Q$ is normal in $L$ satisfying $C_L(Q)=Z(Q)$, and such
that $\F=$ $\F_S(L)$. Choose a finite cyclic subgroup $Y$ of $k^\times$
containing all values of a $2$-cocycle representing the class $\alpha_Q$.
Then $(S,\F,\alpha)$ is realized by a block of the central extension 
$\hat L$ of $L$ by $Y$ determined by $\alpha_Q$, regarded as a class
in $H^2(L,Y)$.
\end{Prop}

In particular $\alpha=0$ if and only if $b$ is the principal block of 
$k\hat L$ (which is isomorphic to the principal block of $kL$). 
More generally, the blocks arising in the previous Proposition are 
twisted group algebras of $L$; we   lay out  the connection between 
$p'$-central extensions and twisted group algebras in the next result

 \begin{Prop}\label{p:centext} 
Let $G$ be a finite group, and $\alpha \in  H^2(G, k^{\times})$.     
\begin{enumerate}

\item There exists a central extension 
$$ 1 \to Z \to  \widetilde G  \to G \to 1 $$  
where $Z$ is a cyclic group of order prime to $p$ and a  primitive 
idempotent $e$ of $kZ$ such for any subgroup $L$ of $G$, we have
an isomorphism $k_{\alpha} L \cong k \widetilde L e$, where
$\widetilde L$ is the inverse image of $L$ in $\widetilde G$. 
In particular $\ell(k_{\alpha} L) =$ $\ell (k\widetilde Le)$ and 
$z(k_{\alpha}L) = z(k\widetilde Le)$.    
 
\item  
Suppose  that  there exists  a normal $p$-subgroup  $Q$ of $G$   
such that $C_G(Q) =Z(Q)$.  Identify $\alpha$ with the corresponding 
element of $H^2(G/Q, k^{\times})$. Let $L$ be a subgroup of $G$  
containing  $Q$,  $S$ a  Sylow $p$-subgroup of $L$, and $\hat S$   
the Sylow $p$-subgroup of the inverse image of $S$ in $\widetilde L$.     
Denote also by $\alpha$ the $\F_S(L)$-compatible family 
determined by the restriction of $\alpha$ to $L$  as in   
Proposition \ref{constrained-alpha}.
Then, $k\widetilde L e $  is a block of $k\widetilde L$   realizing   
$(S, \F_S(L), \alpha)$  through the canonical isomorphism  
$\hat S\cong$ $S$. Moreover, AWC  holds  for $k\widetilde L e$ if  
and only  if 
$$ \ell(k_{\alpha}L) = \sum_{R} z(k_{\alpha} N_{L/Q} (R) /R)$$  
where $R$ runs over a set of representatives of the $L/Q$-classes 
of $p$-subgroups of $L/Q$.
\end{enumerate}
\end{Prop}

\begin{proof} 
Since $k$ is algebraically closed it is well-known that 
$H^2(G,k^\times)$ is finite, and hence $\alpha$ can be represented by 
a $2$-cocycle, abusively still denoted by $\alpha$, with values in
a finite subgroup $Z$ of $k^\times$. Then $Z$ is cyclic of order
prime to $p$, since $k$ is a field of characteristic $p$. Represent
$\alpha$ by a central extension 
$$ 1 \to Z \to  \widetilde G  \to G \to 1 $$  
and denote, for any $x\in$ $G$, by $\tilde x$ an inverse
image of $x$ in $\widetilde G$ satisfying $\tilde x\tilde y=$
$\alpha(x,y)\widetilde{xy}$ for all $x$, $y\in$ $G$. We regard the 
elements of $Z$ as elements in the centre of $\widetilde G$ and not as 
scalars; if we do want to consider the elements of $Z$ as scalars, we 
denote this via the inclusion map $\iota : Z\to$ $k^\times$. Set 
$e=$ $\frac{1}{|Z|}\sum_{z\in Z} \iota(z^{-1})z$. This is a primitive 
idempotent in $kZ$, and $kZe$ is $1$-dimenional. An easy verification 
shows that the map sending $\tilde xe\in$ $k\widetilde Ge$ to $x$ 
induces an algebra isomorphism $k\widetilde G e\cong k_\alpha G$. This
isomorphism restricts to an isomorphism $k\widetilde Le\cong$
$k_\alpha L$, for any subgroup $L$ of $G$. Statement (1) follows.

Let $\hat Q$ be the Sylow $p$-subgroup  of the inverse image 
$\widetilde Q$ of $Q$ in $\widetilde L$. Then $\widetilde Q=$ 
$Z\times \hat Q$, and hence $\hat Q$ is normal in $\widetilde L$. Thus 
all block idempotents of $k\widetilde L$ lie in 
$kC_{\widetilde L}(\hat Q) = k(Z(\hat Q) \times Z)$. In other words, the  
block idempotents of $k\widetilde L$ are  precisely the primitive 
idempotents  of $kZ$. In particular, $k\widetilde L e $  is a block of 
$k \widetilde L$. One easily checks that this block has defect group 
$\hat S$, which is isomorphic to $S$, and through this isomorphism, 
$\F=$ $\F_S(L)$ is the (in this situation unique) fusion system on $S$ 
of the block $e$ of $\widetilde L$.
We need to show that $\alpha$ is the family of K\"ulshammer--Puig
classes of this block. By Proposition \ref{constrained-alpha}, it
suffices to show this for the class $\alpha_{\hat Q}$. We write
again $\alpha$ instead of $\alpha_Q$, and consider $\alpha$ as
a class of $H^2(L, k^\times)$ whenever appropriate.
Note that $e$ remains the unique block of $C_{\widetilde L}(\hat Q)=$
$Z(Q)\times Z$ such that $(\hat Q,e)$ is a $(\widetilde L,e)$-Brauer
pair. So the construction of the K\"ulshammer--Puig class at $\hat Q$
is obtained as the special case of the construction described at
the beginning of this section with $\widetilde L$ and $Z\times \hat Q$
instead of $G$ and $N$, respectively, and with the $1$-dimensional 
quotient $M\cong$ $k$ of $k(Z\times \hat Q)$ given by the map 
$\iota : Z\to k^\times$ extended trivially to $\hat Q$, still denoted 
by $\iota$. Since any group action on a $1$-dimensional algebra is 
trivial, we may choose $s_x=1$ for $x$ running over a set of 
representatives of $\widetilde L/(\hat Q\times Z)\cong$ $L/Q$. Then 
also $s_x=1$ for $x$ running over a set of representatives of 
$\widetilde L/Z\cong$ $L$, because $\iota$ is extended trivially to 
$\hat Q$. Thus, for a general element of the form $\tilde x z$, with 
$x\in$ $L$ and $z\in$ $Z$, we may choose $s_{\tilde xz}=$ $\iota(z)$;
in particular, $s_{\tilde x}=1$ for $x\in$ $L$. 
We need to show that this determines $\alpha^{-1}$. Note that
$\alpha$ is determined by its restriction to $L$ via the map
$L\to$ $L/Q$. Let $x$, $y\in$ $L$. By construction, we have
$s_{\tilde x}=s_{\tilde y}=s_{\widetilde{xy}}=1.$
Since $\tilde x\tilde y = \widetilde{xy}\alpha(x,y)$, it follows
that 
$$s_{\tilde x\tilde y}=s_{\widetilde{xy}}\iota(\alpha(x,y))=
\iota(\alpha(x,y))$$
and hence (writing $\alpha$ instead of $\iota\circ\alpha$) we have
$$1=s_{\tilde x}s_{\tilde y} = \alpha(x,y)^{-1} s_{\tilde x \tilde y}$$
This shows that $\alpha$ is the K\"ulshammer--Puig class of this block
at $\hat Q$.  Note that by the first statement we have 
$k\widetilde L e\cong$ $k_\alpha L$. The last statement on AWC follows 
from the fact that if $P \leq S$ is $\F_S(L)$-centric radical, then $P$ 
contains $Q$ and if $Q \leq P \leq S$, then 
$N_{L/Q} (P/Q)/(P/Q))\cong N_L(P)/P   =$ $\Out_{\F_S(L)}(P)$.  
\end{proof} 

\begin{Lem}\label{l:centext2}
Let $G$ be a finite group with normal subgroup $N$. Fix a cohomology 
class $\alpha \in H^2(G,k^\times)$ and write also $\alpha$ for the 
restriction to $N$. If $z(k_{\alpha}G) \neq 0$, then 
$z(k_{\alpha}N) \neq 0$. 
\end{Lem}

\begin{proof}
Using Proposition~\ref{p:centext}, we fix a $p'$-central extension 
$1 \to Z \to \widehat{G} \to G \to 1$ corresponding to $\alpha$ and a 
central idempotent $e \in kZ$ such that $k_{\alpha}G \cong k\widehat{G}e$.  
Then the restriction $\alpha$ is the class corresponding to the induced 
central extension $\hat{N}$ of $N$, and $k_{\alpha}N \cong k\hat{N}e$. 
Assume now that $k_{\alpha}G$ has a projective simple module.  Then
$k\widehat{G}e$, and hence $k\widehat{G}$, has a projective simple 
module, say $M$. The restriction of $M$ to $\hat{N}$ is both projective 
and semisimple. Hence, any simple summand of $\Res_{\hat{N}}^{\hat{G}} M$ 
is projective. Since $e$ still acts as the identity on the restriction 
of $M$, we see that $k\hat{N}e$ has a projective simple module, and hence 
so does $k_{\alpha}N$. 
\end{proof}

\begin{Lem}\label{l:z=0}
Let $G$ be a finite group and $\alpha \in H^2(G,k^\times)$. If 
$O_p(G) \neq 1$, then $z(k_{\alpha}G) = 0$.
\end{Lem}

\begin{proof}
As in the proof of Lemma~\ref{l:centext2}, let $1 \to Z \to \widehat{G} \to G
\to 1$ be a $p'$-central extension of $G$ determined by $\alpha$, and let $e
\in kZ$ be a central idempotent in $k\widehat{G}$ such that $k_\alpha G \cong
k\widehat{G}e$. Let $P = O_p(G)$ and $\widehat{P}$ be the preimage under the
quotient map.  Since $Z$ is a $p'$-group, the restriction of $\alpha$ to $P$ is
trivial, and so $\widehat{P} = Z \times P_0$ with $P_0$ mapping isomorphically
to $P$. Then $O_p(\widehat{P}) = P_0 \neq 1$ is a normal $p$-subgroup of
$\widehat{G}$. Thus, as $kP$ has no projective simple module, neither does
$k\hat{G}$ by Lemma~\ref{l:centext2}.  Hence neither does $k\hat{G}e \cong
k_\alpha G$.
\end{proof}

The following calculation will be needed in Section \ref{exampleSection}.

\begin{Lem}\label{l:sl2p}
Let $G$ be a finite group with normal subgroup $N$. Suppose that $V$ is an
inertial projective simple $kN$-module and that $G/N$ is a cyclic $p'$-group.
Then $G$ has exactly $|G:N|$ projective simple modules lying over $V$. In
particular, if $q$ is a power of $p$ and $N = \SL_n(q) \le G \le \GL_n(q)$,
then $z(kG)=|G:\SL_n(q)|$.
\end{Lem}

\begin{proof}
Since $G/N$ is a cyclic group, we have that $H^2(G/N,k^\times)$ is trivial.
Thus, the hypotheses of Corollary 5.3.13 of \cite{LinckelmannBT} hold. By that
result and its proof, we may fix a simple $kG$-module $U$ with $\Res_N^G(U)
\cong V$, and then the collection of isomorphism classes of one dimensional
$kG/N$-modules is in one-to-one correspondence with the collection of simple
$kG$-modules restricting to $V$ via the map $W \mapsto U \otimes_k W$, where we
regard $W$ as a module for $G$ by inflation. Also, the $U \otimes_k W$ are
precisely the summands of the induced module $\Ind_N^G V$, and hence are all
projective. This completes the proof of the first statement. 

In the special case of the last statement, by results of  Steinberg  (see  \cite[Theorem~3.7, Theorem~8.3] {Humphreys2006})  $N = \SL_n(p)$ has exactly one
projective simple module, the Steinberg module, which is therefore inertial.
Since $G/N$ is a cyclic $p'$-group in this case, and since any projective
$kG$-module is projective also as a $kN$-module, the last statement follows.
\end{proof}

Fix a finite group $G$ and an abelian group $A$. Let $\P$ be the set of
all chains of proper inclusions  
\[
Q_0=1 < Q_1 \cdots <  Q_m
\]
of $p$-subgroups of $G$. This is a $G$-set with respect to the 
conjugation action of $G$ on chains, and we denote by $N_G(\sigma)$ the 
stabilizer of $\sigma$ in $G$. 
Let $\N$ be the subset of all such chains satisfying in addition 
$Q_i \unlhd Q_m$ for each $0 \le i \le m$. 
Let $\E$ be the set of chains in $\N$ consisting of elementary 
abelian subgroups. Both $\N$ and $\E$ are $G$-subsets of $\P$. For the 
purpose of calculating alternating sums indexed by chains, we can pass 
between $\P$, $\N$, and $\E$: 

\begin{Lem}[{\cite[Proposition~3.3]{KnorrRobinson1989}}] 
\label{l:rad2}
Let $G$, $A$, $\P$, $\N$, and $\E$ be as above. Let $\f$ be a function 
from the set of subgroups of $G$ to $A$ such that $\f$ is constant on 
conjugacy classes  of subgroups of $G$. Then  
$$
\sum_{\sigma \in \P/G} (-1)^{|\sigma|} \f(N_G(\sigma))=
\sum_{\sigma \in \N/G} (-1)^{|\sigma|} \f(N_G(\sigma))=
\sum_{\sigma \in \E/G} (-1)^{|\sigma|} \f(N_G(\sigma)).
$$
\end{Lem}

We shall need the following well-known Lemma in Section
\ref{s:towards}.

\begin{Lem}[{\cite[Lemma~2.1]{Thevenaz1992}, 
\cite[Proposition~3.3]{KnorrRobinson1989}}] 
\label{l:rad}
Let $G$, $A$, and $\N$ be as above and let $\f$ be a function from the set
of subgroups of $G$ to $A$ such that $\f$ is constant on conjugacy classes of
subgroups of $G$.  If $O_p(G) \ne 1$, then 
\[
\sum_{\sigma \in \N/G} (-1)^{|\sigma|} \f(N_G(\sigma))=0.
\]
\end{Lem}

\begin{proof}
We sketch the proof for the convenience of the reader.
Set $R:=O_p(G)$ and assume that $R > 1$.  We show that there exists a
$G$-invariant involution $\eta\colon \N \to \N$ where $N_G(\sigma) =
N_G(\eta(\sigma))$ and $|\eta(\sigma)| = |\sigma| \pm 1$.  Given $\sigma =
(Q_0 < Q_1 < \cdots < Q_m) \in \N$, choose $i$ maximal with the property that
$R \nleq Q_{i}$. Since $R \nleq 1 = Q_0$, we see that there is such an $i$. By
choice of $i$, we have $Q_i < Q_iR$, and we have $Q_iR \leq Q_{i+1}$ if $i <
m$.  Define
\[
\eta(\sigma) = 
\begin{cases} 
Q_0 < \cdots < Q_m < Q_mR & \mbox{ if $i=m$},\\
Q_0 < \cdots < Q_i < Q_{i+2} < \cdots < Q_m 
& \mbox{ if $Q_iR = Q_{i+1}$, and} \\ 
Q_0 < \cdots Q_i < Q_iR < Q_{i+1} < \cdots < Q_m 
& \mbox{ if $Q_iR < Q_{i+1}$}.
\end{cases}
\]
Then $\eta(\sigma) \in \N$ and $N_G(\sigma) = N_G(\eta(\sigma))$ for each
$\sigma \in \N$, since $R$ is a normal $p$-subgroup of $G$. Also,
$|\eta(\sigma)| = |\sigma| \pm 1$. It is a momentary exercise to verify that
$\eta$ is an involution on $\N$. Hence, the alternating sum vanishes as
claimed.  
\end{proof} 

\begin{Rem} \label{reorderRemark}
We finish this section  with a mention of a recurrent elementary tool 
for reordering sums indexed by two or more sets acted upon by a finite 
group $G$, which we will use  without  much  further comment. 
Let    $X$, $Y$ be  finite $G$-sets and  denote by 
$ \pi_X:  X \times Y   \to  X$, $ \pi_Y : Y \times X$ the projection 
maps. Let $A$ be a $G$-invariant subset of  $ X\times Y $ under the diagonal action of  $G$   on
$X\times Y$.   Suppose that for any $(x,y)\in$ $X\times Y$ we have an element 
$\alpha(x,y)$ in some abelian group depending only on the $G$-orbit of 
$(x,y)$.    Then 
$$\sum_{(x,y)\in   A/G}\ \ \alpha(x,y)$$ 
is equal to  any of the following double sums
$$\sum_{x \in X/G}\ \ \sum_{y\in \pi_Y(\pi_X^{-1}(x) \cap A)/G_x}\ 
\ \alpha(x,y)$$
$$\sum_{y \in Y/G}\ \ \sum_{x\in \pi_X(\pi_Y^{-1}(y) \cap A )/G_y}\ 
\ \alpha(x,y).$$

Note that the two double sums make sense as by the $G$-invariance  of $A$, 
for each  $x \in X$,  $  \pi_Y(  \pi_X^{-1}(x) \cap A) $  is 
$G_x$-invariant and for each $y\in Y$, $\pi_X(\pi_Y^{-1}(y)  \cap A) $  
is $G_y$-invariant. Let  ${\mathcal X} $ be a set of representatives   
of the $G$-orbits of $X$ and for each $x \in {\mathcal  X}$, let 
$\mathcal{Y}_x$ be  a   set of  representatives of the $G_x $-orbits of 
$ \mathcal {X} $  and  set 
$$ U:= \{ (x, y) \, : \,  x \in {\mathcal X}, y \in  {\mathcal Y}_x \} . $$
Then,  $  U   \subseteq A$. We will show that  $ U   $ 
is a set of representatives of the $G$-orbits of $A$, and this will yield 
the equality of $\sum_{(x,y)\in   A/G}\ \ \alpha(x,y)$  with the  first 
double sum. Suppose that $ x, x' \in \mathcal{X}$ , 
$ y, y' \in \mathcal {Y}_x$ are    such that $ (x, y) $ and $ (x', y' )$ are 
in the same $G$-orbit  and let $ g\in G$ be such that 
$ (x', y')= \, ^g(x, y)  $.  By comparing the first components, it follows 
that $ x'$ and $x$ are in the same $G$-orbit of $ X$, hence $x'=x $ and 
$ g\in G_x $.  Now comparing the second components  implies $y'=y $.      
Conversely,   let $ (x_0, y_0 ) \in A$. We will show that $(x_0, y_0)$ 
is $ G$-conjugate to an element of $U$. By definition of ${\mathcal X} $,    
there exists $g \in G$ and $ x \in {\mathcal X}$ such that  
$ x_0 = \, ^g x $,  hence by replacing $ (x_0,y_0) $ by     
$\,^g (x_0, y_0 )  $ we may assume that $ x _0 \in \mathcal {X}$.     
Since $(x_0, y_0) \in  A$, $ y_0 \in \pi_Y(\pi_X^{-1} (x_0) \cap A) $. 
Hence by the  definition of $\mathcal{ Y}_{x_0} $,   $ y_0  $ is 
$G_{x_0} $-conjugate to some element  of ${\mathcal Y}_{x_0} $, say 
$z_0  = \, ^h y_0$ with $h \in G_{x_0} $, $z_0 \in \mathcal{Y}_{x_0}$. 
Then   
$$ \, ^ h (  x_0, y_0)   =  (\,^ h x_0,  \, ^h y_0 )  =  
( x_0, z_0 )  \in U $$ 
as required.  The proof  of the equality with the second sum is entirely 
analogous.
\end{Rem}

\section{Towards Theorem \ref{t:main2}}\label{s:towards}
\label{towardsmainSection}

Throughout this section let $\F$ be a saturated fusion system on
a finite $p$-group $S$, and let $\alpha$ be an $\F$-compatible family.

Our first goal will be to reformulate $\m^{*}(\F,\alpha)$ by
reindexing the sum over objects in the full subcategory
$S_\triangleleft(\F^c)$ of the subdivision category of the
category of $\F$-centric subgroups. Recall from  Section~\ref{prelimSection}
that $S_\triangleleft(\F^c)$ 
has as objects chains of proper inclusions 
\[
Q_0 < Q_1 < \cdots < Q_m
\]
of $\F$-centric subgroups with the property that the $Q_i$ are normal in the
maximal term $Q_m$, for each $0 \le i \le m$.
Consider the following sets
\[
\begin{array}{rcl} 
\M &:=& \{(Q,\sigma,[x]) \mid Q \in \F^c, \sigma \in \N_Q, [x] \in Q^{\cl}\}, 
\medskip  \\
\widetilde{\M} &:= & \{(\sigma,x) \mid \sigma \in 
S_\triangleleft(\F^c), x \in Q_\sigma\}.  
\end{array}
\]
The set $\M$ is equipped with the equivalence relation
$$(Q,\sigma,[x]) \sim_{\M} (R,\tau,[y])$$
whenever there exists an isomorphism $\varphi : Q\to$ $R$ in $\F$ 
such that $\bar{c}_\phi(\sigma) = \tau$ and such that 
$\varphi([x])=[y]$.   Here   $\bar {c}_\phi $ is  as  defined before Lemma \ref{l:chain}  and we use 
$\bar {c}_\phi(\sigma) $ to denote the image of $\sigma $  under the natural extension of $\bar{c}_\phi $ to a map  from the set of chains of subgroups  of $\Out_{\F}(Q)  $ to  the set of chains of subgroups of   $\Out_{\F}(R)  $.The set $\widetilde{\M}$ is equipped with the
equivalence relation 
$$(\sigma,x) \sim_{\widetilde{\M}} (\tau, y) $$
whenever there exists an isomorphism $\varphi: \sigma \to \tau$ in
$S_\triangleleft(\F^c)$ such that   $\varphi(x)=y$.

\begin{Prop}\label{p:rewritez}
We have
\[
\m^*(\F,\alpha)=\sum_{(\sigma,x) \in \widetilde{\M} / \sim} (-1)^{|\sigma|} z
(k_{\alpha}C_{\Aut_{\F} (\sigma)} (x)\Aut_{Q_\sigma} (Q^\sigma)/\Aut_{Q_\sigma}
(Q^\sigma)). 
\]
\end{Prop}

\begin{proof}
This follows from Lemmas~\ref{l:rewritem} and \ref{l:rewritez} below.
\end{proof}

We   rewrite    $\m^*(\F,\alpha)$  in terms of $(\M,\sim)$.

\begin{Lem}\label{l:mfa-onesum}
\[
\m^*(\F,\alpha) = 
\sum_{(Q,\sigma,[x]) \in \M/\sim} z(k_\alpha C_{I(\sigma)}([x])). 
\]
\end{Lem}
\begin{proof}   Let  ${\mathcal X} $ be a    set of representatives of $\F$-classes  in $\F^c$  and for each $ Q \in {\mathcal X} $, let ${\mathcal Y}_Q$ be a set of $\Out_{\F} (Q)$ representatives    of
 $ {\mathcal W}_Q ^*$.    Then  
$\{ (Q,  \sigma, [x])  :  Q \in {\mathcal   X},  (\sigma, [x] ) \in {\mathcal  Y}_Q   \} $  is a  set of representatives  of the $\sim $-equivalence classes of $ \M$  and  the   result follows.
\end{proof}

A normal chain $\sigma = (Q_0 < \cdots < Q_m)$ in 
$S_{\triangleleft}(\F^c)$ induces a normal chain 
$\Aut_\sigma(Q_0) := (\Aut_{Q_0}(Q_0) < \cdots < \Aut_{Q_m}(Q_0))$
of $p$-subgroups in $\Aut_\F(Q)$, and a corresponding normal chain
$\Out_\sigma(Q_0) \in \N_{Q_0}$ upon factoring by $\Inn(Q_0)$.  In this
context, bars will denote quotients by $\Inn(Q_0)$.  That is, we set $\bar Q_i
:= \Out_{Q_i}(Q_0)$ for each $0 \le i \le m$ and we set
\[
\bar \sigma := \Out_\sigma(Q_0) = (\bar Q_0 < \bar Q_1 < \cdots < \bar Q_m)
\]
for short. Note that $\bar Q_0$ is trivial. 

\begin{Lem}\label{l:mvstildem}
The map $\widetilde{\M} \longrightarrow \M$ which sends $(\sigma, x)$ to
$(Q_\sigma, \bar \sigma, [x])$ induces a bijection between
$\widetilde{\M}/\!\!\sim$ and $\M/\!\!\sim$. 
\end{Lem}

\begin{proof} 
We first show that the map is well-defined.  Let $(\sigma, x) \sim (\tau, y)$
in $\widetilde {\M}$. Fix an isomorphism $\varphi\colon \sigma \rightarrow
\tau$ in $S_{\triangleleft}(\F^c)$ such that $\varphi(x) = y$. Then $(Q_\sigma,
\bar \sigma, [x])  \sim (Q_\tau, \bar \tau, [y]) $ via  the restriction of
$\varphi $ to $ Q_\sigma $. 

Next, suppose $(Q_\sigma, \bar \sigma, [x]), (Q_\tau, \bar \tau, [y]) \in \M$
are $\M$-equivalent.  Let $\psi\colon Q_\sigma \to Q_\tau$ be an
$\F$-isomorphism such that $\bar{c}_\psi(\bar{\sigma}) = \bar \tau $ and
$\psi([x])=[y]$. By Lemma~\ref{l:chain}, $\psi$ extends to a chain isomorphism
$\hat \psi\colon \sigma \to \tau$. Since $\psi([x]) = [y]$, we have $\psi (x)
=uyu^{-1}$ for some $u \in Q_\tau$. Let $\delta\colon Q^\sigma \to Q^\tau$ be
the composition of $\hat \psi$ with conjugation by $u$. Then $(\sigma, x)$ and
$(\tau,y)$ are $\widetilde{\M}$-equivalent via $\delta$. This proves
injectivity.

It remains to show that whenever $(R, \rho, [z]) \in \M$, there exists
$(\sigma,x) \in \widetilde{\M}$ such that $(Q_\sigma, \bar{\sigma},[x])$ is
$\M$-equivalent to $(R, \rho, [z])$. Let $\rho = (1 < X_1 < \cdots < X_m) \in
\N_R$.  Let $\alpha\colon R \to R'$ be an $\F$-isomorphism with $R'$ fully
$\F$-normalised, and consider the chain
\[
\bar{c}_\alpha(\rho) = (1 < \bar{c}_\alpha(X_1) < \cdots < \bar{c}_\alpha(X_m)).
\]
Since $R'$ is fully $\F$-normalised and $\F$ is saturated, $\Out_S(R')$ is a
Sylow $p$-subgroup of $\Out_\F(R')$, so by Sylow's theorem we may fix $\beta
\in \Out_{\F}(R')$ such that $\beta\bar{c}_\alpha(X_m)\beta^{-1} \leq
\Out_S(R')$. Denote by $R_i'$ the inverse image of $\beta\bar{c}_\alpha(X_i)
\beta^{-1}$ in $N_S(R')$, and set
\[
\sigma := (R' < R_1' < \cdots < R_m') \mbox {\quad and \quad} x:= \hat{\beta}\alpha(z),
\]
where $\hat{\beta} \in \Aut_\F(R')$ is any lift of $\beta$.  Then
$(\sigma, x) \in \widetilde{\M}$, and $(Q_\sigma, \bar{\sigma},[x])$ is
$\M$-equivalent to $(R, \rho, [z])$ via $\hat{\beta}\alpha$.
\end{proof}

The following lemma is now immediate from Lemmas~\ref{l:mfa-onesum} and
\ref{l:mvstildem}.

\begin{Lem}\label{l:rewritem}
We have
\[
\m^*(\F,\alpha)=\sum_{(\sigma,x) \in \widetilde{\M} / \sim} (-1)^{|\sigma|}
z(k_{\alpha}C_{I(\bar{\sigma})}([x])).
\]
\end{Lem}

To complete the proof of Proposition~\ref{p:rewritez}, we give an
interpretation of  $z(k_{\alpha} C_{I(\bar \sigma)}([x])$ in terms of the
automisers of chains in $S_\triangleleft(\F^c)$.  

\begin{Lem}\label{l:rewritez}
Fix $\sigma = (Q_0 < \cdots < Q_m) \in S_\triangleleft(\F^c)$, and let 
$\pi$ be the composite
\[
\Aut_\F(\sigma) \xrightarrow{\quad\res\quad} N_{\Aut_\F(Q_0)}(\Aut_\sigma(Q_0)) \xrightarrow{\quad\quad}  I(\bar \sigma)
\overset{\opdef}{=} N_{\Out_\F(Q_0)}(\Out_\sigma(Q_0)),
\]
which restricts to $Q_0$ and then factors by $\Aut_{Q_0}(Q_0)$. Then
\begin{enumerate}
\item $\pi$ is surjective,
\item $\ker(\pi) = \Aut_{Q_{0}}(Q_{m})$, and
\item for each $x \in Q_0$, the group $ C_{\Aut_{\F}(\sigma)}
(x)\Aut_{Q_0}(Q_m)$ is the inverse image of $C_{I(\bar{\sigma})}(x)$
under $\pi$.  
\end{enumerate}
\end{Lem}
\begin{proof}  
To prove (1), it suffices to show that the restriction map $\res\colon
\Aut_\F(\sigma) \to N_{\Aut_{\F}(Q_0)}(\Aut_\sigma(Q_0))$ is surjective. Let
$\alpha \in N_{\Aut_\F(Q_0)}(\Aut_\sigma(Q_0))$. Then
$c_\alpha(\Aut_{Q_i}(Q_0)) \leq \Aut_{Q_i}(Q_0)$ for all $0 \leq i \leq m$.
The first conclusion of Lemma~\ref{l:chain} then yields an extension
$\widetilde{\alpha}$ of $\alpha$ to $Q_m$. 

Fix $i$ with $0 \leq i \leq m$, and fix $u \in Q_i$. Then since
$\widetilde{\alpha}$ is defined on $u$, we have
\[
c_{\widetilde\alpha(u)}|_{Q_0} = 
\alpha (c_u|_{Q_0}) \alpha^{-1} \in \Aut_{Q_i}(Q_0)
\]
by assumption. Hence, $\widetilde{\alpha}(u)$ lies in the full inverse image of
$\Aut_{Q_i}(Q_0)$ under $N_S(Q_0) \to \Aut_S(Q_0)$, which is $Q_i$ because
$Q_0$ is centric.  This shows that $\widetilde{\alpha}(Q_i) = Q_i$ for each $i$,
and thus the surjectivity of the restriction map.

That $\Aut_{Q_0}(Q_m) \leq \ker(\pi)$ is clear. To see the other inclusion in
(2), fix $\phi \in \ker(\pi)$. Then $\phi|_{Q_0} = c_u$ for some $u \in Q_0$,
so we may fix $z \in Z(Q_0)$ such that $\phi = c_{u}c_z = c_{uz}$ by
\cite[Lemma~A.8]{BrotoLeviOliver2003}. Thus, $\phi \in \Aut_{Q_0}(Q_m)$, as
desired.

Finally, (3) holds because $\ker(\pi) = \Aut_{Q_0}(Q_m)$ acts transitively on
the $Q_0$-class $[x]$. 
\end{proof} 

Define the following subsets of $\widetilde{\M}:$

\begin{enumerate}
\item $\widetilde{\M}^{e}$ is the subset of $\widetilde{\M}$ consisting of
those $(\sigma, x)$ for which $Q^\sigma/Q_\sigma$ is elementary abelian.
\item  $\widetilde{\M}^{\circ}$  is  the subset of $\widetilde{\M}$
consisting of those $(\sigma, x)$ for which $C_{Q^\sigma} (x) \leq Q_\sigma$.
\item $\widetilde{\M}^{e, \circ}$ is the intersection of $\widetilde{\M}^{e}$
and $\widetilde{\M}^{\circ}$.
\item $\widetilde{\M}^{e, \circ,c} $ is the subset of $\widetilde{\M}^{e,
\circ}$ consisting of those $(\sigma, x)$ for which $ C_{Q^\sigma}(x)
\Phi(Q^\sigma)$ is $\F$-centric.
\end{enumerate}

Observe that all these subsets are unions of $\widetilde{\M}$-equivalence
classes. Let
\[
\m^e(\F, \alpha) := \sum_{(\sigma,x) \in \widetilde{\M}^e / \sim}
(-1)^{|\sigma|}
z(k_{\alpha}C_{\Aut_{\F}(\sigma)}(x)\Aut_{Q_\sigma}(Q^\sigma)/\Aut_{Q_\sigma}(Q^\sigma))
\]
and define $\m^{\circ}(\F, \alpha)$, $\m^{e,\circ}(\F, \alpha)$, and
$\m^{e,\circ,c}(\F,\alpha)$ analogously.

\begin{Prop}\label{p:cancel}
The following hold.
\begin{enumerate} 
\item $\m^*(\F,\alpha) = \m^e(\F, \alpha)$.
\item $\m^*(\F,\alpha) = \m^{\circ}(\F,\alpha)$.
\item $\m^*(\F,\alpha) = \m^{e,\circ}(\F,\alpha)$.
\end{enumerate} 
\end{Prop} 
\begin{proof} 
By    Lemma~\ref{l:rewritez}, Remark~\ref{reorderRemark},  the  obvious analogue  of  Lemma~\ref{l:mfa-onesum}   for elementary abelian chains,    and by   restricting the inverse of the  bijection of  Lemma~\ref{l:mvstildem} to classes of    elements  of $\widetilde{\M}^{e}$, we  have
\[
\m^e(\F,\alpha)=\sum_{Q \in \F^{c}} \sum_{\sigma \in \E_Q / \Out_\F(Q)}
(-1)^{|\sigma|} \sum_{[x] \in Q^{\cl}/I(\sigma)} z(k_\alpha
C_{I(\sigma)}([x])),
\]
where $\E_Q \subseteq \N_Q$ is the set of all elementary abelian chains. Thus
(1) follows on applying Lemma \ref{l:rad2} with $G=\Out_\F(Q)$ for each $Q
\in \F^{c}$. We next prove (2). Note that if $(\sigma, x) \in \widetilde{\M}$
and $C_{Q^\sigma}(x)  $ is not contained in $Q_\sigma$, then
$C_{\Inn(Q^\sigma)}(x)\Aut_{Q_\sigma}(Q^\sigma)/\Aut_{Q_\sigma} (Q^\sigma)
\cong C_{Q^\sigma}(x)/C_{Q_\sigma}(x)$ is a non-trivial normal subgroup of
$C_{\Aut_{\F}(\sigma)}(x)\Aut_{Q_\sigma}(Q^\sigma)/\Aut_{Q_\sigma}(Q^\sigma)$
and the result  follows from Proposition \ref{p:rewritez}(3) and
Lemma~\ref{l:z=0}. The same argument holds with $(\sigma, x) \in
\widetilde{\M}^e$, so (3) follows from (1).
\end{proof}

Recall that
\begin{equation}\label{e:k2}
\k(\F,\alpha)=
\sum_{x \in [S/\F]} \sum_{Q \in C_\F(x)^c /C_\F(x)} 
z(k_\alpha \Out_{C_\F(x)}(Q)),
\end{equation}
where $[S/\F] \subseteq S$ is a fixed set of fully $\F$-centralized
$\F$-conjugacy class representatives of the elements of $S$. Define
\[
\begin{array}{rcl}
\C &:= & \{ (Q,x) \mid x \in [S/\F], Q \in C_\F(x)^c\}, \text{ and } \medskip\\
\D &:= & \{ (Q,x) \mid x \in Z(Q), Q \in \F^c\} \\
\end{array}
\]
and equivalence relations
\[
\begin{array}{rcl}
(Q,x) \sim_\C (R,y) & \Longleftrightarrow & x=y \mbox{ and } 
\Iso_{C_\F(x)}(Q,R) \neq \varnothing, \text{ and } \medskip \\
(Q,x) \sim_\D (R,y) & \Longleftrightarrow & \mbox{there exists } 
\varphi \in \Iso_\F(Q,R) \mbox{ such that } \varphi(x)=y. \\
\end{array}
\]

Thus, $\C/\!\!\sim$ may be viewed as an indexing set for 
$\k(\F,\alpha)$. Also, $x \in Z(Q)$ whenever $Q \in C_\F(x)^c$, so that 
$\C$ is a subset of $\D$. 

\begin{Lem}\label{l:CcongD}
The inclusion $\C \hookrightarrow \D$ induces a bijection between
$\C/\!\!\sim \,$ and $\D/\!\!\sim$; in particular,  
\begin{equation}\label{e:k3}
\k(\F,\alpha) = \sum_{Q \in \F^c / \F} \sum_{x \in Z(Q)/ \Out_\F(Q)} 
z(k_\alpha C_{\Out_\F(Q)}(x)).
\end{equation}
\end{Lem}

\begin{proof}
If $(Q,x) \sim_\C (R,y)$, then $x=y$ and there is an $\F$-isomorphism from $Q$
to $R$ which centralizes $x$, so that $(Q,x) \sim_\D (R,y)$. There is indeed a
well-defined map on equivalence classes induced by the inclusion. 

Conversely, assume that $(Q,x), (R,y) \in \C$ are $\D$-equivalent. Fix an
$\F$-isomorphism $\phi$ from $Q$ to $R$ with $\phi(x) = y$. As $x, y \in
[S/\F]$ are $\F$-conjugate, we have $x = y$, and so $Q$ and $R$ are
$C_\F(x)$-conjugate. This shows that $(Q,x) \sim_\C (R,y)$, so the induced
map is injective.

To complete the proof of the first assertion, it remains to show that each
element of $\D$ is $\D$-equivalent to a member of $\C$.  Fix $(R,y) \in \D$.
Let $x \in [S/\F]$ be the unique element which is $\F$-conjugate to $y$. Since
$\gen{x}$ is fully $\F$-centralized, we may choose a morphism $\alpha \in
\Hom_\F(C_S(\gen{y}),C_S(\gen{x}))$ such that $\alpha(y) = x$ by
\cite[I.2.6(c)]{AschbacherKessarOliver2011}. Set $Q = \alpha(R)$.  Then $(R,y)
\sim_{\D} (Q,x)$ via $\alpha$. Since $R$ is $\F$-centric, also $Q$ is
$\F$-centric, so that $Q$ is $C_\F(x)$-centric by Lemma~\ref{l:cfxcent}.  This
yields $(Q,x) \in \C$ and completes the proof of the first assertion.

Now $\Out_{C_{\F}(x)}(Q) = C_{\Out_\F(Q)}(x)$ for each $x \in Z(Q)$ by
Lemma~\ref{l:cfxcent}.  Hence, as $\C/\!\!\sim$ is an indexing set for a single
sum computing $\k(\F, \alpha)$ as in \eqref{e:k2}, and as $\D/\!\!\sim$ is an
indexing set for a single sum computing the right hand side of \eqref{e:k3}, we
have that \eqref{e:k3} follows from \eqref{e:k2}.
\end{proof}

\begin{Prop}\label{p:kmeoc}
We have, $\k(\F,\alpha)=\m^{e,\circ,c}(\F,\alpha)$. 
\end{Prop}

\begin{proof}
Let $\D'$ be the subset of $\widetilde{\M}^{e,\circ,c}$ consisting of
the pairs $(\sigma, x)$ such that $|\sigma| = 0$ and $x \in Z(Q_\sigma)$.
Then $\D'$ is a union of $\widetilde{\M}$-equivalence classes. Regarding an
$\F$-centric subgroup $Q$ as a chain of length zero yields a canonical
bijection $\D/\!\!\sim_{\D} \,\,\to \D'/\!\!\sim_{\widetilde{\M}}$, and so we
may regard $\k(\F,\alpha)$ as indexed over $\D'/\!\!\sim_{\widetilde{M}}$. We
use chain pairing to remove the terms from $\m^{e,\circ,c}(\F,\alpha)$ not in
$\D'$.  This will yield
\[
\m^{e,\circ,c}(\F,\alpha) = \sum_{(\sigma,x) \in \D' / \sim}
(-1)^{|\sigma|} z (k_{\alpha}C_{\Aut_{\F} (\sigma)} (x)\Aut_{Q_\sigma}
(Q^\sigma)/\Aut_{Q_\sigma} (Q^\sigma)).
\]
The Proposition then follows from the expression for $\k(\F,\alpha)$ in
Lemma~\ref{l:CcongD}, along with Lemma~\ref{l:rewritez}(3).

For each $\sigma  = (Q_0 < \cdots < Q_m) \in S_{\triangleleft}(\F^c)$, we let
$Q_{-1} := C_{Q_m}(x)\Phi(Q_m)$.  
Define a map 
\[
\eta\colon \widetilde{\M}^{e,\circ,c}\backslash \D'
\longrightarrow \widetilde{\M}^{e,\circ,c}\backslash \D'
\]
via 
\[
(\sigma, x) \longmapsto (\eta(\sigma),x),
\]
where 
\begin{align*}
\eta(\sigma) =
\begin{cases}
Q_{-1} < Q_0 < \cdots < Q_m & \text{ if $Q_{-1} < Q_0$, and}\\
           Q_1 < \cdots < Q_m & \text{ if $Q_{-1} = Q_0$.}
\end{cases}
\end{align*}
It is straightforward to see that $\eta$ is an involution that preserves
$\widetilde{\M}$-equivalence classes if well-defined. 

To prove that $\eta$ is well-defined, we assert three points for a given pair
$(\sigma,x) \notin \D'$ with $\sigma$ as above.  First, observe that
$\eta(\sigma)$ is a first component of some member of
$\widetilde{\M}^{e,\circ,c}$ by definition of $Q_{-1}$ and the fact that
$Q_{-1} \in \F^c$ by assumption. In particular, $\eta(\sigma)$ is never the
empty chain: if $\sigma$ has length zero, then $C_{Q_0}(x) = C_{Q_m}(x) < Q_0$
as $(\sigma, x) \notin \D'$, so also $Q_{-1} = C_{Q_0}(x)\Phi(Q_0) < Q_0$,
and hence $\eta(\sigma)$ has length $1$.  Second, note that $x \in C_{Q_m}(x)
\leq Q_{-1}$ in case $Q_{-1}$ is contained properly in $Q_0$, so that indeed
$(\eta(\sigma),x) \in \widetilde{\M}^{e,\circ,c}$. Lastly, continue to
consider a pair $(\sigma,x)$ not in $\D'$.  We claim that $(\eta(\sigma),x)$ is
not in $\D'$, and the only case where this is not immediate has $|\sigma| = 1$
and $|\eta(\sigma)| = 0$.  In this case either $x$ is not in $Z(Q_0)$, in which
case $x$ is likewise not in $Z(Q_1) \leq Z(Q_0)$, or $x \in Z(Q_0)$, in which
case $C_{Q_1}(x) = C_{Q_0}(x) = Q_0 < Q_1$ so that again $x$ is not in
$Z(Q_1)$. This shows that $(\eta(\sigma),x) \notin \D'$ and completes the proof
of the last point.

Having shown that $\eta$ is a well-defined involution, it remains to prove that
it preserves the value of each summand appearing in
Proposition~\ref{p:rewritez}. To establish this, it suffices to show that
\[
C_{\Aut_\F(\sigma)}(x)=C_{\Aut_\F(\eta(\sigma))}(x) \mbox{\quad and \quad }
\Aut_{C_{Q_0}(x)}(Q_m) =  \Aut_{C_{Q_{-1}}(x)}(Q_m).
\]
As $Q_{-1}$ is invariant under $\Aut_\F(\sigma)$, one has
$\Aut_{\F}(\sigma) \leq \Aut_{\F}(\eta(\sigma))$ if $\eta(\sigma)$ has length
one more than $\sigma$. Also, one has the same containment if $\eta(\sigma)$
has length one less, since $\eta(\sigma)$ is a subchain of $\sigma$ in that
case. Equality therefore holds in both cases, because $\eta$ is an involution.
This completes the proof of the first displayed equality. Finally, the second
equality holds since $C_{Q_0}(x) = C_{Q_{-1}}(x)$ for each $(\sigma,x) \in
\widetilde{\M}^{e,\circ,c}$.
\end{proof}

\section{Proof of Theorem \ref{t:main2}}
\label{main2Section} 

In light of Propositions \ref{p:kmeoc} and \ref{p:cancel}(3), to 
complete the proof of Theorem \ref{t:main2} it  suffices  to establish  
an equality between $\m^{e,\circ}(\F,\alpha)$ and 
$\m^{e,\circ,c}(\F,\alpha)$. We will achieve that in this section. 

If $G$ is a finite group and $\sigma$ is a chain of $p$-subgroups 
in $G$ such that the first subgroup is a normal subgroup of the 
last subgroup, then we denote by $G_{\sigma}\leq N_G(Q^{\sigma})$ 
the stabiliser in $G$ of the chain and by $\Aut_G(\sigma)$ the 
image of $G_{\sigma}$ in $N_G(Q^{\sigma})/C_G(Q^{\sigma})$. 

\begin{Lem}\label{l:model}
Let $\sigma= (Q_0 < \cdots < Q_m)$ be a chain of proper inclusions of 
subgroups of $S$ such that $Q_i $ is normal in $Q_m$ for each 
$0 \leq i \leq m$, and let $x \in Q_0$ be such that 
$C_{Q_m}(x) \leq Q_0$. Suppose that $Q_m$ is $\F$-centric, $Q_0$ is 
fully $\F$-normalised, and $Q_m$ is fully $N_{\F}(Q_0)$-normalised. 
Let $L$ be a model of $N_{N_{\F} (Q_0)}(Q_m)$. The following hold:
\begin{enumerate}
\item 
$C_{L_{\sigma}}(x)Q_{0}/Q_{0} \cong 
C_{\Aut_{\F} (\sigma)}(x)\Aut_{Q_0}(Q_m)/\Aut_{Q_0}(Q_m)$.
\item 
If $Q_0$ is not $\F$-centric, then
$z(k_{\alpha}(C_{L_{\sigma}}(x)Q_{0}/Q_{0}))=0$.
\end{enumerate}
\end{Lem}

\begin{proof}
Set $Q_{\sigma} = Q_0$ and $ Q^{\sigma}= Q_{m}$.  We have 
\[
L_{\sigma}/ Z(Q_m) \cong \Aut_L(\sigma)=
\Aut_{N_{N_{\F}(Q_{0})}(Q_{m})} (\sigma) = \Aut_{N_{\F}
(Q_{0})}(\sigma) = \Aut_{\F}(\sigma). 
\]
The quotient map $\pi\colon L_\sigma \to \Aut_{\F}(\sigma)$ sends
$C_{L_{\sigma}}(x)$ to $C_{\Aut_\F(\sigma)}(x)$. It also sends $Q_m$ 
to $\Aut_{Q_0}(Q_m)$, since $Z(Q_m) \le C_{Q_m}(x) \le Q_0$ by 
assumption. Part (1) follows from this.

We now turn to (2), where we first claim that $C_L(Q_0)$ is a $p$-group
under the given assumptions. Let $y$ be an element of $C_{L}(Q_0)$ of order
prime to $p$, and let $c_y$ be the image of $y$ in $\Aut(Q_m)$.  Since
$C_{Q_m}(\Aut_{Q_0}(Q_m)) = C_{Q_m}(Q_0) \leq C_{Q_m}(x) \leq Q_0$, we have
\[
[c_y,C_{Q_m}(\Aut_{Q_0}(Q_m))] \le [c_y,Q_0]=1.
\]
Now Lemma \ref{l:axb} implies that
$c_y = \Id_{Q_m}$, so that $y \in C_{L}(Q_m) \leq Q_m$ is of order a power of
$p$, since $Q_m$ is self-centralising in $L$. Hence, $y = 1$.

Assume that $z(k_{\alpha}(C_{L_{\sigma}}(x)Q_0/Q_0)) \ne 0$. As $Q_0$ is
normal in $L_{\sigma}$, we know that $C_{L_\sigma}(Q_0)Q_0$ is likewise normal
in $L_\sigma$. But $C_{L_\sigma}(Q_0) \leq C_{L_\sigma}(x)$, so
$C_{L_\sigma}(Q_0)Q_0$ is normal in $C_{L_\sigma}(x)Q_0$.
Hence, $z(k_{\alpha}(C_{L_{\sigma}}(Q_0)Q_0/Q_0)) \ne 0$ by Lemma
\ref{l:centext2}. It was just shown that $C_{L_\sigma}(Q_0)$ is a
$p$-group, so we have $C_{L_{\sigma}}(Q_{0}) \leq Q_0$ by Lemma \ref{l:z=0}. In
other words, $Q_0$ is $N_{\F}^K(Q_0)$-centric, where $K \leq \Aut_{\F}(Q_0)$ is
the subgroup consisting of those automorphisms which extend to automorphisms of
$\sigma$. Hence, $Q_0$ is $\F$-centric by Lemma~\ref{l:nfkcentric}. 
\end{proof}

\begin{Lem}\label{l:centric} 
Let $(\sigma, x) \in \widetilde M^{e, \circ}$, with $\sigma = (Q_0 < \cdots <
Q_m)$ as before. If $C_{Q_m}(x)\Phi(Q_m)$ is not $\F$-centric, then
$z(k_{\alpha}C_{\Aut_{\F}(\sigma)}(x)\Aut_{Q_0}(Q_m)/\Aut_{Q_{0}}(Q_m)) = 0$.   
\end{Lem} 

\begin{proof}  
Write $Q_{-1} = C_{Q_m}(x)\Phi(Q_m)$, and recall that $Q_{-1} \leq Q_0$ by
definition of $\widetilde{\M}^{e,\circ}$.  Using
\cite[I.2.6(c)]{AschbacherKessarOliver2011}, we choose a morphism $\phi \in
\Hom_\F(Q_m, S)$ with $\phi(R)$ fully $\F$-normalized, and then a morphism
$\psi \in \Hom_{N_\F(\phi(R))}(Q_m, N_S(\phi(R)))$ with $\psi\phi(Q_m)$ fully
$N_\F(\phi(R))$-normalized. Set $\tau = \psi\phi(\sigma)$ and $y =
\psi\phi(x)$. 
Conjugation by $\psi\phi$ yields an isomorphism
\[
C_{\Aut_{\F} (\sigma)} (x)\Aut_{Q_\sigma}
(Q^\sigma)/\Aut_{Q_\sigma} (Q^\sigma) \cong C_{\Aut_{\F} (\tau)}
(y)\Aut_{Q_\tau} (Q^\tau)/\Aut_{Q_\tau} (Q^\tau).
\]
Upon replacing $(\sigma,x)$ by $(\tau,y)$, we may therefore assume 
$Q_{-1}$ to be fully $\F$-normalized and $Q_m$ to be fully
$N_\F(Q_{-1})$-normalized.

Assume on the contrary that $Q_{-1} \myeq C_{Q_m}(x)\Phi(Q_m)$ is not
$\F$-centric, but that
$z(k_{\alpha}C_{\Aut_\F(\sigma)}(x)\Aut_{Q_0}(Q_m)/\Aut_{Q_0}(Q_m)) \neq 0$.
As $Q_0$ is $\F$-centric, 
$Q_{-1}$ is a proper subgroup of $Q_0$.  Consider the chain 
\[
\sigma' = (Q_{-1}  <  Q_0 < \cdots Q_m).
\]
It was shown in the last part of the proof of Lemma~\ref{p:kmeoc} that
\[
C_{\Aut_{\F}(\sigma)}(x)\Aut_{Q_0}(Q_m)/\Aut_{Q_0}(Q_m) =
C_{\Aut_{\F}(\sigma')}(x)\Aut_{Q_{-1}}(Q_m)/\Aut_{Q_{-1}}(Q_m),
\]
and that argument did not require $(\sigma,x) \in \widetilde{\M}^c$. But then
from Lemma~\ref{l:model} applied to $\sigma'$, we conclude that $Q_{-1}$ is
$\F$-centric after all, a contradiction.
\end{proof}

\begin{proof}[Proof of Theorem~\ref{t:main2}] 
By Proposition~\ref{p:rewritez}, Proposition~\ref{p:cancel}(3), and
Lemma~\ref{l:centric}, we have $\m^*(\F,\alpha)=\m^{e,\circ,c}(\F,\alpha)$.
The result now follows from Proposition \ref{p:kmeoc}.  
\end{proof}

\section{Proof of Theorem \ref{t:main}}
\label{mainSection}

\begin{Lem}[Robinson] \label{l:QvsQ}
Suppose that $G$ is a finite group, $Q \unlhd G$ is a $p$-subgroup and   
$\alpha \in H^2(G/Q, k^\times)$. 
We have 
$$\sum_{[x] \in Q^{\cl}/G} \ell(k_\alpha C_G([x]))= 
\sum_{\mu \in \Irr(Q)/G} \ell(k_\alpha C_G(\mu))\ .$$
\end{Lem}

This Lemma is  due to  Robinson,  and it is obtained as a combination of  
 \cite{robinson1987characters},  \cite{robinsonstaszewski1990}  (see  
discussion before Theorem 1.2 of  \cite{robinson1996local}). 
As a convenience to the reader, the  main ideas of the proof are presented   
in the Appendix.  

For a   finite group $H$ denote by  $ {\mathcal S}(H)$  the poset   of 
$p$-subgroups of $H$  (including the trivial  subgroup - so  notation is 
not standard). If $Q$ is a normal $p$-subgroup  of  a finite group $G$, 
then  for any  $[x]\in  Q^{\cl} $ (respectively  $\mu \in \Irr (Q)  $), 
we denote  by   $I([x]) $  (respectively  $I(\mu)) $   the   stabiliser 
in  $G/Q$  of $[x] $ (respectively  $\mu$)  under the   action of $G/Q$  
and for any subgroup   $R$ of   $ G/Q $, we denote by  $I([x], R) $ the 
intersection  of   $I([x])  $ with  $ N_{G/Q} (R) $ etc.

\begin{Lem}\label{L:sbluecon} 
Suppose that $G$ is a finite group, $Q \unlhd G$ is a $p$-subgroup and   
$\alpha \in H^2(G/Q, k^\times)$. Suppose that $C_G(Q) \leq Q$.  
If AWC holds, then 
$$\sum_{[x] \in Q^{\cl}/  G}  \  \   
\sum_ {R \in{\mathcal S}(I([x]))/ I ( [x]) }
z(k_{\alpha} (  I  ([x], R )/  R) ) = 
\sum_{\mu \in \Irr(Q)/G} \  
\sum_ {R \in{\mathcal S}(I(\mu))/ I(\mu)} 
z(k_{\alpha} (I ( \mu, R)/R)).$$ 
\end{Lem}  

\begin{proof}      
Let $\mu \in  \Irr(Q)$.  The full inverse image  of $I(\mu)  \leq G/Q$ 
in $G$ is $C_G(\mu) $  and for any $p$-subgroup  $R$  of  
$L/Q= I(\mu) $,  $I(\mu, R) =N_{L/Q} (R) $. Hence,   by AWC  and   
Proposition \ref{p:centext} applied with $ L =  C_G(\mu) $, we have that    
$$\ell(k_{\alpha} (kC_G(\mu)  ) =  
\sum_ {R \in{\mathcal S}(I(\mu))/ I(\mu)} z(k_{\alpha} (I(R, \mu)/R)).$$
Similarly,   let $x \in  Q^{\cl} $.  The full inverse image  of  
$ I([x])  \leq G/Q  $  in $G$  is $C_G ([x])  $. Thus,   by  AWC  and   
Proposition \ref{p:centext} applied with $ L =  C_G([x] )$, we have that    
$$\ell(k_{\alpha} (kC_G([x]  ) = 
\sum_ {R \in{\mathcal S}(I([x]))/ I ( [x]) }
z(k_{\alpha} (  I  ( R,  [x] )/  R) ). $$

The result follows by  Lemma \ref {l:QvsQ}.
\end{proof}

Let $\F$ be a saturated fusion system on a finite $p$-group $S$ and let
$\alpha$ be an $\F$-compatible family. We recall some
earlier notation.  For any $\F$-centric $Q \leq  S$, by 
Remark \ref{reorderRemark}, we have 
\begin{equation}
\w_Q(\F, \alpha) =
\sum_{\sigma \in \N_Q/\Out_\F(Q)} (-1)^{|\sigma|} 
\sum_{\mu \in \Irr_K(Q)/I(\sigma)} z(k_{\alpha_Q}C_{I(\sigma)}(\mu))
\end{equation}

\begin{equation}
\w^*_Q(\F, \alpha) =\sum_{\sigma \in \N_Q / \Out_\F(Q)} 
(-1)^{|\sigma|} \sum_{[x] \in Q^{\cl}/I(\sigma)} 
z(k_\alpha C_{I(\sigma)}([x])) 
\end{equation} 
Also, since $\Out_\F(Q)=\Out_{N_\F(Q)}(Q)$    we have 
\begin{equation}\label{e:wvsw*}
\w_Q(\F, \alpha)=  \w_Q(N_{\F}(Q),  \alpha) \hspace{2mm} 
\mbox{ and } \hspace{2mm}  
\w^*_Q(\F, \alpha)=  \w^*_Q(N_{\F}(Q),  \alpha) .
\end{equation}

\begin{Lem} \label{l:pumpupw}  
Suppose that $G$ is a finite group and $Q \unlhd G$ is a $p$-subgroup 
with $C_G(Q) \le Q$.  Let  $S$ be a Sylow $p$-subgroup  of  $G $,   
$\F=  \F_{S}(G) $, $\bar G =G/Q $  and  let    ${\mathcal P}_Q$  denote 
the set of all strictly increasing  chains   of $p$-subgroups  in  
$\Out_{\F}(Q) $  starting at $1$.  Then, 
$$\w_Q(\F, \alpha)  = \sum_{\sigma  \in {\mathcal P}_Q/ \Out_{\F} (Q)  }  
(-1)^{|\sigma|}  \sum_{\mu \in \Irr(Q)/ I (\sigma) }  \  \  
\sum_ { R  \in{\mathcal S}(I(\sigma, \mu))/  I(\sigma, \mu)   } 
z(k_{\alpha} (I  (  R,  \sigma,  \mu  )/ R)) $$
and  
$$\w^*_Q(\F, \alpha)  = \sum_{\sigma  \in {\mathcal P}_Q/\Out_{\F}(Q) }  
(-1)^{|\sigma|} \sum_{[x] \in Q^{\cl}/ I (\sigma) }  \  \  
\sum_ {R \in{\mathcal S}(I(\sigma, [x]))/  
I(\sigma, [x])   } 
z(k_{\alpha} ( I  ( R,  \sigma,  [x])/R)) .$$
\end{Lem}  

\begin{proof}    By definition 
$$\w_Q(\F, \alpha)  = \sum_{\sigma  \in {\mathcal N_Q}/ \Out_{\F} (Q) }  
(-1)^{|\sigma|}   \sum_{\mu \in \Irr(Q)/ I (\sigma) } 
z(k_{\alpha} (I  ( \sigma,  \mu  ))). $$
 We claim that 
 $$\w_Q(\F, \alpha)  = \sum_{\sigma  \in {\mathcal P}_Q/ \Out_{\F} (Q) }   
(-1)^{|\sigma|}    \sum_{\mu \in \Irr(Q)/ I (\sigma) } 
z(k_{\alpha} ( I (  \sigma,  \mu  ))) .$$
Indeed, this follows  immediately from Lemma \ref{l:rad2} (or 
\cite[Proposition 3.3]{KnorrRobinson1989}). Next,  interchanging the  
order of   summation  on the  right hand side of the above  equation  
we  obtain
$$\w_Q(\F, \alpha) = \sum_{\mu \in \Irr(Q)/ \Out_{\F}(Q)} \  \  
\sum_{\sigma  \in {\mathcal P}_Q/ I(\mu)  }   (-1)^{|\sigma|}  
z(k_{\alpha} (I(\sigma, \mu))) .$$
Now we  claim that  
\begin{equation} \label{e:wpumpup}  
\w_Q(\F, \alpha)  =  \sum_{\mu \in \Irr(Q)/ \Out_{\F} (Q) } \  \  
\sum_{\sigma  \in {\mathcal P}_Q/ I(\mu)  } 
\sum_ { R  \in {\mathcal S}(I(\sigma, \mu))/  
I(\sigma, \mu)   }  (-1)^{|\sigma|}  
z(k_{\alpha} (I ( R,  \sigma,  \mu  )/ R)).    
\end{equation}

To prove the claim, let $\mu \in \Irr(Q)$ and for $R$ a $p$-subgroup of   
$ I(\mu) $, let ${\mathcal P}_Q ^{R}$  be the subset of   
${\mathcal P}_Q$ consisting of those chains  which are  normalised by 
$R$, i.e.  those chains  $\sigma $ such that $R \leq I(\sigma) $.
Then   
$$ \sum_{\sigma  \in {\mathcal P}_Q/ I(\mu)  } 
\sum_ { R  \in {\mathcal S}(I(\sigma, \mu))/  
I(\sigma, \mu)   }  (-1)^{|\sigma|}  
z(k_{\alpha} (I ( R,  \sigma,  \mu  )/ R)) $$   
is equal  to  
$$  
\sum_ { R  \in {\mathcal S}(I(\mu ))/ I( \mu)   }   
\sum_{\sigma  \in {\mathcal P}_Q^{R} / I( R, \mu)  }  
(-1)^{|\sigma|}  
z(k_{\alpha} (I ( R, \sigma,  \mu  )/ R))\ ,$$
where we  use Remark~\ref{reorderRemark}  with $ G = I(\mu)  $,   
$ X  =  \mathcal  {P}_Q   $, $ Y =  \mathcal {S}(I(\mu) )$ and 
$A$ equal to the subset of $X \times Y$  consisting of pairs  
$(\sigma, R)$  such that $R \leq  I (\mu, \sigma) $.

Suppose that $  R  \ne 1$ and let  $\sigma  =  
Q_0:=1 < Q_1 < \cdots < Q_n$ be an element of   
$ {\mathcal P}_Q^{ R}  $.  
If $R$ is not contained in $Q_n$, let $\sigma' $ be the chain 
obtained  from $\sigma $  by   appending $Q_n  R$.   Otherwise, let 
$j$ be the   smallest  integer   such that   $  R $ is  contained 
in $Q_j$. Note that $j \ne 0 $  since $R > 1$.   If $ Q_{j-1} R  
= Q_j $, then let $\sigma'$ be the  chain obtained from $\sigma $ by 
deleting $ Q_j$. Otherwise, let   $\sigma'$ be obtained from $\sigma $ 
by  inserting $Q_{j-1} R  $  in between $Q_{j-1}$ and   $Q_j$.  Then the 
pairing $\sigma \to \sigma' $  kills   
$$\sum_ { R  \in {\mathcal S}(I (\mu) )/  I ( \mu)}   
\sum_{\sigma  \in {\mathcal P}_Q^{R} / I(R, \mu)  }  
(-1)^{|\sigma|}  z(k_{\alpha}(I(  R,  \sigma, \mu)/ R))\ . 
$$
Hence, only the terms with $ R  =1 $  survive, and the claim follows.   
Interchanging the  order of summation in the outer two  terms of 
Equation \ref{e:wpumpup}  gives the desired    
expression  for $ \w_Q(\F, \alpha)$. The  proof  for $\w^*_Q(\F, \alpha)$ 
is entirely similar.
\end{proof}

\begin{Prop}  \label{p:weakAWCw}   
Let $\F$  be a saturated fusion system on a finite $p$-group $S$ and 
let $\alpha $ be an $\F$-compatible family. Suppose that  
AWC holds. Then $\w_Q(\F, \alpha)=  \w^*_Q(\F, \alpha)$ 
for all $\F$-centric subgroups $Q$ of $S$.
\end{Prop}

\begin{proof}     
Let $Q \leq S$ be $\F$-centric. By  Equation \ref {e:wvsw*}  we may 
assume that  $\F = N_{\F}(Q)$  and hence by
\cite[Proposition C]{BCGLO2005}  that $\F=\F_S(G)$  for some finite 
group   $G$   with $S$ as  a Sylow $p$-subgroup and containing  $Q$ as 
a normal subgroup with $C_G(Q) = Z(Q)$. By Lemma~\ref{l:pumpupw}, we have
$$\w_Q(\F, \alpha)  = \sum_{\sigma  \in {\mathcal P}_Q/ \Out_{\F} (Q)  }  
(-1)^{|\sigma|}  \sum_{\mu \in \Irr(Q)/ I (\sigma) }  \  \  
\sum_ { R  \in{\mathcal S}(I(\sigma, \mu))/  I(\sigma, \mu)   } 
z(k_{\alpha} (I  (  R,  \sigma,  \mu  )/ R))  $$  
and
$$\w^*_Q(\F, \alpha)  = \sum_{\sigma  \in {\mathcal P}_Q/\Out_{\F}(Q) }  
(-1)^{|\sigma|} \sum_{[x] \in Q^{\cl}/ I (\sigma) }  \  \  
\sum_ {R \in{\mathcal S}(I(\sigma, [x]))/  
I(\sigma, [x])   } 
z(k_{\alpha} ( I  ( R,  \sigma,  [x]  )/\ R)) .$$

Let  $\sigma  \in {\mathcal  P}_Q $. By  applying   Lemma \ref{L:sbluecon}  
to the  inverse  image  $N_{G}(\sigma )$  of $I (\sigma) $  in   $G$, 
we obtain
 $$ \sum_{\mu \in \Irr(Q)/ I (\sigma) }  \  \  
\sum_ { R  \in{\mathcal S}(I(\sigma, \mu))/  I(\sigma, \mu)   } 
z(k_{\alpha} (I  (  R,  \sigma,  \mu  )/ R))=  $$
$$\sum_{[x] \in Q^{\cl}/ I (\sigma) }  \  \  
\sum_ {R \in{\mathcal S}(I(\sigma, [x]))/  
I(\sigma, [x])   } 
z(k_{\alpha} ( I  ( R,  \sigma,  [x]  )/\ R)) .$$
The result follows.
\end{proof}

\begin{proof}[Proof of Theorem~\ref{t:main}]This  is immediate from Proposition   
\ref{p:weakAWCw}.\end{proof}

\medskip
We present an alternate  proof of Theorem \ref{t:main} which is shorter
but makes use of the fact, due to Robinson \cite{robinson1996local}, 
that AWC implies SOWC.
Let $\F$ be a saturated fusion system on a finite $p$-group $S$, and let
$\alpha$ be an $\F$-compatible family. As a consequence of Lemma \ref{l:rad}, the quantities $\m(\F,\alpha)$,
$\m^*(\F,\alpha)$, and $\m(\F, \alpha,d)$ remain unchanged 
under restricting the sums over isomorphism classes of $\F$-centric
subgroups of $S$ to $\F$-centric radical subgroups. We spell this out.

\begin{Lem}\label{l:red}
Let $Q$ be an $\F$-centric subgroup of $S$ and let $d$ be a non-negative
integer. Suppose that $Q$ is not
$\F$-radical. Then 
$$\w_Q(\F,\alpha) = \w^*_Q(\F,\alpha) = \w_Q(\F,\alpha,d) = 0\ .$$
\end{Lem}

\begin{proof}
Using Remark \ref{reorderRemark}, we have
$$\w_Q(\F,\alpha) = \sum_{\sigma\in\N_Q/\Out_\F(Q)} (-1)^{|\sigma|} 
\sum_{\mu\in\Irr(Q)/I(\sigma)}\ z(k_\alpha I(\sigma,\mu))$$
The quantity in the second sum depends only on $I(\sigma)$.
Since $Q$ is not radical, we have $O_p(\Out_\F(Q))\neq$ $1$.  
Thus Lemma \ref{l:rad}, applied to the group $G=\Out_\F(Q)$
and the function $\f$ on subgroups of $G$ defined by 
$$\f(H) := \begin{cases}
\displaystyle\sum_{\mu\in\Irr(Q)/I(\sigma)}\ z(k_\alpha I(\sigma,\mu)) 
& \mbox{ if $H=I(\sigma)$ for some $\sigma\in$ $\N_Q$ } \\ 
0 &  \mbox{ otherwise}  \end{cases}$$ 
implies that $\w_Q(\F,\alpha)=0$. Similar arguments show that
$\w^*_Q(\F,\alpha) = \w_Q(\F,\alpha,d) = 0$.
\end{proof}

Note that  by Lemma \ref{l:red}, we have
\begin{equation}\label{e:mw} 
\m(\F,\alpha) = \sum_{Q \in \F^{cr}/\F} \w_Q(\F,\alpha) 
\hspace{2mm} \mbox{ and } \hspace{2mm} 
\m^*(\F,\alpha) = \sum_{Q \in \F^{cr}/\F} \w^*_Q(\F,\alpha). 
\end{equation}

\begin{Lem} \label{l:compw}  
Suppose that $\m^*(\G,\beta) = \m(\G,\beta)$ for all pairs $(\G,\beta)$, 
where $\G$ is a saturated constrained fusion system and $\beta$ is a
$\G$-compatible family. Then $\m(\F,\alpha)=\m^*(\F,\alpha)$.
\end{Lem}

\begin{proof} 
We prove that $\w_Q(\F, \alpha) = \w^*_Q(\F, \alpha)$ for each fully
$\F$-normalized, $\F$-centric, $\F$-radical subgroup $Q \le S$. 
Since $\F$ is saturated, there is a fully $\F$-normalized subgroup 
in each $\F$-conjugacy class, and so the result will then follow from 
$\eqref{e:mw}$. 

Suppose the above assertion is false, so that $\w_Q(\F, \alpha) \neq 
\w^*_Q(\F, \alpha)$ for some $Q$. Among all such counterexamples $\F$ 
and $Q$, choose one such that $|\F| + |S:Q|$ is minimal, where $|\F|$ 
denotes the number of morphisms in $\F$. Note that 
$\Out_{N_\F(Q)}(Q) = \Out_\F(Q)$, and $Q$ is also fully 
$N_\F(Q)$-normalized, $N_\F(Q)$-radical, and $N_\F(Q)$-centric. Since the
sums $\w_Q(\F,\alpha)$ and $\w_Q^*(\F,\alpha)$ depend only on $Q$ and
$\Out_\F(Q)$ and not on $\F$, it follows by minimality that $\F = N_\F(Q)$. 

We have shown that $\F$ is constrained with normal centric subgroup $Q$. 
In particular, $\m(\F,\alpha) = \m^*(\F,\alpha)$ by assumption, and $Q$ 
is contained in every $\F$-centric radical subgroup (see e. g. 
\cite[Lemma 2.4]{lynd2017weights}).  From \eqref{e:mw}, $\m(\F, \alpha)$ 
is the sum of $\w_Q(\F,\alpha)$ and $\w_{R}(\F,\alpha)$ as $R$ ranges 
over the fully $\F$-normalized, $\F$-centric radical subgroups with 
$R > Q$. The same holds for $\w_Q^*(\F,\alpha)$.  By induction 
$\w_{R}(\F,\alpha) = \w_R^*(\F,\alpha)$ for each such $R > Q$ 
(since $N_\F(R) \subsetneq \F)$.  It follows that 
$\w_Q(\F,\alpha) = \w_Q^*(\F,\alpha)$ after all, a contradiction.
\end{proof}  

It thus suffices by Lemma~\ref{l:compw} to prove
$\m(\F,\alpha)=\m^*(\F,\alpha)$ in the case where $\F$ is constrained. 

\begin{Prop}\label{p:weightconj}
Suppose AWC holds for all blocks of all finite groups. If $\F$ is constrained,
then $\k(\F,\alpha)=\m(\F,\alpha)$. 
\end{Prop}

\begin{proof}
Assume that $\F$ is constrained. By Proposition~\ref{p:fareal}, we may fix a
model $G$ for $\F$, a $p'$-central extension $\widehat{G}$ of $G$, and a block
$b$ of $k\widehat{G}$ such that $(\F,\alpha)$ is realized by $k\widehat{G}b$.
By Proposition \ref{p:links}, since AWC holds for all blocks, we have 
$$\k(\F,\alpha)  = \k(B).$$
On the other hand, again since AWC holds for all blocks, the results of 
\cite{robinson1996local}, \cite{Robinsonweight2004} show that
$\m(\F,\alpha) = \k(B)$.
\end{proof}

\begin{proof}[Proof of Theorem~\ref{t:main}]
Assume AWC holds for all blocks of finite groups. By 
Theorem~\ref{t:main2}, we have $\k(\F,\alpha) = \m^*(\F,\alpha)$. Hence, 
$\m(\F,\alpha) = \m^*(\F,\alpha)$ whenever $\F$ is constrained by 
Proposition~\ref{p:weightconj} and assumption. Therefore, 
$\m(\F,\alpha) = \m^*(\F,\alpha)$ by Lemma~\ref{l:compw}.
\end{proof}

\begin{remark} \label{r:bluecon}  Suppose that with the notation  and hypothesis  of Lemma~\ref{L:sbluecon}   we have  
\begin{equation}  \label{e:blue}\sum_{[x] \in Q^{\cl}/G} z(k_\alpha C_{\Out_G(Q)}([x]))= 
\sum_{\mu \in \Irr(Q)/G} z(k_\alpha C_{\Out_G(Q)}(\mu))\ . 
\end{equation}  

Then, Theorem \ref{t:main}   is an immediate consequence   of    Theorem \ref{t:main2}. If $Q$ is a Sylow $p$-subgroup of $G$, then  for all $x \in Q$,  $z(k_\alpha C_{\Out_G(Q)}([x]))= \ell(k_\alpha C_{\Out_G(Q)}([x])) $ and for all $\mu \in \Irr(Q)$, $z(k_\alpha C_{\Out_G(Q)}([\mu]))= \ell(k_\alpha C_{\Out_G(Q)}([\mu])) $, hence \ref{e:blue}   is the same as the equality in Lemma \ref{l:QvsQ}.  Also note that by results of K\"{u}lshammer-Robinson \cite{KulshammerRobinson1987},  the  right  side of the equality of   Lemma~\ref{l:QvsQ}   is the rank of the  subgroup of the  group of generalised characters  of $G$,   generated by relatively $Q$-projective characters   in  a certain   sum of  $p$-blocks  of $G$,    and     the right hand side of   \ref{e:blue}    is   the  number  of   $Q$-projective irreducible  characters  of $G$, again in the relevant sum of blocks. 
\end{remark}  

\section{Example: fusion systems on an extraspecial group of order $p^3$ }
\label{exampleSection} 

In this section we consider the saturated fusion systems over an 
extraspecial group $S$ of order $p^3$ and exponent $p$, as classified 
by Ruiz and Viruel \cite{RuizViruel2004}. There are three exotic fusion 
systems over $S$ when $p = 7$, and many of the conjectures listed in 
Section~\ref{conjSection} are still nontrivial to verify in these cases.  
We show here that each nonconstrained saturated fusion system over $S$ 
supports only the zero compatible family, and we verify using 
computations in Magma \cite{Magma} that a number of the conjectures in 
Section~\ref{conjSection} hold in this case.

For the remainder of this section, $S \cong p_+^{1+2}$ denotes an 
extraspecial group of order $p^3$ and exponent $p$. Fix the presentation
\begin{eqnarray}
\label{e:p3pres}
\langle \a,\b,\c \mid \a^p=\b^p=\c^p=[\a,\c]=[\b,\c]=[\a,\b]\c^{-1}=1\rangle
\end{eqnarray}
for $S$. Each element of $S$ can then be written in the form $\a^r\b^s\c^t$,
for unique nonnegative integers $0 \leq r,s,t \leq p-1$. Let $\epsilon$ be a
fixed primitive $p$'th root of $1$ in $\mathbb{C}$. 
We first list some basic information about $S$, along with general information
about the associated saturated fusion systems on $S$. 

\begin{Lem}\label{l:exsp-setup}
Let $S$ be an extraspecial $p$-group of order $p^3$ and of exponent $p$, as
given above.

\begin{enumerate}
\item The map which sends $X:=\left(\begin{smallmatrix} x & y \\ z & w
\end{smallmatrix} \right)$ to the class $[\varphi] \in \Out(S)$ of the
automorphism $\varphi$ given by $$\a \longmapsto \a^x\b^z, \hspace{4mm} \b
\longmapsto \a^y\b^w, \hspace{4mm}  \c  \longmapsto \c^{\det(X)},$$ is an
isomorphism from $\GL_2(p)$ to $\Out(S)$.
\item A complete set of $S$-conjugacy class representives of elements 
of $S$ is given by 
\[
\{\a^i\b^j \mid 0 \le i,j \le p-1 \} \cup \{\c^k \mid 1 \le k \le p-1\}.
\] 
\item $\Irr(S)$ consists of $p^2$ linear characters $\chi_{u,v}$ \textup{(}$0
\le u,v \le p-1$\textup{)}, and $p-1$ faithful characters $\phi_u$
\textup{(}$1 \le u \le p-1$\textup{)} of degree $p$. The characters are given
explicitly by 
$$\chi_{u,v}(\a^r\b^s\c^t):= \epsilon^{ru+sv} \quad\text{ and }\quad
\phi_u(\a^r\b^s\c^t):=
\begin{cases} 
p\epsilon^{ut} & \mbox{ if $r=s=0$} \\ 0 & \mbox{ otherwise.} 
\end{cases}.$$ 
\end{enumerate}
\end{Lem}

\begin{proof}
See \cite{RuizViruel2004}, for example, for (1) and (2). Part (3) is contained
in \cite[Theorem~5.5.4]{Gorenstein1980}.   
\end{proof}

The elementary abelian subgroups of $S$ of order $p^2$ are in one-to-one
correspondence with the points of the projective line over 
$\mathbb{F}_p$; set
\[
Q_i:=\langle \c,\a\b^i \rangle  
\hspace{4mm} (0 \le i \le p-1) \hspace{4mm} \mbox{ and } \hspace{4mm} 
Q_p:=\langle \c,\b \rangle.
\]

From \cite[Section III.6.2]{AschbacherKessarOliver2011}, we have the 
following description of the fusion systems on $S$.

\begin{Thm}\label{t:classp3}
Let $\F$ be a saturated fusion system on $S$. Then 
\[
\F^{cr} \subseteq \{Q_i \mid 0 \le i \le p\} \cup \{S\}
\]
and the following hold:
\begin{enumerate}
\item $\Out_\F(S)$ is a $p'$-group.
\item For each $i$ and $\alpha \in \Aut_\F(S)$ such that $\alpha(Q_i)=Q_i,
\alpha|_{Q_i} \in \Aut_\F(Q_i)$.
\item If $Q_i \in \F^r$, then $\SL_2(p) \le \Aut_\F(Q_i) \le \GL_2(p)$. If $Q_i
\notin \F^r$ then $\Aut_\F(Q_i)=\{\varphi|_{Q_i} \mid \varphi \in
N_{\Aut_\F(S)}(Q_i)\}$.
\item If $Q_i \in \F^r$ and $\beta \in N_{\Aut_\F(Q_i)}(Z(S))$, then $\beta$
extends to an element of $\Aut_\F(S)$.
\end{enumerate}
Conversely, any fusion system $\F$ over $S$ for which conditions (1)-(4) hold
and for which each morphism is a composition of restrictions of
$\F$-automorphism of $S$ and the $Q_i$ is saturated. Moreover, in this case,
$\F$ is constrained if and only if at most $1$ of the $Q_i$ is centric and
radical.
\end{Thm}

When $\F$ is nonconstrained, the possibilities for $\Out_\F(S)$ and
$\Out_\F(Q_i)$ (and hence $\F$) are given in 
\cite[Tables 1.1 and 1.2]{RuizViruel2004}. 
Apart from the $p$-fusion systems of $\PSL_3(p)$, $p \ge 3$
and their almost simple extensions, there are thirteen exceptional fusion
systems for $3 \le p \le 13$, three of which are exotic. 

\medskip   
Next we show that there are no nonzero compatible
families for the nonconstrained fusion systems over $S$.  
Recall that the Schur multiplier $M(G)$ of a finite group $G$ is a finite
abelian group, which may be defined as the second cohomology group
$H^2(G,\mathbb{C}^\times)$.  In computing the Schur multiplier of various
groups, we make use of its connection with stem extensions.  A \textit{stem
extension} of a group $G$ is a central extension $1 \to Z \to \hat{G} \xra{\pi}
G \to 1$ such that $\ker(\pi) \leq Z(\hat{G}) \cap [\hat{G},\hat{G}]$. We will
often identify $\ker(\pi)$ with $Z$.  If $Z \cong M(G)$ in this situation, then
the extension (or $\pi$, or $\hat{G})$ is said to be a \textit{Schur covering}
of $G$.  
Given a central extension as above, there is an associated
inflation-restriction exact sequence

\begin{eqnarray}
\label{e:LHS}
1 \to H^{1}(G,\mathbb{C}^\times) \xra{\inf} H^1(\hat{G},\mathbb{C}^\times)
\xra{\res} \Hom(Z,\mathbb{C}^\times) \xra{\tra} H^2(G,\mathbb{C}^\times)
\xra{\inf} H^2(\hat{G},\mathbb{C}^\times) 
\end{eqnarray}

in which three of the maps are given by inflation or restriction, and the
fourth is the transgression map. This is defined by first choosing a cocycle
$\alpha$ representing the class $[\alpha] \in H^2(G,Z)$ of the extension.  For
any homomorphism $\phi \in \Hom(Z,\mathbb{C}^\times)$, post-composition with
$\phi$ yields a 2-cocycle with values in $\mathbb{C}^\times$, and then
$\tra(\phi)$ is defined as the class $[\phi \circ \alpha] \in H^2(G,
\mathbb{C}^\times)$. 

The next lemma collects a number of general results regarding the Schur
multiplier which will be used later in special cases. The first four 
parts are due to Schur. In parts (4) and (5) results of Schur and of 
Blackburn \cite{Blackburn1972} are quoted, and require the following 
additional notation. Denote the abelianization of a group $G$ by 
$G^{\ab}$, and write $G^{\ab} \wedge G^{\ab}$ for the quotient of 
$G^{\ab} \otimes_{\mathbb{Z}} G^{\ab}$ by the subgroup generated by 
$a \otimes b + b \otimes a$ as $a$ and $b$ range over $G^{\ab}$. 

\begin{Lem}\label{l:coverings}
The following hold for a finite group $G$.
\begin{enumerate}
\item If $1 \to Z \to \hat{G} \to G \to 1$ is a stem extension, then 
the transgression map $\tra$ in $\eqref{e:LHS}$ is injective, and 
hence $Z \cong \Hom(Z,\mathbb{C}^\times)$ is isomorphic to a subgroup 
of $M(G)$. 
\item There exists a Schur covering $1 \to Z \to \hat{G} \xra{\pi} G \to 1$ 
and, for any such covering, the associated transgression map is an 
isomorphism.
\item If there exists a Schur covering as in (2) such that $\pi\colon
\pi^{-1}(H) \to H$ is a stem extension of some subgroup $H$ of $G$, then the
restriction map $M(G) \to M(H)$ is injective.  
\item If $G = G_1 \times G_2$ is a direct product, then one has 
\[
M(G) \cong M(G_1) \times M(G_2) \times (G_1^{\ab} 
\otimes_{\mathbb{Z}} G_2^{\ab}). 
\]
\item (Blackburn) Assume $G = K \wr H$ is a wreath product, and write $m$ for
the number of involutions in $H$. Then $M(G)$ is isomorphic to the direct
product of $M(H)$, $M(K)$, $\frac{1}{2}(|H|-m-1)$ copies of $K^{\ab} \otimes
K^{\ab}$, and $m$ copies of $K^{\ab} \wedge K^{\ab}$.
\end{enumerate}
\end{Lem}

\begin{proof}   We refer to \cite[Lemma~11.42]{CurtisReiner1990} for (1) and to
\cite[Theorem~11.43]{CurtisReiner1990} for (2).  Assume the hypotheses of 
(3) and set $\widehat{H} = \pi^{-1}(H)$. Choose a $2$-cocycle $\alpha$ 
representing the class of the central extension 
$1 \to Z \to \hat{G} \to G \to 1$. Then $\alpha|_{H\times H}$ represents 
the class of the central extension $1 \to Z \to \hat{H} \to H \to 1$, 
and the square 
\[
\xymatrix{
\Hom(Z,\mathbb{C}^\times) \ar[r]^{\tra_G} \ar[d]_{\id} & 
H^2(G,\mathbb{C}^\times) \ar[d]^{\res}\\
\Hom(Z,\mathbb{C}^\times) \ar[r]_{\tra_H} & H^2(H,\mathbb{C}^{\times})
}
\]
commutes. As $\tra_H$ is injective by (1) and $\tra_G$ an isomorphism 
by (2), part (3) follows. Finally, a  proof  of  (4)  may be found in  \cite[Corollary~3]{Wiegold71}   and  (5) is \cite[Theorem~1]{Blackburn1972}  but see also \cite[Theorem~2.2.10,  Theorem~6.3.3]{Karpilovsky1987}.
\end{proof}

Whenever $p$ is a prime, we write $M(G)_p$ for the $p$-primary part of $M(G)$,
and $M(G)_{p'}$ for the $p'$-primary part of $M(G)$.  The following lemma
collects some basic information about the various primary parts of the Schur
multiplier. 

\begin{Lem}\label{l:cohom}
Let $G$ be a finite group, let $p$ be a prime, and let $k$ be an algebraically
closed field of characteristic $p$.
\begin{enumerate}
\item $H^2(G,k^\times)$ is isomorphic to the $p'$-part $M(G)_{p'}$ of the Schur
multiplier. 
\item If $H \leq G$ contains a Sylow $p$-subgroup of $G$, then the restriction
map $M(G)_p \to M(H)_p$ is injective. 
\item If $G$ has cyclic Sylow $p$-subgroups, then $M(G)_p = 1$.
\end{enumerate}
\end{Lem}

\begin{proof}
See \cite[Proposition~2.1.14]{Karpilovsky1987} for part (1). Part (2) is
derived from the fact that restriction to $H$ followed by transfer to $G$ is
multiplication by the index of $H$ in $G$, which by assumption is prime to $p$.
Then part (3) follows from (2) and the fact that   the second cohomology group of a finite cyclic group   with coefficients in a divisible group with trivial action is trivial    but see also\cite[Proposition~11.46]{CurtisReiner1990}.
\end{proof}

We now specialize to the following computations of Schur multipliers of
specific finite groups that appear as subgroups of automorphism groups of
centric radicals in certain nonconstrained saturated systems over $S$. 

\begin{Lem}\label{l:mg}
The following hold. 
\begin{enumerate}
\item If $p$ is any prime and $G$ is a subgroup of $\GL_2(p)$ containing
$\SL_2(p)$, then $M(G) = 1$. 
\item Let $G$ be a $2$-group of maximal class. Then $M(G) \cong C_{2}$ if $G$
is dihedral, while $M(G) = 1$ if $G$ is semidihedral or quaternion.  If $G$ is
a dihedral $2$-group and $V$ is any four subgroup of $G$, then the restriction
$M(G) \to M(V)$ is injective. 
\item Let $G = C_n \wr C_2$ with $n \geq 2$. Then $M(G) \cong C_2$ if $n$ is
even, and $M(G) = 1$ if $n$ is odd.  If $J \leq G$ is the homocyclic subgroup
of rank $2$ and exponent $n$, then the restriction $M(G) \to M(J)$ is
injective. 
\item If $G = \Out_{\F}(S)$ for some fusion system $\F$ over $S$ appearing in
Table~1.2 of \cite{RuizViruel2004}, then $M(G)_{2'} = 1$. Moreover, either
$M(G)_2 = 1$, or $M(G)_2 \cong C_2$ and $G$ is $D_8$, $S_3 \times C_6$, $C_6
\wr C_2$, $D_8 \times C_3$, or $D_{16} \times C_3$.
\end{enumerate}
\end{Lem}

\begin{proof}
The fact that $\SL_2(p)$ has trivial multiplier for all primes $p$ is 
standard: note that all Sylow $r$-subgroups of $\SL_2(p)$ are cyclic, 
except when $p \geq 3$ and $r = 2$, in which case a Sylow $r$-subgroup 
is generalized quaternion. It follows that $M(G)_r = 1$ for all primes 
$r$ by (2) below and Lemma~\ref{l:cohom}(2). Also $\GL_2(3)$ also has 
trivial multiplier by (2) and Lemma~\ref{l:cohom}(2) since a Sylow 
$2$-subgroup of $\GL_2(3)$ is semidihedral. So to finish the proof of 
(1), we may assume that $p \geq 5$. Let $G$ be a subgroup of $\GL_2(p)$ 
containing $N = \SL_2(p)$. Since $p \geq 5$, $N$ is perfect. Fix a 
group $\widehat{G}$ having a central subgroup $Z$ such that 
$\widehat{G}/Z \cong G$.  Identify $\widehat{G}/Z$ with $G$ and let
$\pi \colon \widehat{G} \to G$ be the canonical projection.  We shall 
show that there is a complement to $Z$ in $\widehat{G}$. Let 
$\widehat{N}$ be the preimage of $N$ under $\pi$.  Then $\widehat{N}$ 
contains $Z$ and splits over it, as $M(N) = 1$. Fix any complement 
$N_0$ of $Z$ in $\widehat{N}$, so that $\widehat{N} = N_0 \times Z$.  
Then $N_0 \cong \SL_2(p)$. We claim that $N_0$ is normal in 
$\widehat{G}$; it is clear that $N_0$ is normal in $\widehat{N}$. In 
general, conjugation by $g \in \widehat{G}$ sends an element $h \in N_0$ 
to $h'\zeta_g(h)$, where $h' \in N_0$ and $\zeta_g(h) \in Z$ are
uniquely determined. Also since $Z$ is central in $\widehat{G}$, the 
assignment $h \mapsto \zeta_g(h)$ is a group homomorphism from $N_0$ to 
$Z$.  Since $N_0$ is perfect, it follows that $\zeta_g = 1$ for each 
$g \in \widehat{G}$, i.e. $N_0$ is normal in $\widehat{G}$. Write 
quotients by $N_0$ with pluses. Now $\widehat{G}^+$ is a central 
extension of $Z$ by $\widehat{G}/\widehat{N} \cong G/N$, which is a 
cyclic $p'$-group. On the other hand, we have $M(G/N) = 1$ by
Lemma~\ref{l:cohom}(3), applied with each prime divisor $r$ of $|G/N|$ 
in the role of ``$p$'' there.  Thus, we may fix a complement $K^+$ of 
$Z^+$ in $\widehat{G}^+$, and let $K$ be the preimage of $K^+$ in 
$\widehat{G}$.  Then $K$ is a complement to $Z$ in $\widehat{G}$. 

Part (2) is implied by Lemma~\ref{l:coverings} as follows. Let $G$ be 
a $2$-group of maximal class, so that $G$ is dihedral, semidihedral or
quaternion. Then $M(G)$ is a $2$-group.  Let $\pi \colon \hat{G} \to G$ 
be any Schur covering of $G$ with $Z = \ker(\pi)$. Then since 
$Z \leq [\hat{G}, \hat{G}]$, it follows that 
$\hat{G}^{\ab} \cong G^{\ab}$ is of order $4$. By
\cite[Theorem~5.4.5]{Gorenstein1980}, $\hat{G}$ is of maximal class, so
$Z(\hat{G})$ is of order $2$. Then either $M(G) = 1$, or 
$Z(\hat{G}) = Z$ and $M(G) = C_2$. In the latter case, since $\hat{G}$ 
is of maximal class, we have $\hat{G}/Z \cong G$ is dihedral.  
Conversely, the dihedral group $\hat{G} = D_{2^{k+1}}$ provides a Schur 
covering $\pi\colon \hat{G} \to G$ of $G = D_{2^{k}}$.  Fix a four 
subgroup $V$ of $G$.  Then as $\pi^{-1}(V)$ is dihedral of order $8$, 
we have that the restriction map $M(G) \to M(V)$ is injective by
Lemma~\ref{l:coverings}(3). This completes the proof of (2).

To prove (3), apply Lemma~\ref{l:coverings}(5) with $K = C_n$ and 
$H = C_2$. Hence, $m = 1$ there. By Lemma~\ref{l:cohom}(3) and that 
result, $M(G) = K^{\ab} \wedge K^{\ab}$. The multiplication map 
$C_n \otimes C_n \to C_n$ is an isomorphism, where $C_n$ is viewed as 
an additive group, and under that map $a \otimes b + b \otimes a$ is 
sent to $2ab$.  Hence, $M(G) \cong C_n/2C_n \cong C_2$ if $n$ is even, 
and $1$ if $n$ is odd.  In order to prove the claim about the 
restriction $J$, we may by Lemma~\ref{l:cohom}(2) assume that $n = 2^l$ 
for some $l$, and then it suffices by Lemma~\ref{l:coverings}(3) to 
produce a double covering of $G$ which restricts to a stem extension of 
$J$.  To this end, the group $G = C_{2^l} \wr C_2$ has a presentation 
with generators $x$, $y$, and $t$, and defining relations 
$x^{2^l} = y^{2^l} = t^2 = xyx^{-1}y^{-1} = 1$ and $txt^{-1} = y$.  
Consider the group $\hat{G}$ with generators $\x, \y, \t$ and defining 
relations $\x^{2^l} = \y^{2^l} = \t^2 = 1$, $\t\x\t^{-1} = \y$, and 
$\z = [\x,\y]$ is of order $2$ and central.  Thus 
$\hat{G}/\gen{\x^{2^{l-1}},\y^{2^{l-1}}} \cong D_{16}$, and the obvious 
map $\pi\colon \hat{G} \to G$ is the pullback of the Schur covering 
$D_{16} \to D_{8}$. Let $\mathbf{J}$ be the preimage of $J$ in 
$\hat{G}$, and set $Z = \ker(\pi) = \gen{\z}$. Then by construction 
$Z \leq [\mathbf{J},\mathbf{J}] \cap Z(\mathbf{J})$, so 
$\pi\colon \mathbf{J} \to J$ is a stem extension of $J$. As noted 
above, this completes the proof of (3).

We now prove (4).  When $G$ is $D_8$, $SD_{16}$, $C_6 \wr C_2$, 
$S_3 \times C_3$, $S_3 \times C_6$, $D_8 \times C_3$, 
$D_{16} \times C_3$, or $SD_{32} \times C_3$, the claim follows from 
(1), (2), (3), and Lemma~\ref{l:coverings}(4).  It remains to consider 
the groups ``$4S_4$'' and ``$C_3 \times 4S_4$'' in 
\cite[Table~1.2]{RuizViruel2004}.  Note that $4S_4$ as appears in 
\cite{RuizViruel2004} is the normalizer $G$ in $\GL_2(5)$ of a Sylow
$2$-subgroup $Q \cong Q_8$ of $\SL_2(5)$.  The normalizer $N$ in 
$\SL_2(5)$ of $Q$ is the commutator subgroup of $G$, isomorphic to 
$\SL_2(3)$, and the quotient $G/N$ is cyclic of order $4$. Thus, 
$G^{\ab}$ is cyclic of order $4$. Therefore, it suffices to show that 
$M(G) = 1$, for then by Lemma~\ref{l:coverings}(4), we'll have 
$M(C_3 \times G) = 1$.  Since $G$ has Sylow $3$-subgroups of order $3$, 
it follows from Lemma~\ref{l:cohom}(3) that $M(G) = M_2(G)$.  Since 
$\SL_2(3) \cong Q \rtimes C_3$ is $2$-perfect and $G/N$ is cyclic, the 
exact same argument as given in (1) applies with $Z$ a $2$-group to 
show that $M(G)_2 = 1$.  This completes the proof of (4) and the lemma.
\end{proof}

\begin{Prop}\label{l:kp-extraspecial}
Let $p$ be an odd prime, and let $\F$ be a nonconstrained saturated fusion
system on an extraspecial $p$-group $S$ of order $p^3$ and of exponent $p$.
Then $\lim \A^2_\F = 0$. 
\end{Prop}
\begin{proof}
Consider the cochain complex $(C^*(\A_\F^2), \delta)$ computing the limits of
$\A_\F^2$ as in \cite{Linckelmann2009}. By Lemma \ref{l:mg}(1),
$M(\Out_\F(Q))_{p'} = 1$ for all elementary abelian subgroups $Q \in \F^{cr}$.
Thus, the $0$-th cochain group is $C^0(\A_\F^2) = M(\Out_\F(S))_{p'} =
M(\Out_\F(S))$, and the coboundary map 
\[
\delta^0 \colon M(\Out_\F(S)) \longrightarrow 
\bigoplus_{Q \in \F^{cr}, |Q| = p^2} M(\Out_\F([Q < S]))
\]
is the sum of the restriction maps $M(\Out_\F(S)) \to M(\Out_\F([Q<S]))$.
Thus, to complete the proof it suffices to show that at least one of these
restriction maps is injective.

We regard $\Out_\F(S) \leq \GL_2(p)$ as acting on $\{Q_i \mid  0 \leq i \leq
p\}$ as it does on the projective line, and then $\Out_{\F}([Q < S])$ is the
stabilizer of the point $Q$. We go through the possibilities for $\F$ appearing
in Tables~1.1 and 1.2 of \cite{RuizViruel2004}. Consider first a fusion system
$\F$ over $S$ occuring in Table~1.2. Then $C^0(\A_\F^2) = 1$ unless $G :=
\Out_\F(S)$ appears in Lemma \ref{l:mg}(4), and in those cases there exists
some $Q \in \F^{cr}$ with $|Q|=p^2$ such that $M(\Out_\F([Q < S]))$ is a four
subgroup.  Hence, $\ker(\delta^0) = 1$ by Lemma \ref{l:mg}(2), and so $\lim
\A_\F^2 = 0$ in this case. Now consider a nonconstrained fusion system
appearing in Table~1.1. Then $\F^{cr} = \{Q_0,Q_p\}$, and $G := \Out_\F(S)$ may
be taken in the normalizer in $GL_2(p)$ of the subgroup $T$ of diagonal
matrices, which stabilizes $Q_0$ and $Q_p$. Write $G_0 = \Out_\F([Q_0<S])$, for
short. Then $G_0 = G \cap T$.  Assume first that $r \neq p$ is an odd prime. If
$r$ divides $p+1$, then a Sylow $r$-subgroup of $GL_2(p)$ is cyclic, and hence
$M(G)_r = 1$ by Lemma~\ref{l:cohom}(3). Suppose $r$ divides $p-1$.  Then a
Sylow $r$-subgroup of $G$ is contained in $T$, and so $G_0$ contains a Sylow
$r$-subgroup of $G$.  Thus the restriction $M(G)_r \to M(G_0)_r$ is injective
by Lemma \ref{l:cohom}(2). It remains to consider $r = 2$.  Inspection of
Table~1.1 of \cite{RuizViruel2004} shows that either a Sylow $2$-subgroup of
$G$ stabilizes $Q_0$, in which case $M(G)_2 \to M(G_0)_2$ is injective by
Lemma~\ref{l:cohom}(2), or a Sylow $2$-subgroup $R$ of $G$ is isomorphic to
$C_{2^l} \wr C_2$ for some $l \geq 1$, and $R \cap G_0$ is the homocyclic
subgroup of $R$ of index $2$.  The restriction map $M(G)_2 \to M(G_0)_2$ is
injective in this latter case by Lemma \ref{l:mg}(3).  
\end{proof}

We now verify a number of the conjectures in Section~\ref{conjSection} for a
nonconstrained saturated fusion system over $S$.  In the thirteen 
exceptional cases complete proofs by hand could be written down, but 
since we have no reasonably general argument for the specific numerical 
computations, the conjectures are ultimately verified using computer 
calculations in Magma \cite{Magma}. 

In a saturated fusion system $\F$ on $S$, Lemma~\ref{l:exsp-setup}(a) and
Lemma~\ref{t:classp3}(1) allow one to identify $\Out_\F(S)$ with a subgroup of
$\GL_2(p)$ of order prime to $p$.  Relative to this setup, we will identify
elements of $\Out_\F(S)$ with matrices with respect to the basis $\{\a,\b\}$
(or rather, the basis $\{\a Z(S), \b Z(S)\}$ of $S^{\ab}$).  Write
$\Out^*_\F(S)$ for those elements of $\Out_\F(S)$ which have determinant $1$.
The next lemma gives a general calculation of the quantity $\m(\F,0,d)$. 

\begin{Lem}\label{l:det=1}
Let $\F$ be a saturated fusion system on $S$.  Then
\begin{enumerate}
\item $\m(\F,0,0) = \m(\F,0,1) = 0$,
\item $\m(\F,0,2) = \frac{p-1}{l} \cdot |\Out_\F^*(S)^{\cl}|$, where $l$
denotes the index of $\Out_\F^*(S)$ in $\Out_\F(S)$, and 
\item $\m(\F,0,3) = \sum_{\mu} z(kC_{\Out_\F(S)}(\mu))$, where $\mu$ runs over
a set of representatives for the $\Out_\F(S)$-orbits of linear characters of
$S$, and where $C_{\Out_\F(S)}(\mu)$ denotes the stabilizer of $\mu$ in
$\Out_\F(S)$. 
\end{enumerate}
Thus, 
\[
\m(\F,0) = \frac{p-1}{l} \cdot |\Out^*_\F(S)^{\cl}| + 
\sum_{\mu \in \Irr^3(S)/\Out_{\F}(S)} z(kC_{\Out_\F(S)}(\mu)).
\]
\end{Lem}

\begin{proof}
Each character of $Q_i \cong C_p \times C_p$ is linear, so has defect $2$,
while each character of $S$ has defect $2$ or $3$ by
Lemma~\ref{l:exsp-setup}(3).  This shows (1).

The quantity $\m(\F,0,2)$ is a sum over $\F$-conjugacy classes of centric
radicals $Q$ of the quantities $\w_Q(\F,0,2)$.  Fix first a centric radical $Q$
of order $p^2$. We claim $\w_{Q}(\F,0,2) = 0$.  By Theorem~\ref{t:classp3},
$\Out_\F(Q)$ is a subgroup of $\GL_2(p)$ containing $\SL_2(p)$ with index $a$,
say. First consider the trivial chain $\sigma = (\bar{Q}) \in \N_Q$. Then
$\Out_\F(Q)$ stabilizes $\sigma$. There are two orbits on characters, one nontrivial and one
trivial. Given a nontrivial character $\mu$, the stabilizer of $\mu$ in
$\Out_\F(Q)$ is a Frobenius group with normal subgroup of order $p$, and so has
$0$ projective simple modules by Lemma~\ref{l:z=0}.  Also, the trivial
character has stabilizer $\Out_\F(Q)$, which by Lemma~\ref{l:sl2p} has $a$
projective simple modules. Thus, the contribution to $\w_Q(\F,0,2)$ from the
trivial chain is $a$. Next let $\sigma = (\bar{Q} < \bar{S}) \in \N_Q$ be the
nontrivial chain. The stabilizer $I(\sigma)$ in $\Out_\F(Q)$ of $\sigma$ has a
normal subgroup $\bar{S}$ and the quotient by $\bar{S}$ is abelian order
$(p-1)a$.  Then $I(\sigma)$ has one orbit of size $1$ containing the trivial
character, and one of size $p-1$ consisting of those characters with kernel
$Z(S)$. The respective stabilizers have $\bar{S}$ as a normal subgroup, so the
contribution from these orbits is $0$, by Lemma~\ref{l:z=0} .  The group
$\bar{S}$ acts regularly on the remaining points of $\mathbb{P}^1(Q)$, and a
subgroup in $I(\sigma) \cap SL(Q)$ of order $p-1$ acts regularly on the
characters with kernel a given point. Thus, there is one further orbit with
stabilizer of order $a$, and this contributes $-a$ to $\w_{Q}(\F,0,2)$. Hence,
$\w_Q(\F,0,2) = a-a = 0$, as claimed.

Lastly, consider $\w_{S}(\F,0,2)$. Here, there is only the trivial chain with
stabilizer $\Out_\F(S)$. Note that $\Irr^2(S)$ consists of the $p-1$ faithful
characters denoted $\phi_u$ in Lemma~\ref{l:exsp-setup}. Let $l$ be the index
of $\Out^*_\F(S)$ in $\Out_\F(S)$. The action of $\Out_\F(S)$ on $\Irr^2(S)$ is
the same as the action on $Z(S)^\#$ (nonidentity elements), and there
$\Out_\F(S)$ acts semiregularly via the determinant map. So there are $(p-1)/l$
orbits, each of size $l$, and a representative for each orbit has stabilizer
$\Out^*_\F(S)$.  As $\Out^*_{\F}(S)$ is a $p'$-group, we have
$z(k\Out^*_{\F}(S))$ is the number of simple $k\Out^*_\F(S)$-modules, which is
the number of conjugacy classes.  This completes the proof of (2).

In (3), since $\Irr^3(Q_i)$ is empty, there is only one pair $(Q,\sigma)$ for
which the innermost sum in $\m(\F,0,3)$ is nonzero, namely the pair $(S, 1)$.
We thus have $I(\sigma) = \Out_\F(S)$ for this pair. Thus, $$\m(\F,0,3) =
\sum_{\mu \in \Irr^3(S)/\Out_\F(S)} z(kC_{\Out_\F(S)}(\mu)),$$ which completes
the proof of (3). The last statement then follows from (1)-(3), since clearly
$\m(\F,0,d) = 0$ for $d > 3$.
\end{proof}

We now calculate $\k(\F,0)$.

\begin{Lem}\label{l:k}
Let $\F$ be a saturated fusion system on $S$.  Then 
$$\k(\F,0)=\w(\F,0)+\frac{p-1}{l} \cdot |\Out^*_\F(S)^{\cl}|
+\sum_{([x] \in S^{\cl}\backslash \{1,\c\})/\F} z(kC_{\Out_\F(S)}([x])).$$
\end{Lem}

\begin{proof}
By definition, we have
$$\k(\F,0)= \w(\F,0)+\w(C_\F(\c),0)+
\sum_{[x] \in (S^{\cl}\backslash \{1,\c\})/\F} \w(C_\F(x),0),$$ 
where in the latter sum $[x]$ runs over $S$-classes for which 
$\langle x \rangle$ is fully $\F$-centralized. Let $1 \neq x \in S$ be 
such that $x$ is not $\F$-conjugate to $\c$ and $\langle x \rangle$ is 
fully $\F$-centralised. If $Q \le C_S(x)$ is $C_\F(x)$-centric radical, 
then $|Q|=p^2$ and $x \in C_S(Q)=Q$. Therefore 
$Q = \langle x, \c \rangle$ and  $\Out_{C_\F(x)}(Q)$ does not contain a 
copy of $\SL_2(p)$, a contradiction. We conclude that 
$C_S(x) \unlhd C_\F(x)$, and 
$$\w(C_\F(x),0)=z(k\Out_{C_\F(x)}(C_S(x))=
z(kC_{\Out_\F(C_S(x))}(x)),$$ 
by Lemma \ref{l:cfxcent}. We claim that 
$C_{\Out_\F(C_S(x))}([x]) \cong C_{\Out_\F(S)}([x])$. Since $\F$ is saturated, 
the restriction map 
$$\res: C_{\Aut_\F(S)}(x) \rightarrow 
C_{\Aut_\F(C_S(x))}(x)$$ 
is surjective with kernel containing $C_{\Inn(S)}(x)$ as a Sylow 
$p$-subgroup. If $A$ is a $p'$-subgroup of $\ker(\res)$ then
$A$ commutes with both $C_{\Inn(S)}(x) \cong \langle x \rangle$ and $C_S(C_{\Inn(S)}(x))$. Hence Lemma \ref{l:axb} implies that $A=1$, and 
$C_{\Out_\F(S)}([x]) \cong C_{\Aut_\F(S)}(x)/C_{\Inn(S)}(x) 
\cong C_{\Out_\F(C_S(x))}(x)$ as claimed.

Clearly $S \unlhd C_\F(\c)$ and so 
$$\w(C_\F(\c),0)=z(k\Out_{C_\F(\c)}(S))=
z(kC_{\Out_\F(S)}([\c]))=\frac{p-1}{l} \cdot |\Out^*_\F(S)^{\cl}|,$$ 
where $l$ is as given in Lemma \ref{l:det=1}. This completes the proof.
\end{proof}

\begin{Rem}
From the classification in \cite{RuizViruel2004} we see that for any 
nonconstrained fusion system $\F$ over $S$, one has $l = p-1$ in 
Lemmas~\ref{l:det=1} and \ref{l:k} . Indeed, if $\F$ is nonconstrained, then $Z(S) =
ZJ(S)$ is not weakly $\F$-closed, and then a result of Glauberman then suggests
that $\Out_{\F}(Z(S)) \cong C_{p-1}$ for any nonconstrained fusion system over
$S$; c.f.  \cite[Theorem~14.14]{Glauberman1971}.
\end{Rem}

\begin{Lem}\label{l:someconjclasses}
Let $\omega$ be a generator of the multiplicative group
$\mathbb{F}_p^\times$.
\begin{enumerate}
\item The wreath product $ \langle
\left(\begin{smallmatrix} \omega & 0 \\ 0 & 1 
\end{smallmatrix} \right), 
\left(\begin{smallmatrix} 1 & 0 \\ 0 & \omega \end{smallmatrix} \right), 
\left(\begin{smallmatrix} 0 & 1 \\ 1 & 0 \end{smallmatrix} \right)\rangle \cong C_{p-1} \wr C_2$ 
has $(p-1)(p+2)/2$ conjugacy classes.
\item $ \langle \left(\begin{smallmatrix} \omega & 0 \\ 0 & \omega^{-1} 
\end{smallmatrix} \right), 
\left(\begin{smallmatrix} 0 & -1 \\ 1 & 0 \end{smallmatrix} \right) \rangle \cong C_{p-1}.C_2$ 
has $(p+5)/2$ conjugacy classes.
\item If $3 \mid p-1$ then 
$\langle \left(\begin{smallmatrix} \omega^3 & 0 \\ 0 & 1 
\end{smallmatrix} \right), 
\left(\begin{smallmatrix} \omega & 0 \\ 0 & \omega 
\end{smallmatrix} \right), 
\left(\begin{smallmatrix} 0 & 1 \\ 1 & 0 \end{smallmatrix} \right) \rangle$ 
has $(p-1)(p+8)/6$ conjugacy classes.
\item If $3 \mid p-1$ then 
$\langle \left(\begin{smallmatrix} \omega^3 & 0 \\ 0 & \omega^{-3} 
\end{smallmatrix} \right), 
\left(\begin{smallmatrix} 0 & -1 \\ 1 & 0 \end{smallmatrix} \right) \rangle$ 
has $(p+17)/2$ conjugacy classes.
\end{enumerate}
\end{Lem}

\begin{proof}
If $B$ denotes the base of the wreath product $G:=C_n \wr C_2$, then $Z(G)$ is
a cyclic subgroup of order $n$ in $B$ so $B \backslash Z(G)$ is the union of
$(n^2-n)/2$ classes. There are $n$ classes of elements in $G$ outside
$B$ which yields $n(n+3)/2$ classes altogether and (1) holds. Next, if $G$
denotes the group in (2) and $H$ is the cyclic subgroup of order $p-1$, we see
that, apart from $Z(G)$ (of order $2$), there are $(p-3)/2$ classes of elements
in $H$. There are $2$ classes of elements outside of $H$, which yields
$(p+5)/2$ classes altogether. A similar argument proves (4). Finally, we prove
(3). Let $G$ denote the group in question, and set  
$B:=\langle \left(\begin{smallmatrix} \omega^3 & 0 \\ 0 & 1 
\end{smallmatrix} \right), 
\left(\begin{smallmatrix} \omega & 0 \\ 0 & \omega 
\end{smallmatrix} \right) \rangle$. 
We see that $Z(G)$ has order $p-1$ and so $B \backslash Z(G)$ is the union of
$(p-1)(p-4)/6$ classes. There are $p-1$ classes of elements in $G$ outside $B$
and this yields $(p-1)(p+8)/6$ classes altogether, as needed. 
\end{proof}

We now describe the computations of the quantities $\m(\F,0)$, $\w(\F,0)$, 
and $\k(\F,0)$ that were carried out in Magma \cite{Magma} and listed in
Tables~\ref{t:rv2} and \ref{t:rv3}.  The list of nonconstrained saturated
fusion systems on $S$ is given in Table~\ref{t:rv2}, based on the list in
\cite[Tables~1.1, 1.2]{RuizViruel2004}. Generators for $\Out_{\F}(S)$ are
listed in the third column of Table~\ref{t:rv2} for the convenience of 
the reader: in each case there is exactly one $\Out(S)$-conjugacy class 
of subgroups isomorphic with $\Out_\F(S)$, and a representative is 
chosen to contain as many diagonal matrices as possible. Then 
$\Out_\F(S)$-orbit representatives and stabilizers for the actions on 
linear characters of $S$ and $S$-conjugacy classes were computed and 
listed in columns four through seven, using the notation of
Lemma~\ref{l:exsp-setup}. In each case, $\m(\F,0,3)$ is computed using 
Lemma~\ref{l:det=1}(3), by summing up the number of projective simple
modules of the stabilizers 
listed in the fifth column of Table \ref{t:rv2}.  Then $\m(\F,0,3)$
is listed in Table~\ref{t:rv3}.  The quantity $\m(\F,0,2)$ is computed
using Lemma~\ref{l:det=1}(2) and Lemma~\ref{l:someconjclasses}, which computes
the number of conjugacy classes of the various $\Out_\F^*(S)$ listed in
Table~\ref{t:rv3}. This completes the description of the computation of
$\m(\F,0)$ as the sum of $\m(\F,0,2)$ and $\m(\F,0,3)$.  Then $\w(\F,0)$ is
calculated using the list of outer automorphism groups of centric radicals in
\cite[Tables 1.1 and 1.2]{RuizViruel2004}. For example, we have denoted 
the three exotic fusion systems at the prime $7$ by
$\RV_{1}$, $\RV_2$, and $\RV_2:2$ in the tables, where $\Out_{\RV_1}(S) \cong
C_6 \wr C_2$, $\Out_{\RV_2}(S) \cong D_{16} \times C_3$, and $\Out_{\RV_2:2}(S)
\cong SD_{32} \times C_3$. Note that $\RV_{2}:2$ contains $\RV_2$ as a normal
subsystem of index $2$. These systems have the following invariants.
\begin{itemize}
\item $\m(\RV_1,0) = 41$ and $\w(\RV_1,0) = 35$, 
\item $\m(\RV_2,0) = 33$ and $\w(\RV_2,0) = 25$, and
\item $\m(\RV_2:2,0) = 42$ and $\w(\RV_2:2,0) = 35$. 
\end{itemize} 

Finally, $\k(\F,0)$ is calculated using Lemma \ref{l:k} by adding 
$\w(\F,0)$ and $|\Out_\F^*(S)^{\cl}|$ to the sum of the number of 
projective simple modules of the stabilizers listed in the seventh 
column of Table \ref{t:rv2}.

\begin{Prop}
Let $\F$ be a nonconstrained saturated fusion system on $S$. 
Then $\k(\F,0)=\m(\F,0)$, and Conjectures \ref{conj:k(b)}, 
\ref{c: malle-robinson},  \ref{c:defect}, \ref{c:heightzero},  
\ref{c:eatonmoreto} and \ref{c:malle-navarro} all hold for $\F$.
\end{Prop}

\begin{proof}
This can be easily verified using the tables.
\end{proof}

\begin{landscape}
\small

\begin{table}
\renewcommand{\arraystretch}{1.5}
\centering
\caption{$\Out_\F(S)$-orbits of $\Irr^3(S)$, $\F$-classes of $S^{\cl}$ 
and  their $\Out_\F(S)$-stabilisers }
\label{t:rv2}
\begin{tabular}{|c|c|c|c|c|c|c|}

\hline
$p$ & $\F$ & $\Out_\F(S)$ & $\Irr^3(S)/\Out_\F(S)$ & stabilisers 
& $(S^{\cl} \backslash \{1,\c\})/\F$ & stabilisers \\ 
\hline

$3 \nmid (p-1)$ & $\PSL_3(p)$ & $\langle \left(\begin{smallmatrix} \omega 
& 0 \\ 
0 & 1 \end{smallmatrix} \right), \left(\begin{smallmatrix} 1 & 0 \\ 
0 & \omega \end{smallmatrix} \right) \rangle$ 
& $\chi_{0,0},\chi_{0,1},\chi_{1,0},\chi_{1,1}$ 
& $C_{p-1}^2,C_{p-1},C_{p-1},1$ &$[\a\b]$&$1$\\ 
\hline

$3 \nmid (p-1)$ & $\PSL_3(p):2$ 
& $\langle \left(\begin{smallmatrix} \omega & 0 \\ 
0 & 1 \end{smallmatrix} \right), 
\left(\begin{smallmatrix} 1 & 0 \\ 0 & \omega \end{smallmatrix} \right), 
\left(\begin{smallmatrix} 0 & 1 \\ 1 & 0 \end{smallmatrix} \right) \rangle$ 
& $\chi_{0,0},\chi_{0,1},\chi_{1,1}$ & $C_{p-1} \wr C_2 ,C_{p-1},C_2$ 
&$[\a\b]$&$C_2$\\ 
\hline

$3 \mid (p-1)$ & $\PSL_3(p)$ 
& $\langle \left(\begin{smallmatrix} \omega^3 & 0 \\ 
0 & 1 \end{smallmatrix} \right), \left(\begin{smallmatrix} \omega & 0 \\ 
0 & \omega \end{smallmatrix} \right) \rangle$  
& $\makecell{\chi_{0,0},\chi_{0,1},\chi_{1,0},\\
\chi_{1,1},\chi_{\omega,1},  \chi_{1,\omega}}$ 
& $\makecell{C_{p-1} \times C_{(p-1)/3}, C_{(p-1)/3}, \\
C_{(p-1)/3},1,1,1}$  &$[\a\b]$, $[\a\b^\omega]$, $[\a^\omega \b]$ 
&$1$, $1$, $1$\\ 
\hline

$3 \mid (p-1)$ & $\PSL_3(p):2$ 
& $\langle \left(\begin{smallmatrix} \omega^3 & 0 \\ 
0 & 1 \end{smallmatrix} \right), 
\left(\begin{smallmatrix} \omega & 0 \\ 0 & \omega \end{smallmatrix} \right), 
\left(\begin{smallmatrix} 0 & 1 \\ 1 & 0 \end{smallmatrix} \right) \rangle$  
& $\makecell{\chi_{0,0},\chi_{0,1},\\ \chi_{1,1},\chi_{\omega,1}}$ 
& $\makecell{(C_{p-1} \times C_{(p-1)/3}):C_2, \\ C_{(p-1)/3},C_2,1}$  
&$[\a\b]$, $[\a\b^\omega ]$ & $C_2$, $1$ \\ 
\hline

$3 \mid (p-1)$ & $\PSL_3(p):3$ 
& $\langle \left(\begin{smallmatrix} \omega & 0 \\ 
0 & 1 \end{smallmatrix} \right), 
\left(\begin{smallmatrix} 1 & 0 \\ 
0 & \omega \end{smallmatrix} \right) \rangle$ 
& $\chi_{0,0},\chi_{0,1},\chi_{1,0},\chi_{1,1}$ 
& $C_{p-1}^2,C_{p-1},C_{p-1},1$ & $[\a\b]$ & $1$ \\ 
\hline

$3 \mid (p-1)$ & $\PSL_3(p):S_3$ 
& $\langle \left(\begin{smallmatrix} \omega & 0 \\ 
0 & 1 \end{smallmatrix} \right), 
\left(\begin{smallmatrix} 1 & 0 \\ 0 & \omega \end{smallmatrix} \right), 
\left(\begin{smallmatrix} 0 & 1 \\ 1 & 0 \end{smallmatrix} \right) \rangle$ 
& $\chi_{0,0},\chi_{0,1},\chi_{1,1}$ & $C_{p-1} \wr C_2 ,C_{p-1},C_2$ 
& $[\a\b]$ & $C_2$ \\ 
\hline

$3$ & $^2F_4(2)'$ 
& $\langle \left(\begin{smallmatrix} -1 & 0 \\ 
0 & 1 \end{smallmatrix} \right), 
\left(\begin{smallmatrix} 1 & 0 \\ 0 & -1 \end{smallmatrix} \right), 
\left(\begin{smallmatrix} 0 & 1 \\ 1 & 0 \end{smallmatrix} \right) \rangle$ 
& $\chi_{0,0},\chi_{0,1},\chi_{1,1}$ &$D_8$, $C_2$, $C_2$ & $-$ & $-$ \\ 
\hline

$3$ & $\J_4$ 
& $\langle \left(\begin{smallmatrix} 2 & 0 \\ 
0 & 1 \end{smallmatrix} \right), 
\left(\begin{smallmatrix} 1 & 0 \\ 0 & 2 \end{smallmatrix} \right), 
\left(\begin{smallmatrix} 1 & 2 \\ 2 & 2 \end{smallmatrix} \right) \rangle$ 
& $\chi_{0,0},\chi_{0,1}$ &$SD_{16}$, $C_2$  &$-$ & $-$ \\ 
\hline

$5$ & $\Th$ & $\langle \left(\begin{smallmatrix} 2 & 0 \\ 
0 & 1 \end{smallmatrix} \right), 
\left(\begin{smallmatrix} 1 & 0 \\ 0 & 2 \end{smallmatrix} \right), 
\left(\begin{smallmatrix} 3 & 3 \\ -1 & 1 \end{smallmatrix} \right) \rangle$ 
&$\chi_{0,0},\chi_{0,1}$ &$4.S_4,C_4$ &$-$ & $-$ \\ 
\hline

$7$ & $\He$ & $\langle \left(\begin{smallmatrix} 2 & 0 \\ 
0 & 1 \end{smallmatrix} \right), 
\left(\begin{smallmatrix} 1 & 0 \\ 0 & 2 \end{smallmatrix} \right), 
\left(\begin{smallmatrix} 0 & -1 \\ -1 & 0 \end{smallmatrix} \right) \rangle$ 
& $\makecell{\chi_{0,0},\chi_{0,1},\chi_{1,0},\\
\chi_{1,1},\chi_{3,1},  \chi_{1,3}}$ 
&$\makecell{S_3 \times C_3, C_3 \\ 
C_3,1,C_2,C_2}$ &$[\a]$, $[\b]$, $[\a\b^3]$, $[\a^3 \b]$ 
& $C_3$, $C_3$, $C_2$, $C_2$ \\ 
\hline

$7$ & $\He:2$ & $\langle \left(\begin{smallmatrix} 2 & 0 \\ 
0 & 1 \end{smallmatrix} \right), 
\left(\begin{smallmatrix} 3 & 0 \\ 0 & 3 \end{smallmatrix} \right), 
\left(\begin{smallmatrix} 0 & 1 \\ 1 & 0 \end{smallmatrix} \right) \rangle$ 
&  $\chi_{0,0}, \chi_{0,1}, \chi_{1,1},\chi_{3,1}$ 
&$S_3 \times C_6, C_3, C_2, C_2 $ & $[\a]$, $[\a\b^3]$ & $C_3$, $C_2$ \\ 
\hline

$7$ & $\Fi'_{24}$ & $\langle \left(\begin{smallmatrix} 2 & 0 \\ 
0 & 1 \end{smallmatrix} \right), 
\left(\begin{smallmatrix} 3 & 0 \\ 0 & 3 \end{smallmatrix} \right), 
\left(\begin{smallmatrix} 0 & -1 \\ -1 & 0 \end{smallmatrix} \right) \rangle$ 
& $\chi_{0,0},\chi_{0,1},\chi_{1,1},  \chi_{3,1}$ 
&$S_3 \times C_6$, $C_3$, $C_2$, $C_2$ &$[\b]$ & $C_3$ \\ 
\hline

$7$ & $\Fi_{24}$ & $\langle \left(\begin{smallmatrix} 3 & 0 \\ 
0 & 1 \end{smallmatrix} \right), 
\left(\begin{smallmatrix} 1 & 0 \\ 0 & 3 \end{smallmatrix} \right), 
\left(\begin{smallmatrix} 0 & 1 \\ 1 & 0 \end{smallmatrix} \right) \rangle$    
& $\chi_{0,0},\chi_{0,1},\chi_{1,1}$  
& $C_6 \wr C_2$, $C_6$, $C_2$  &$[\b]$ & $C_6$\\ 
\hline

$7$ & $\RV_1$ & $\langle \left(\begin{smallmatrix} 3 & 0 \\ 
0 & 1 \end{smallmatrix} \right), 
\left(\begin{smallmatrix} 1 & 0 \\ 0 & 3 \end{smallmatrix} \right), 
\left(\begin{smallmatrix} 0 & 1 \\ 1 & 0 \end{smallmatrix} \right) \rangle$    
& $\chi_{0,0},\chi_{0,1},\chi_{1,1}$ 
& $C_6 \wr C_2$, $C_6$, $C_2$  & $-$ & $-$ \\ 
\hline

$7$ & $\ON$ & $\langle \left(\begin{smallmatrix} 3 & 0 \\ 
0 & 3 \end{smallmatrix} \right), 
\left(\begin{smallmatrix} 1 & 0 \\ 0 & -1 \end{smallmatrix} \right), 
\left(\begin{smallmatrix} 0 & 2 \\ 3 & 0 \end{smallmatrix} \right) \rangle$ 
& $\chi_{0,0},\chi_{0,1},\chi_{1,1},  \chi_{1,3}$ 
& $D_8 \times C_3, C_2,1,C_2$ & $[\a\b]$ & $1$ \\ 
\hline

$7$ & $\ON:2$ & $\langle \left(\begin{smallmatrix} 3 & 0 \\ 
0 & 3 \end{smallmatrix} \right), 
\left(\begin{smallmatrix} 1 & 0 \\ 0 & -1 \end{smallmatrix} \right), 
\left(\begin{smallmatrix} 2 & 4 \\ -1 & 2 \end{smallmatrix} \right) \rangle$
& $\chi_{0,0},\chi_{0,1},\chi_{1,1}$ 
& $D_{16} \times C_3$, $C_2$, $C_2$ & $[\a\b]$ & $C_2$ \\ 
\hline

$7$ & $\RV_2$ & $\langle \left(\begin{smallmatrix} 3 & 0 \\ 
0 & 3 \end{smallmatrix} \right), 
\left(\begin{smallmatrix} 1 & 0 \\ 0 & -1 \end{smallmatrix} \right), 
\left(\begin{smallmatrix} 2 & 4 \\ -1 & 2 \end{smallmatrix} \right) \rangle$ 
& $\chi_{0,0},\chi_{0,1},\chi_{1,1}$ 
& $D_{16} \times C_3$, $C_2$, $C_2$ & $-$ & $-$ \\ 
\hline

$7$ & $\RV_2:2$ & $\langle \left(\begin{smallmatrix} 3 & 0 \\ 
0 & 3 \end{smallmatrix} \right), 
\left(\begin{smallmatrix} 1 & 0 \\ 0 & -1 \end{smallmatrix} \right), 
\left(\begin{smallmatrix} 2 & 1 \\ 5 & 2 \end{smallmatrix} \right) \rangle$
& $\chi_{0,0},\chi_{0,1}$ & $SD_{32} \times C_3$, $C_2$ & $-$ & $-$\\ 
\hline

$13$ & $\MM$ & $\langle \left(\begin{smallmatrix} 1 & 0 \\ 
0 & 8 \end{smallmatrix} \right), 
\left(\begin{smallmatrix} 2 & 0 \\ 0 & 2 \end{smallmatrix} \right), 
\left(\begin{smallmatrix} 10 & 9 \\ 5 & 2 \end{smallmatrix} \right) \rangle$ 
& $\chi_{0,0},\chi_{0,1},\chi_{1,1}$ & $C_3 \times 4.S_4$, $C_4$, $C_3$  
& $[\a\b]$ & $C_3$ \\ 
\hline
\end{tabular}
\end{table}
\end{landscape}

\begin{table}
\renewcommand{\arraystretch}{1.4}
\centering
\caption{$\w(\F,0)$ and $\m(\F,0,d)$ for $d=2,3$ }
\label{t:rv3}
\begin{tabular}{|c|c|c|c|c|c|}

\hline
 $p$ & $\F$ & $\Out^*_\F(S)$  & $\m(\F,0,2)$ & $\m(\F,0,3)$ & $\w(\F,0)$ \\ 
\hline

$3 \nmid (p-1)$ & $\PSL_3(p)$ & $C_{p-1}$  & $p-1$ & $p^2$ & $p^2-1$ \\ 
\hline

$3 \nmid (p-1)$ & $\PSL_3(p):2$ & $C_{p-1}.C_2$ & $(p+5)/2$ & $p(p+3)/2$ 
& $(p-1)(p+4)/2$ \\ 
\hline

$3 \mid (p-1)$ & $\PSL_3(p)$ & $C_{(p-1)/3}$ & $(p-1)/3$ & $(p^2+8)/3$ 
& $(p^2-1)/3$ \\ 
\hline

$3 \mid (p-1)$ & $\PSL_3(p):2$ & $C_{(p-1)/3}.C_2$ & $(p+17)/6$ 
& $(p+1)(p+8)/6$ & $(p-1)(p+10)/6$ \\ 
\hline

$3 \mid (p-1)$ & $\PSL_3(p):3$ & $C_{p-1}$  & $p-1$ & $p^2$ & $p^2-1$ \\ 
\hline

$3 \mid (p-1)$ & $\PSL_3(p):S_3$ & $C_{p-1}.C_2$  & $(p+5)/2$ 
& $p(p+3)/2$ & $(p-1)(p+4)/2$ \\ 
\hline

$3$ & $^2F_4(2)'$ & $C_4$ & $4$ & $9$ & $9$ \\ \hline
$3$ & $J_4$ & $Q_8$ & $5$ & $9$ & $9$ \\ \hline
$5$ & $\Th$ & $\SL_2(3)$ & $7$ & $20$ & $20$ \\ \hline
$7$ & $\He$ & $C_3$ & $3$ & $20$ & $10$ \\ \hline
$7$ & $\He:2$ & $C_6$ & $6$ & $25$ & $20$ \\ \hline

$7$ & $\Fi'_{24}$  & $C_6$ & $6$ & $25$ & $22$ \\ \hline
$7$ & $\Fi_{24}$ & $D_{12}$ & $6$ & $35$ & $29$ \\ \hline
$7$ & $\RV_1$ & $D_{12}$ & $6$ & $35$ & $35$ \\ \hline
$7$ & $\ON$ & $C_4$ & $4$ & $20$ & $19$ \\ \hline
$7$ & $\ON:2$ &  $C_8$ & $8$ & $25$ & $23$ \\ \hline
$7$ & $\RV_2$ &  $C_8$ & $8$ & $25$ & $25$ \\ \hline
$7$ & $\RV_2:2$ &  $Q_{16}$ & $7$ & $35$ & $35$ \\ \hline
$13$ & $\MM$ & $\SL_2(3)$ & $7$ & $55$ & $52$ \\ \hline
\end{tabular}
\end{table}

\newpage

\appendix
\section{On Lemma \ref{l:QvsQ} }
By Proposition \ref{p:centext}, Lemma \ref{l:QvsQ} is equivalent to the
following.

\begin{Lem} \label{l:QsvQ*central}
Let $G$ be a finite group, $Q$ a normal $p$-subgroup of $G$, $Z$ a
central $p'$-subgroup of $G$ and $e$ a central idempotent of $kZ$.
Then
 \begin{equation}\label{e:robinsoncent} 
\sum_{[x] \in Q^{\cl}/G} \ell(kC_G([x])e)= 
\sum_{\mu \in \Irr(Q)/G} \ell(k C_G(\mu) e).
\end{equation}
\end{Lem}

The rest of the section is devoted to a proof of Lemma
\ref{l:QsvQ*central}. The basic idea is that, when $e =1_{kZ}$, then
both sides count the number of $p$-sections in $G$ of elements of $Q$,
or the dimension of the space of  ordinary  class functions of $G$
vanishing outside  $p$-sections of elements of $Q$.

\bigskip\noindent
{\bf Notation.}
Let $(K, \O, k)$ be a $p$-modular system which we assume is big enough
for the finite groups considered in this section.  Denote by $\C(G)$
the $K$-vector space of all  $K$-valued class functions  on $G$ and
by $\Irr(G) \subset \C(G)$ the  set of ordinary  irreducible characters
of $G$ viewed as $K$-valued functions.

For $X \subset G$, denote by $ d^{X} : \C(G) \to \C(G) $,
the $K$-linear map defined by $\varphi \to  d^X(\varphi),  \varphi \in  
\C(G)$ where $d^X(\varphi)(g) = 0 $ if $ g_p   $ is not conjugate to
an element of $X$  and  $d^X(\varphi)(g) = \varphi(g)$ otherwise.
Thus, $d^X(\C(G))$ is the subspace of  all class functions which
vanish outside the $p$-sections  of elements of $X$, that is those
class functions $\varphi$ such that $\varphi(x) =0$ unless $x_p$ is
conjugate to an element of $X$.

If $ X=\{x\}$ we write $d^x$ for $d^X$. For general $X$ and
$x \in X$,  $d^x(\C(G))$  is a subspace of  $d^X(\C(G))$ and
$d^X(\C(G)) = \oplus_{x}\ d^x(\C(G))$, where $x$ runs
over a set of conjugacy class representatives of $p$-elements in  $X$.
Note that if $X$ is a normal $p$-subgroup of $G$, then $d^X{\C}(G) $
consists  of precisely those functions which take the value zero on
elements $g$ such that $g_p \notin Q$.

For a central idempotent  $f$  of $KG$ denote by  $\Irr(G, f)$ the
subset of  ordinary  irreducible characters  of $G$ corresponding to
simple $KGf$ modules and by $\C(G, f) $ the subspace of
$\C(G)$ consisting of those class functions which are in the
$K$-span of $\Irr(G, f)$.
Recall that the  canonical surjection $\O G \to kG$ induces a
bijection between the set of  central idempotents of  $ {\O}G $
and  of $kG$. By abuse of notation, if $e$ is a central
idempotent of $ kG$ corresponding to the central idempotent $\hat e$ of
$\O G$ we write $\Irr(G, e)$ for $\Irr(G, \hat e)$
and  $\C(G, e)$ for $\C(G, \hat e)$. Thus, if $e$ is a block
of $kG$, then $\Irr(G, e)$ is the subset of ordinary irreducible
characters of $G$ belonging to $\hat e$.
For $N$ a normal subgroup of $G$ and $\mu\in \Irr(N)$, let
$ \C (G, \mu)$ denote the subspace of  $\C(G)$ consisting of those
class functions which are in the $K$-span of irreducible characters
of $G$ which cover $\mu$ and for $f$ a central idempotent of $KG$
(or  $kG$) denote by $\C(G, \mu, f)$   the intersection of
$\C(G, \mu)$  and $\C(G, f) $.

The  following gives the desired interpretation   of the  left hand side
of  Lemma \ref{l:QsvQ*central}.  When  $e =1_{kZ}$, the statement is
elementary.  Passage to arbitrary $e$ requires an  application of
Brauer's  second main theorem which we now recall. Denote  by
$\IBr(G)$ the set of  Brauer characters of   simple $kG$-modules
viewed as  $K$-valued  class functions on  $ G_{p'} $, the set of
$p$-regular   elements of $G$.  For $x \in G$ a $p$-element,
$\chi \in \Irr(G)  $   and $ \varphi \in \IBr(C_G(x) )$ denote
by  $ d^x_{\chi, \varphi} $ the  corresponding generalised decomposition
number.  By Brauer's  second main theorem, if $b$ is the block of $kG$
containing  $\chi $, then   $ d^x_{\chi, \varphi} $  is zero unless
$\varphi $ is  the  Brauer character of a  simple  $kC_G(x)$ module
lying  in a  block  $c$ of $kC_G(x)$  which is in  Brauer correspondence
with $b$.   In other words,    for all $ y \in C_G(x)_{p'} $   we have that
$$ \chi (xy) = \sum_{\varphi}  d^x_{\chi, \varphi} \varphi(y)\ , $$
where $\varphi $ runs over the set of   irreducible Brauer characters
of $C_G(x)$  lying in Brauer correspondents of $b$.

\begin{Lem} \label{l:2main}
Let $x$ be a $p$-element  of $G$. Let $Z \leq  G$ be a central
$p'$-subgroup  of  $G$ and   $e$ a central idempotent of $kZ$.  Then,
\begin{equation}  
\dim_K   d^x  ( \C(G, e) )=  \ell(k C_G(x)e ).  
\end{equation}
If $ Q $ is a normal $p$-subgroup of $G$, then
\begin{equation}  \dim_K   d^ Q ( \C (G, e) )=   
\sum_{x \in Q^{cl}/G}   \ell(k C_G([x])e ).  
\end{equation}
\end{Lem}

\begin{proof}
The space $d^x(\C(G))$ consists of the  class functions on $G$ which
vanish outside the $p$-section of $x$, hence $\dim_K d^x(\C(G))$ equals
the  number of $p'$-conjugacy classes of $C_G(x)$ and this number is in
turn equal to the number of isomorphism classes  of simple
$kC_G(x)$-modules.
This proves that the first equation holds when $ e=1_{kZ} = 1_{kG}$.
For the general case,  first note that  since $Z$ is central in $G$,
$e$ is a central idempotent of  $kG$ and of $kC_G(x)$ and
$\Br_{\langle x\rangle}(e)  = e$, where
$\Br_{\langle x\rangle} : (kG)^{\langle x \rangle }  \to  
kC_G(x)$ denotes the Brauer homomorphism.
We claim that if $b$ is a block of $kG$ such that $be =b $  and $c$
is  a block of $kC_G(x)$ in Brauer correspondence with $ b$, then
$ce =c$. Indeed, by the uniqueness of  central idempotent
decompositions  and the  primitivity of $b$, we have $ be=b $.
By definition of Brauer correspondence,
$\Br_{\langle x\rangle}  (b)   c  =  c $.  Hence
$$c = \Br_{\langle x\rangle}  (b)   c =   
\Br_{\langle x\rangle}  (be) c =  
\Br_{\langle x\rangle}  (b )  e   c    =cec= ce\ , $$
proving the claim.  It follows from the claim that all simple
$kC_G(x)c$-modules are $kC_G(x)e$-modules.
Thus by Brauer's second main theorem (and the linearity of $d^x $),
if $\tau \in(\C(G,e))$, then for all $ y\in C_G(x)_{p'} $
we have
$$ \tau (xy)  = \sum_{\varphi}  d^x_{\chi, \varphi} \varphi,   $$
where $\varphi $ runs over the set of   Brauer characters of  simple
$kC_G(x)e$=modules.    Since $d^x \tau  $   is determined  by its
restriction to the  subset of  $C_G(x)$ consisting of elements whose
$p$-part is $x$, it follows that
$ \dim_K d^x(\C(G,e))  \leq   \ell(k C_G(x)e ) $.
By the same considerations,
$ \dim_K d^x(\C(G,1-e)) \leq \ell(k C_G(x) (1-e) )$.
Since $\C(G)  =   \C (G, e )  \oplus \C (G,   1-e) $,
$\dim_K  d^x(\C(G)) \leq \dim_K   d^x(\C(G,e))  +  
\dim_K   d^x(\C (G,1-e)) $. The  first
equation  now follows from  the case $e=1_{kZ}$.

Let  $\bar e $ be  the image of $ e$ under the canonical  surjection
of  $kG  \to  k (G/Q) $. Recall that  restriction along   $ kG \to kG/Q$
induces  a bijection between    the set of isomorphism classes of simple
$ kG/Q$-modules and  $ kG$-modules  sending   simple
$k (G/Q) \bar e $-modules  to  $ kG e $-modules.  Also, for any $x \in Q$,
$ e$ is a central idempotent of $ kC_G(x)$  and   identifying
$ C_G(x)/C_G(x) \cap Q   $ with  $ C_G(x) Q/Q $       via the
isomorphism  induced by inclusion of $C_G(x)  $ in $C_G(x) Q $,
the image of $ e$ in $ k (C_G(x)/C_G(x) \cap Q) $   is   $\bar e$. Hence
$$ \ell  (kC_G([x])  e) =   \ell  (kC_G(x) Q e ) =  
\ell (k (C_G  (x)Q/ Q ) \bar e)   =
\ell (k (C_G (x)/  C_G(x) \cap  Q) \bar e ) =   \ell (kC_G(x) e ). $$
Now the second equation follows from the first  since
$$ d^Q (\C(G, e) ) =\bigoplus _{[x]\in Q^{cl}/G}   
d^x(\C(G,e)) .  $$
\end{proof}

\begin{Lem}  \label{l:cliff1}
Let $Z$ be a central $p'$-subgroup of $G$  and $e$ a central idempotent
of $kZ$. Let $Q$ be a normal $p$-subgroup of  $G$.   Then
\begin{equation} d^Q  (\C(G,  e)) =    
\bigoplus _{\mu \in \Irr(Q)/G} d^Q( \C( G, \mu, e)) . 
\end{equation}
\end{Lem}

\begin{proof}
Since
$$\C(G)  = 
\bigoplus _{\mu \in \Irr(Q)/G} \C(G, \mu) , $$
we  have
$$d^Q  (\C(G)) = 
\sum_{\mu \in \Irr(Q)/G} d^Q(\C(G, \mu)). $$
We show that the sum on the right of  the second equation   is direct.
First note that  if $\varphi $   is an element of $ \C(G, Q)$,
then $ d^Q(\varphi) =0 $ if and only if the restriction of $ \varphi $
to  all subgroups $H$ containing $Q$  as a Sylow $p$-subgroup    equals
zero.
Now  suppose that $\varphi_{\mu}  \in   \C(G, \mu)$,
$\mu  \in  \Irr(Q)/G  $  are such that
$\sum_{\mu  \in \Irr(Q)/G} d^Q(\varphi_{\mu} ) =0 $  and let
$  H$ be a  subgroup of $G$ containing $Q$ as  a Sylow $p$-subgroup. Then
the restriction  of $\sum_{\mu  \in \mathrm {Irr}(Q)/G} \varphi_{\mu} =0$.
But  it is easy to see that   the restriction  of $\varphi_{\mu} $    to
$H$ is  in the   $K$-span of   irreducible characters of $H$  which cover
$G$-conjugates of $\mu $. In particular the restriction of $\varphi_{\mu}$
and  $\varphi_{\mu'}  $  for  $\mu'\ne \mu $ are  orthogonal class
functions on $H$.  Hence   the restriction of  $\varphi_{\mu}$ to $H$
equals  zero   for all $H$ and all $\mu $. It follows that
$d^Q( \varphi_{\mu})  =0 $ for all $\mu $. Thus
\begin{equation} 
d^Q  (\C (G)) = 
\bigoplus _{\mu \in \Irr(Q)/G} d^Q(\C( G, \mu)) . 
\end{equation}
The assertion  of the lemma now follows as $\C(G, e)  $
is the direct sum
$\bigoplus _{\mu \in \Irr(Q)/G} \C (G,  \mu, e)  $.
 \end{proof}

Given the above Lemma,  it remains to analyse
$d^Q(\C( G, \mu, e)) $ for each irreducible character $ \mu $
of $ Q$. This is done  via standard Clifford theoretic reductions.

\begin{Lem}\label{l:cliff2}
Let $Z$ be a central $p'$-subgroup of $G$  and $e$ a central idempotent
of $kZ$. Let $Q$ be a normal $p$-subgroup of  $G$ and let
$\mu \in \Irr(Q)$. Then
$\dim_K d^Q ( \C (G,\mu, e))  =  
\dim_K   d^Q  ( \C (C_G(\mu),\mu, e))  $.
\end{Lem}

\begin{proof}
Induction from $C_G(\mu)$ to $G$ induces a bijection between
$\Irr (\C_G(\mu))$ and  $\Irr(G)$. Since
$ Z \leq C_G(\mu)$, if $\chi \in \Irr(G, \mu, e)$,
then  $\Ind_{C_G(\mu)}^G(\chi) \in \Irr (G, \mu, e)$.
Hence induction induces an isometric isomorphism between
$\C(C_G(\mu), \mu, e)$ and
$\C(G, \mu, e )$. Further, it is easy to check from the
induction formula that $d^Q(\Ind_{C_{G}(\mu)}^G(\tau)) =    
\Ind_{C_{G}(\mu)}^G(d^Q (\tau))$ for all $\tau$ in
$\C(C_G(\mu))$. The result follows.
\end{proof}

\begin{Lem} \label{l:cliff3}
Let $Q$ be a normal $p$-subgroup of  $G$ and let $\mu$  be a $G$-stable
irreducible  character of $Q$.  There exist a  central extension
$$ 1 \to Y \to   \widetilde G \stackrel{\pi} {\to}  G \to 1, $$
an irreducible character $ \widetilde \mu $ of $\widetilde G$ and    a
one dimensional  character $\eta $ of $Y$ such that the following holds.
\begin{enumerate}

\item $Y$ is a  finite   $p$-group,  the  inverse image  of $Q$ in
$\widetilde G$  is a direct product of  $Y$ with  a normal subgroup
$ Q' $ of  $\widetilde G$  such that $\pi $ maps  $Q'$  isomorphically onto
$Q$.

\item  Identifying $ Q'$ with $Q$ through $\pi$,  there exists a
bijection
$$\Irr(G, \mu ) \to \Irr(\widetilde G, \eta^{-1} 1_{Q}), 
\  \  \  \chi \to \chi_0     $$
such that for any  $ g  \in G$ and   $\widetilde g \in \widetilde G$
lifting $g$
$$ \chi(g) = \widetilde \mu (\widetilde  g) \chi_0  (\widetilde g) .$$

\item  Suppose  that  $Z$  is  a  central $p'$-subgroup of $G$, and
$e$   is  a central idempotent of $kZ$.  Let $ \widetilde Z$ be the inverse
image of $Z$ in $ \widetilde G$. Then $\widetilde Z =   Y \times  Z' $, where
$Z'$  is a central $p'$-subgroup of $ \widetilde G$   mapping
isomorphically onto $Z$ by $\pi$. Identifying  $Z'$ with $Z$ the
bijection  $\chi \to \chi_0$ restricts to a bijection between
$\Irr(G, \mu, e)$ and
$\Irr(\widetilde G, \eta^{-1} 1_{Q}, e)$.
\end{enumerate}
\end{Lem}

\begin{proof}
The  proof combines elements of standard Clifford theory. We  briefly
sketch the basic constructions.
Let $m$ be the dimension of  $\mu$ and  let  $e_{\mu}$ be the
central idempotent of  $  K Q$  corresponding to $\mu$ Then
$ S=K Qe_{\mu} $ is a matrix algebra of dimension $m^2$. Since
$\mu $ is  $G$-stable, the conjugation action of $ G$ on $KG$  induces
an action of  $G$ on $S$. The group $\widetilde G$ is   constructed
as a subgroup of $G \times S^{\times}  $.
Let  $\pi: G \times S^{\times} \to G$ and $\pi': G\times S^{\times} $
be the projections onto the first and second components  respectively
and identify  $K $  with the scalar  matrices in $S$.
Let $\hat G$ be the the subgroup of $G  \times  S^{\times}  $
consisting of all elements of the form $(x, s) $, $x \in G$  and
$s\in S^{\times}$ such that  $s_x  a  s_x^{-1} = xax^{-1}$ for all
$ a \in S$.  Since the  action  of each element of $G$   on $S$ is by
an inner automorphism   and  $K=Z(S)$,  the restriction of  $\pi$ to
$\hat G $ is a surjective homomorphism with kernel
$1\times K^{\times}$.

Choose  a transversal  $I$  for $ Q$ in $G$  containing $O_{p'}  (G)$.
In particular, $I$ contains every central $p'$-element of $G$. For
each $ x \in  I$, choose $s_x \in  S^{\times} $ such that
$(x, s_x )\in \hat G$  and such that   the determinant
$ \mathrm{det}(s_x) $ of $s_x  $ equals $1$.   This can be achieved
by replacing $K$   by a suitable extension containing   the $m$-th roots
of  $ \mathrm{det}(s_x) $, $x \in G$.
Further, if  $z\in I$ is a central $p'$-element  of $G$, we choose
$s_z $ to be the identity.  Extend the map   $ x \to s_x $ to
$s:  G \to  S^{\times} $ by setting $s_{g}  = u  s_x $   if $ g= ux $,
$ u\in Q$, $  x \in  I$.   For all  $g, h   \in G$,    we have
$s_gs_h s_{gh} ^{-1}  \in  K^{\times}$  is a  scalar matrix.
Note that  since   $ u^{|Q|} =  1 $    for all $ y \in Q$,   we have
that   $ \mathrm{det}(s_g)^{|Q|}=1 $ for all $ g\in  G$ and consequently
by taking determinants we  see that $ (s_gs_h s_{gh} ^{-1} )^ {m^2|Q|} =1$
for all $ g, h\in G$.

Let  $ \widetilde G$  be the subgroup of $   \hat  G$ generated by
$(s_g,  g )$, $ g\in G$.  The  restriction $ \pi : \widetilde G \to G$ of
$pi$ to $\widetilde g$  is  surjective.  Let $Y \leq  1\times K^{\times}$
be the kernel of $\pi$.  For $g, h\in \widetilde G$,
$$ (g, s_g) (h, s_h) =  (gh,   s_gs_h) =     
(1, s_gs_h s_{gh}^{-1}) ( gh, s_{gh}),    $$
$$ (g,  s_g)^{-1}= (1, s_g s_{g^{-1} } )    (g^{-1}, s_{g^{-1}})=  
(1,  s_{g} s_{g^{-1}}s_{gg^{-1}})  (g^{-1}, s_{g^{-1}}).$$
It follows that
$Y =\langle  (1,  s_gs_h s_{gh} ^{-1} ), g, h \in G\rangle $.
As noted above,  $Y$   has exponent     dividing  $m^2|Q| $.  Since  $ Y$
is  isomorphic to a subgroup of   the multiplicative group of a field,
$Y$  is cyclic  of order  dividing  $m^2|Q| $. In particular, $Y$ is a
finite $p$-group.
Let $ Q' =  \{ (u, s_u) : u \in Q\}$. Since  $s_us_v = uv$
for all $ u, v \in  Q$,  $Q'$ is a  subgroup  of $\widetilde G$  with the
required properties.  This proves (1).

Let $ \eta: Y \to  K ^{\times}  $ be the irreducible character of $ Y$
which sends  $ (1, \lambda \cdot \mathrm{id}_S) $ to  $ \lambda $.
The map $ \pi' : \widetilde G  \to   S^{\times}$       defines  a
representation of $ \widetilde G$  whose  restriction to $ YQ $ equals
$ \eta  \mu $. Let $\widetilde \mu $  be the corresponding character.
Then  $\widetilde \mu $ is irreducible and    covers $\eta \mu $.
Let $\tau = \frac{1}{|Y| |Q| } 
\sum_{y \in Y, u \in Q}  \eta^{-1} (y)  (uy)^{-1} $ be the    central
idempotent of  $  K Y  Q $ corresponding to  $ \eta^{-1}  1_Q $.
There is  a $K$-algebra isomorphism
\begin{equation}\label{eq:cliffiso}    
\phi:    K G  e_{\mu} \to    S  \otimes_{ K}   K \widetilde G  \tau 
\end{equation}
satisfying
$$\phi ( g e_{\mu }) = s_g \otimes (g, s_g)  \tau, \ \  g \in G. $$
Let  $g \in  G$ and let $\widetilde  g \in  \widetilde G$ be a lift of $g $.
Then  $\widetilde g =  y (g, s_g)  $ for some $y\in Y$.  Since
$ys_g = \eta (y)  s_g  $ and $y (g, s_g) \tau =  \eta^{-1} (g, s_g) \tau $
it follows that   $s_g \otimes   (g, s_g)  \tau=  
\pi_2(\widetilde g) \otimes \widetilde g \tau $. Now (2)  follows since
$\Irr (\widetilde G , \eta^{-1} 1_ Q)$ coincides with the set of  irreducible
$K  \widetilde G \tau $ characters.

Let $Z$ be a central $p'$-subgroup of  $G$. By our choices above,
$s_z$ is the identity matrix for all $z \in Z$. Hence
$Z':=\{(z, 1) : z  \in Z  \}$ is a central  subgroup  of $\widetilde G$
and the inverse image $\widetilde Z$ of $Z$ in $\widetilde G$ is a
direct product  $\widetilde Z = Y \times  Z'$. Identifying $Z'$ with $Z$,
the image of the idempotent $e e_{\mu}$  under the isomorphism
\ref{eq:cliffiso}  is  $\mathrm{id_S} \otimes  e\tau $, proving (3).
\end{proof}

\begin{Lem}\label{l:cliff4}
Let $Z$ be a central $p'$-subgroup of $G$  and $e$ a central idempotent of
$kZ$. Let $Q$ be a normal $p$-subgroup of  $G$ and let $\mu$  be a
$G$-stable  irreducible  character of $Q$. Then
\begin{equation} 
\dim_K   d^Q  ( \C(G,\mu, e)) = \ell (kG e) . 
\end{equation}
\end{Lem}

\begin{proof}
Let $ \widetilde G$, $Y$, $\eta $ and $ \widetilde \mu $  be as in Lemma
\ref{l:cliff3}.   The bijection $\chi \to \chi_0 $ extends by linearity
to  a  $K$-linear isomorphism
$i: \C( G, \mu, e) \to \C(\widetilde G, \eta^{-1} 1 _{Q}, e)$ defined by
$$ \phi(g) = \widetilde \mu (\widetilde g ) i(\phi) (\widetilde g ),   \  \  
i^{-1} (\psi) (g)  =\widetilde \mu (\widetilde g)    \psi (\widetilde g )$$      
for all $\phi \in \C( G, \mu, e) $,
$\psi \in \C(\widetilde  G, \eta^{-1} 1 _{Q}, e)  $.
$g \in G$ and $\widetilde g \in \widetilde G$  lifting $\widetilde g$.
Now $g_p \in  Q$ if and only if $(\widetilde g)_p \in YQ $. It follows that
$$  i^{-1} \circ  d^{YQ}\circ  i  = d^{Q} ,$$
hence
$$  d^{YQ}   \circ i =i \circ  d^{Q} , $$
where by $d^{YQ} $ we mean  the relevant map on class functions on
$\widetilde G$. In particular, $\dim_K d^Q  (\C(G,\mu, e)) = 
\dim_K d^Q (\C(\widetilde G, \eta^{-1} 1 _{Q}, e)  $.

Let $\psi \in \C(\widetilde  G, \eta^{-1}  1 _{Q}, e) $. For any
$u \in Q$, $y \in Y$,  $\widetilde g \in \widetilde G$,  we have
$\psi(yu \widetilde g) = \eta (y) \psi(\widetilde  g)$ from which it 
follows that
$$ \dim_K d^{ Y  Q}\C ( \widetilde G, \eta^{-1}  1 _{Q}, e  ) = 
 \dim_K d^{1}\C(\widetilde  G,  e )   =  
\ell (k\widetilde G e ) =  \ell(kGe) $$
where the second equality holds by  Lemma ~\ref{l:2main} and the last
equality holds since every simple $ k\widetilde G e $-module has $Y$ in its
kernel.
\end{proof}

\begin{proof}[Proof of Lemma~\ref{l:QsvQ*central}]
This follows from Lemmas \ref{l:2main},~\ref{l:cliff1}, ~\ref{l:cliff2}
 and  ~\ref{l:cliff4}.
\end{proof}

\bibliographystyle{amsalpha}
\bibliography{mybib}
\end{document}